\documentclass[a4paper,12pt,reqno]{extarticle}
\usepackage[english]{babel}
\usepackage[left=3.2cm,right=3.2cm,top=3cm,bottom=3cm]{geometry} 
\usepackage{newtxtext}
\usepackage{amsmath,amssymb,mathrsfs,amsthm,stmaryrd}
\usepackage{mathtools} 
\usepackage{enumerate,enumitem}
\usepackage[hidelinks]{hyperref}
\usepackage{esint} 
\usepackage{upgreek} 
\usepackage[toc,page]{appendix}
\usepackage{scalerel} 

\theoremstyle{plain}
\newtheorem{theorem}{Theorem}[section]
\newtheorem{theoremdef}[theorem]{Theorem and definition}
\newtheorem{lemma}[theorem]{Lemma}

\newtheorem{assumption}[theorem]{Assumption}
\newtheorem{corollary}[theorem]{Corollary}
\theoremstyle{definition}
\newtheorem{example}[theorem]{Example}
\newtheorem{definition}[theorem]{Definition}
\newtheorem{remark}[theorem]{Remark}
\newtheorem{notation}[theorem]{Notation}
\newtheorem{claim}[theorem]{Claim}

\makeatletter

\@addtoreset{equation}{section}
\makeatother
\providecommand{\msc}[1]{{\small \textit{Mathematics Subject Classification ---} #1}}

\renewcommand{\tilde}{\widetilde}
\DeclareFontEncoding{FMS}{}{}
\DeclareFontSubstitution{FMS}{futm}{m}{n}
\DeclareFontEncoding{FMX}{}{}
\DeclareFontSubstitution{FMX}{futm}{m}{n}
\DeclareSymbolFont{fouriersymbols}{FMS}{futm}{m}{n}
\DeclareSymbolFont{fourierlargesymbols}{FMX}{futm}{m}{n}
\DeclareMathDelimiter{\vvert}{\mathord}{fouriersymbols}{152}{fourierlargesymbols}{147}

\newcommand\nn[1]{\vvert#1\vvert_{\gamma,\gamma';\alpha,\sigma}}
\newcommand\nnn[1]{\vvert#1\vvert_{\gamma,\gamma';\alpha,\gamma}}

\def\cB{\mathcal{B}}
\def\cC{\mathcal{C}}
\def\cD{\mathcal{D}}

\def\cF{\mathcal{F}}

\def\cK{\mathcal{K}}
\def\cL{\mathcal{L}}

\def\cO{\mathcal{O}}

\def\cR{\mathcal{R}}

\def\cX{\mathcal{X}}	
\def\cY{\mathcal{Y}}
\def\cZ{\mathcal{Z}}

\newcommand{\R}{\ensuremath{\mathbf R}} 
\newcommand{\C}{\mathbf{C}} 
\newcommand{\T}{\mathbf{T}} 

\newcommand{\N}{\mathbf{N}_0} 
 
\newcommand{\X}{\mathbf{X}}
\newcommand{\XX}{\mathbb{X}}
\newcommand{\xx}{(\mathbb{X}-\updelta X^{\otimes2})}
\newcommand{\xxx}{(\mathbb{\bar X}-\updelta \bar X^{\otimes2})}

\newcommand{\dd}{\updelta}

\renewcommand{\Re}{{\mathrm{Re}}}

\newcommand{\dom}{{\mathcal O}}

\newcommand{\id}{I}

\usepackage{xcolor}
\colorlet{darkred}{red!90!black}

\begin{document}
\title{Quasilinear rough evolution equations}

\author{Antoine Hocquet\thanks{Technische Universit\"at Berlin, Straße des 17. Juni 136
		10623 Berlin, Germany.~E-Mail: antoine.hocquet86@gmail.com}~~~and~~Alexandra Neam\c{t}u\thanks{University of Konstanz, Department of Mathematics and Statistics,  Universit\"atsstra\ss{}e~10 78464 Konstanz, Germany. E-Mail: alexandra.neamtu@uni-konstanz.de }}

\maketitle

\begin{abstract}
We investigate the abstract Cauchy problem for a quasilinear parabolic equation in a Banach space of the form
\( du_t  -L_t(u_t)u_t dt = N_t(u_t)dt + F(u_t)\cdot d\mathbf X_t \),
where \( \mathbf X\) is a \( \gamma\)-H\"older rough path for \( \gamma\in(1/3,1/2)\). We explore the mild formulation that combines functional analysis techniques and controlled rough paths theory which entail the local well-posedness of such equations. We apply our results to the stochastic Landau-Lifshitz-Gilbert and Shigesada-Kawasaki-Teramoto equation. In this framework we obtain a random dynamical system associated to the Landau-Lifshitz-Gilbert equation.
\end{abstract}

\msc{
60L50, 
60L20, 
60H15, 
37H05 
}


 \tableofcontents

\section{Introduction}
In this work we contribute to the solution theory of quasilinear SPDEs within the rough paths framework. More precisely, for a given initial data \( x \) in a Banach space \( \cY\)
we consider the following Cauchy problem for a quasilinear non-degenerate parabolic rough partial differential equation of the form 
\begin{equation} 
\label{problem}
d u_t - L_t(u_t)udt = N_t(u)dt + \sum_{i=1}^dF^i(u)d\X^i_t,
\quad \quad 
u_0=x\in \cY,
\end{equation}
where $\X$ is a $\gamma $-H\"older, $d$-dimensional rough path 
$\X=(X^{\mu },\XX^{\mu \nu })_{1\leq \mu ,\nu \leq d}$,
with $\gamma\in(\frac{1}{3},\frac{1}{2})$ and $d\geq 1.$ 
We obtain existence and uniqueness of solutions for~\eqref{problem} and also continuous dependence of the solution on the corresponding rough path, i.e.~the continuity of the map $\X\mapsto u$ with respect to a suitable topology.\\
The method we employ is a mixture of Amann's and Yagi's classical treatment of quasilinear problems \cite{amann1986quasilinear,Yagi1} and recent progress on rough evolution equations \cite{gerasimovics2020non,gerasimovics2019hormander,gubinelli2010rough,deya2012nonlinear,HN19}. More precisely we construct in a pathwise way, solutions satisfying the variation of constants formula
\begin{equation}
\label{variation_of_constants}
u_t = S^u_{t,0}u_0 + \int_0^t S^u_{t,s}N(u_s)ds + \int_0^t S^u_{t,s}F(u_s)\cdot d\X_s,
\end{equation}
where \( S^u_{t,s}=\exp(\int_s^tL(u_r)dr) \) are parabolic evolution operators. In order to make sense of the rough convolution 
\( \int_0^t S^u_{t,s}F(u_s)\cdot d\X_s \),
we explore the mixed parabolic/rough regularity properties that result from the smoothing effect of the evolution operators and a variation of the usual sewing lemma argument with a weight at the origin.
To this aim we make use of the controlled rough paths framework which has been intensively studied in the context of rough evolution equations~\cite{deya2012nonlinear,gubinelli2004controlling,gerasimovics2019hormander,gerasimovics2020non,HN19} and the references specified therein. 
We extend the framework developed in \cite{gerasimovics2020non}, where spaces of controlled rough paths 
\begin{equation}
\label{space_D}
\cD_{X,\alpha}^{2\gamma}=\cD_{X,\alpha}^{2\gamma}((\cB_{\beta})_{\beta\in \R})
\end{equation}
have been introduced, for a monotone family of Banach spaces \( (\cB_{\beta})_{\beta\in \R} \) subject to some interpolation properties (these include Sobolev or Bessel-potential spaces). 
In this setting, a pair $(y,y')\colon [0,T]\to \cB_{\alpha}\times (\cB_{\alpha-\gamma})^d$ of continuous paths is called a \textit{controlled rough path} in \( \cD^{2\gamma}_{X,\alpha} \), if it satisfies an inequality of the form
\begin{equation}
\label{remainder_intro}
|y_t-y_s - y'_s \cdot (X_t-X_s)|_{\alpha-2\gamma}\lesssim (t-s)^{2\gamma}
\end{equation}
uniformly for each \( 0\le s\le t\le T \), and provided \( y \) (resp.\ \( y' \)) is \( \gamma \)-H\"older as a path with values in \( \cB_{\alpha-\gamma} \) (resp.\ in \( \cB_{\alpha-2\gamma} \)).
Because the first component of a controlled path is only required to be \( \gamma \)-H\"older, the relation \eqref{remainder_intro} incorporates a cancellation between the increment \( \dd y_{s,t}:=y_t-y_s \) and the term \( y_s'\cdot (X_{t}-X_s)\). Hence the space of controlled rough paths is built in such a way that, at small scales, the first component \( y \) always resembles the reference path \( X \) (informally speaking), while the \textit{Gubinelli derivative} which is the second component \( y' \), gives a corresponding modulation.
In \eqref{space_D}-\eqref{remainder_intro}, we remark that the parameter \( \gamma \) is present in the ``time-like'' quantity \( (t-s)^{2\gamma} \) but also as a degree of spatial regularity on the left hand side. 
As already underlined in \cite{gerasimovics2020non}, this happens for at least 2 reasons. The first is that solutions of parabolic PDEs, when they exist, are generally not H\"older regular paths in the Banach space where the initial condition lies, but rather in a larger space. Hence the need to take a lower spatial regularity index in \eqref{remainder_intro}, and this is so regardless of the H\"older regularity of the control \( X \). The second reason is that in many applications, the rough non-linearity \( u\mapsto F(u) \) is unbounded. A plausible scenario is when \( F \) is a deregularizing operator, going from \( \cB_{\beta}\to \cB _{\beta-\sigma}\) for some \( \sigma\ge0 \) and every \( \beta\) in some compact interval \([\beta_0,\alpha]\).
Since we expect any reasonable controlled rough path solution \( (u,u') \) of \eqref{problem} to satisfy \( u'_t=F(u_t),\forall t\in[0,T] \), we see that in this case the spatial regularity of the Gubinelli derivative needs to be lowered down from \( \cB_{\alpha} \) to \( \cB_{\alpha-\sigma}\). A further analysis of \eqref{problem} reveals moreover that the deregularizing parameter (or ``spatial loss'') \( \sigma \) cannot exceed the H\"older regularity of \( X \), hence the choice made in \cite{gerasimovics2020non} to work in a setting when \( u'\colon [0,T]\to \cB_{\alpha-\gamma}\) by default (even if it means losing valuable information when \( F \) is bounded).

Due to the quasilinear nature of the ansatz \eqref{problem}, we have to restrict the range of $\beta\in[0,1]$ and consequently work with the scale $(\cB_\beta)_{\beta\in[0,1]}$. To make sure that all typical operations with rough paths (for~e.g.~composition with a smooth function) are well-defined within this scale of Banach spaces, this constraint on the possible range of indices forces us to keep track of the spatial loss \( \sigma \) and to introduce a more specific class of controlled rough paths 
\begin{equation}
	\label{new_CP}
	\cD^{\gamma,\gamma'}_{X,\alpha,\sigma}=\cD^{\gamma,\gamma'}_{X,\alpha,\sigma}((\cB_{\beta})_{\beta\in [0,1]})\,.
\end{equation}
The parameter $\gamma'>1-2\gamma$ is also needed to distinguish the H\"older regularity of the Gubinelli derivative from that of the solution, which turns out to be a key ingredient in our analysis.
This new framework is consistent with the one developed in~\cite{gerasimovics2020non} in that \( \cD_{X,\alpha,\gamma}^{\gamma,\gamma} \simeq \cD_{X,\alpha}^{2\gamma} \) as Banach spaces
and moreover \(\cD_{X,\alpha,\sigma}^{\gamma,\gamma'}(J) \hookrightarrow \cD_{X,\alpha}^{2\gamma}(J) \) holds provided \( \gamma'\ge \gamma\), \( \sigma\le\gamma\) and
\( J\subset(0,T] \) is open (see Lemma \ref{lem:interpolation} for a precise statement).
However, another typical challenge in the quasilinear case is to quantify the dependence on the initial data in the fixed point argument. Here one needs that the difference between two different generators applied to the same initial data be $\gamma$-H\"older regular in $\cB_{\alpha-\sigma}$. This is not true for initial data belonging to $\cB_\alpha$, a problem that was easily circumvented in the semilinear ansatz considered in \cite{gerasimovics2020non} by subtracting the linear part \( S_{t,0}x \) to the unknown. Such a reduction procedure is not possible for \eqref{problem} in general, we need instead to introduce a weight in zero similar to \cite{HN19}, which accounts from the possible blow-up behaviour near the origin for solutions of \eqref{problem}, with speed \( t^{-\varepsilon} \) where $\varepsilon:=(\gamma-\sigma)_+$.
This trick enables us to solve~\eqref{problem} by means of a fixed point argument in the space of controlled rough path $\cD^{\gamma,\gamma'}_{X,\alpha,\sigma}([0,T])$, where all the previous norms are slightly modified according to the corresponding weight. To make the argument work, we had to extend the affine sewing lemma developed in~\cite{gerasimovics2020non} to define the stochastic convolution and to quantify the gain of spatial regularity obtained by this operation, taking into account the weight at the origin.\\
In conclusion, the results obtained in this work provide a theory for a pathwise solvability of parabolic quasilinear SPDEs of the form
\begin{equation}
\label{abstract_SPDE}
du_t -L_t(u_t)udt= N_t(u_t)dt + F(u_t)\cdot dW_t ,\quad \quad u_0=x\in \cY
\end{equation}
(which can be interpreted in the It\^o or Stratonovich sense), where \( W \) is a \( d \)-dimensional Brownian motion and \( \cY \) is an interpolation space between a reference Banach space \( \cX \), and the domain \( \cX_1=D(L(x))\) of the \( \cX\)-realization of \( L(x)\). This includes for instance \( \cX=L^p(\R^n)\) for \( p\in[1,\infty) \), which has the UMD 2-smooth property if \(p\ge2\). In that case, the It\^o integration theory carries over, however the mild solution theory breaks down due to measurability issues.
This was observed in~\cite{PronkVeraar} in the context of semilinear non-autonomous random evolution equations when $(L(t,\omega))_{t\in[0,T],\omega\in\Omega}$ is a family of (linear) non-autonomous random operators generating parabolic evolution families $S(t,s,\omega)=\exp\int_s^t L(r,\omega)dr$.
In this setting, as justified in~\cite[Proposition 2.4]{PronkVeraar},
  the mapping $\omega\mapsto S_{t,s,\omega}$ is only $\mathcal{F}_t$-measurable. However, in order to define the stochastic  convolution $\int\nolimits_0^t S_{t,s,\omega}F(u_s)~dW_s $ as an It\^o integral, the $\mathcal{F}_s$-measurability of the mapping $\omega\mapsto S_{t,s,\omega}F(u_s)$ is required. The non-adaptedness of the integrand in the definition of a  stochastic integral was firstly discussed by Al\'os, Le\'on and Nualart in \cite{AlLeNu99, leon1998stochastic} using the Skorokhod and the Russo-Vallois~\cite{RuVa93} forward integral. Similar to~\cite{PronkVeraar}, in~\cite{leon1998stochastic} such problems arise for semilinear SPDEs with random, non-autonomous generators.
 Furthermore, in \cite{AlLeNu99, leon1998stochastic} it was shown that a Skorokhod-mild solution for such SPDEs does not satisfy the weak formulation, whereas the forward mild one (based on the Russo-Vallois integral) does. The extension of~\cite{leon1998stochastic} from the semilinear to the quasilinear context (as treated here) remains unclear.
 For semilinear SPDEs with random non-autonomous generators, the concept of {\em pathwise mild solution} was introduced in~\cite{PronkVeraar}. This can be justified by the integration by parts formula given by
 \[
 \int _0^t S_{t,r} F(u_r) dW_r =  S_{t,0}\int _0^t F(u_\tau) dW_\tau -\int _0^tS_{t,r}L_r\int_r^t F(u_\tau) dW_\tau~dr.
 \]
 Moreover, the pathwise mild solution is equivalent to the forward mild one, as established in~\cite{PronkVeraar}. This idea was the starting point for~\cite{kuehn2018pathwise}, which extends the pathwise mild solution theory to the quasilinear case. The equivalence between pathwise mild and weak solutions was justified in~\cite{DhariwalHuberN}.
Concerning nonlinear stochastic partial differential equations with Gaussian noise, we should also mention the recent approach of Agresti and Veraar \cite{agresti2020nonlinear1} (see also \cite{agresti2020stability}), which is based quite differently on \( L^p(L^q) \) type estimates.\\
 As already emphasized a first goal of this work is to go beyond the Gaussian setting and extend the solution theory of~\eqref{problem} to possibly rougher random inputs. Even though our results are deterministic in nature (note that It\^o calculus is only used in the stochastic examples), we will nonetheless connect rough paths and pathwise mild solutions in the Brownian case in a forthcoming paper.\\ 
We point out that controlled rough paths can be seen as particular instances of modelled or paracontrolled distributions such as the ones considered \cite{BDH19,FuGu19,GeHa19,OtWe19, OttoS}. A purpose of these works was to deal with an input \( X(t,x) \) which is irregular in the \( x \)-variable as well (such as space-time white noise).  Here we focus on situation when the regularity is low only in the time-like variable, however we obtain more in return. Indeed, aside from a relative gain of simplicity in our computations, a merit of our functional analytic framework is that it is quite general and integrates particularly well with boundary value problems (such as Dirichlet or Neumann homogeneous).\\
As an application of our main abstract result, we will also establish local well-posedness results for quasilinear parabolic initial boundary value problems of the form
\[
\left \{\begin{aligned}
&du - \mathscr L(t,x,u,Du)udt 
\\
&\quad \quad \quad 
=  g(t,x,u_t,Du_t)dt
 +\sum\nolimits _{i=1}^df^i(x,u_t)d\X^i_t,
\quad \text{ in }(0,T]\times\dom
\\
&\mathscr Bu=0\quad \text{ on }(0,T]\times \partial \dom,
\\
&u(0,\cdot)=u_0\quad \text{ on }\dom,
\end{aligned}\right .
\]
where $\dom$ is a domain in $\R^n$ with a smooth boundary and $\mathscr B$ is an operator whose kernel encodes either Dirichlet, Neumann, or periodic boundary conditions if \( \dom=(-1,1)^n. \)
Here $\mathscr L(t,x,u,Du)$ is a strongly parabolic differential operator of order $2$ and the nonlinear terms $g,f$ are Nemyitskii operators.\\
Another advantage of our approach is that it provides a natural framework to study random dynamical systems associated to~\eqref{problem}. In contrast to the deterministic setting~\cite{amann95quasilinear,Yagi1}, results in the literature concerning flows for stochastic quasilinear evolution equations (based on a semigroup approach) are difficult to find.
Combining rough path tools together with semigroup arguments, our approach allows to deduce immediately that the solution operator of~\eqref{problem} generates a random dynamical system, provided however that we can ensure its global well-posedness. We provide more details on this topic in Section~\ref{rds}, where we further obtain a random dynamical system for the stochastic Landau-Lifshitz-Gilbert (LLG) equation. Based on these results one could further investigate the existence of equilibria and their transitions~\cite{br2} combining random dynamical systems and rough paths methods, which may be done in a future contribution.
To our best knowledge, this is the first work that constructs mild solutions for the stochastic LLG equation and establishes a random dynamical system for it.

\paragraph{Structure of the paper.}
 In Section \ref{s:preliminaries} we collect important properties on parabolic evolution operators and controlled rough paths according to a monotone family of interpolation spaces. We also introduce an appropriate solution concept for~\eqref{problem} and state the main result (Theorem~\ref{thm:main}). We also state and prove a main auxiliary result regarding approximation of integrals via dyadic sums (Lemma \ref{lem:dyadic}) which is one of the core arguments used throughout the rest of the paper.
In Section \ref{sec:rough_convolutions}, we explain our functional analytic framework and formulate our main structural assumptions on scales of Banach spaces and operators. We then state results on the existence and uniqueness of certain rough convolutions with a blow-up behaviour near \( t=0\), and then collect the key perturbation estimates which are required for the proof of our main result. Section \ref{sec:evolution} contains the proof of Theorem \ref{thm:main} by means of a fixed point argument. 
We show in Section \ref{sec:quasilinear_systems} how to apply our main result to a class of rough quasilinear parabolic systems with a non-degenerate elliptic part. 
The results are further specified in Section \ref{sec:stochastic_examples}, where a stochastic framework is considered. Section~\ref{rds} deals with random dynamical systems for~\eqref{problem}. 
In Appendix \ref{app:perturbation}, we provide important estimates for the difference of two propagators. These are crucial for a perturbation result of the sewing map (Theorem~\ref{thm:perturbation}). Finally, in Appendix~\ref{app:plain} we establish some results on the H\"older norms of controlled rough paths in terms of their reduced increments. 
 These essentially simplify the computations in Section~\ref{sec:evolution}.

\subsubsection*{Acknowledgements}
AH was funded by Deutsche Forschungsgemeinschaft (DFG) through grant CRC 910 ``Control of self-organizing nonlinear systems: Theoretical methods and concepts of application'', Project (A10) ``Control of stochastic mean-field equations with applications to brain networks.''
AH wishes to thank A.~Gerasimovi\v{c}s for a helpful discussion on weights.

\section{Preliminaries}\label{s:preliminaries}
\subsection{Frequently used notation}
We let $\mathbf N=1,2,\dots,$ $\N:=\mathbf N\cup \{0\}$. Real numbers are denoted by $\R$  and complex numbers by \( \mathbf C\). For any \( x,y\in \R ,\) we denote by \( x\vee y=\max (x,y) \), \( x\wedge y=\min(x,y) \) and we also introduce the notation \( x_+=x\vee0. \)
For \( \vartheta \ge 0 \), we denote by \( \Sigma_\vartheta \subset \mathbf C\) the sector
\[
\Sigma_\vartheta= \{z\in \mathbf C:|\arg z|\le \vartheta\}\cup \{0\}\,.
\]
If \( \cX ,\cY\) are Banach spaces, the notation $\cY\hookrightarrow \cX$ means that \( \cY\) is a dense subset of \(\cX \) and that the inclusion map is continuous. The Banach space of bounded linear operators from \( \cX\to \cY \) is denoted by \( \cL(\cX,\cY) \) or \( \cL(\cX)\) when \( \cX=\cY\). The induced operator norm is denoted by \( |Q|_{\cX\to\cY}=\sup_{|x|_{\cX}\le 1}|Qx|_{\cY}\).

We denote by \(C(\cX,\cY)\) (resp.\ \( C_b(\cX,\cY)\)) the space of all continuous (resp.\ continuous and bounded) functions from $\cX$ to $\cY$, endowed with the topology of uniform convergence on compact subsets. 
Moroever if $\rho\in (0,1)$ we define the usual space of H\"older functions $C^\rho(\cX,\cY)$. It is equipped with the locally convex topology defined by the semi-norms
\begin{equation}
	\label{holder_semi}
\sup_{x\in U}|f|_{V} + [f]_{\rho,\cY}^U,\quad U\subset \cX,\quad U\text{ compact.}
\end{equation}
where $[f]^U_{\rho,\cY}$ is the usual H\"older semi-norm on $U$.
The previous definition is still meaningful for $\rho=1$ and we write instead $C^{1-}$, which corresponds to the usual space of locally-Lipshitz functions. 
Moreover, given another Banach space $\cX'$ and $\rho'\geq 0$, $C^{\rho,\rho'}(\cX\times \cX';\cY)$ denotes the set of functions which are jointly H\"older continuous on \( \cX, \cY\), with respective exponents \( \rho,\rho'\).

If \( F\colon \cX\to \cY \) is a Fr\'echet-differentiable map, we denote by \( DF(x)\circ h \) its differential evaluated at \( x\), in the direction \( h\). If \( \cZ \hookrightarrow \cX\) is a third Banach space such that
\( DF(\cdot)|_{\cZ} \) is itself Fr\'echet-differentiable, we think of its derivative \( D^2F(x) \) as a bounded linear map from \( \cZ\otimes \cX\to \cY \) where, unless otherwise stated, the tensor product \( \cZ\otimes \cX \) is equipped with the projective norm.\\
If a scale of Banach spaces \( (\cB_{\beta})_{\beta \in J}\) (see Subsection~\ref{scale}) is given and \( \alpha,\alpha',\beta \in J\) are fixed indices, we adapt the previous notations and write
\( |Q|_{\alpha\to\beta}:=|Q|_{\cB_{\alpha}\to\cB_{\beta}} \), \( |Q|_{(\alpha,\alpha')\to\beta}:=|Q|_{\cB_{\alpha}\otimes\cB_{\alpha'}\to\cB_{\beta}}  \) and similarly for higher order tensor products.

\paragraph{Paths.}
Throughout the paper we consider a finite, fixed time horizon $T>0.$ We introduce the simplices
\( \Delta_2:=\{(t,s)\in[0,T]^2:t\ge s\}\) and
\( \Delta_3:=\{(t,u,s)\in[0,T]^3:t\ge u\ge s\}.\)
If \( \cX \) is a Banach space, \( y\colon [0,T]\to \cX \) is a path and \( S\colon \Delta_2\to \cL(\cX) \), we say that \( S \) is \textit{multiplicative} if \( S_{t,u}S_{u,s}=S_{t,s} ,\) \( \forall (s,u,t)\in \Delta_3 \). 
Given such \( S \), we introduce the $S$-increment of \(y\colon [0,T]\to\cX\) as
\( \dd^S y_{t,s}:= y_t-S_{t,s}y_s\) for \((t,s)\in \Delta_2\),
and similarly \(\dd^S z_{t,u,s}:= z_{t,s}-z_{t,u}-S_{t,u}z_{u,s},\) 
if  \((t,u,s)\in \Delta_3\) and \(z\colon\Delta_2\to \cX.\)
Recall that $\dd^S z_{t,u,s}\equiv0$ if and only if there is a \(y\colon [0,T]\to \cX\) such that $z_{t,s}=\dd^{S}y_{t,s},$ $\forall(t,s)\in \Delta_2$.
Throughout the paper, we will use the abbreviation $\dd=\dd^{\id}$ which corresponds to increments in the usual sense. It is worth noticing the tautological but nonetheless useful identity
\begin{equation}\label{id_delta_S}
	\dd z _{t,s} = \dd^{S} z_{t,s} + (S_{t,s}-I)z_s,\quad \quad (s,t)\in \Delta_2.
\end{equation}
By convention, we let \(\cC_1^\gamma (0,T;\cX )=C^\gamma(0,T;\cX)\) denote the usual Banach space of \( \gamma\)-H\"older functions, with norm \( |y|_{\gamma,\cX}=\sup_{t\in [0,T]}|y_t|_{\cX}+[\dd y]_{\gamma,\cX}.\)
We denote by \(\cC^\gamma_2(0,T;\cX) \) the space of H\"older continuous (in a generalized sense) functions from \( \Delta_2\to\cX\). It is defined as the Banach space of 2-parameter maps  \( z=z_{t,s} \) that vanish on the diagonal \( s=t \), and such that the norm \( [z]_{\gamma,\cX}:= \sup_{0\le s<t\le T}|z_{t,s}|_{\cX}/(t-s)^{\gamma}\) is finite. \\
Furthermore, given a scale \( (\cB_{\beta})_{\beta \in J}\), we use the abbreviations \(|\cdot|_{\gamma,\beta }:=|\cdot |_{\gamma,\cB_\beta}\), \([\cdot ]_{\gamma ,\beta }:=[\cdot ]_{\gamma ,\cB_\beta }\), and similarly for operator-valued paths: \( |\cdot |_{\gamma,\alpha\to\beta}:=|\cdot|_{\gamma,\cB_\alpha\to\cB_{\beta}}\), \( |\cdot|_{\gamma,(\alpha,\alpha')\to\beta} \), etc.

\subsection{Weights at the origin}\label{weight}
For our aims (in order to set up a fixed point argument in Section~\ref{sec:evolution}), it is necessary to introduce suitable weights for the supremum and H\"older norm at the origin. Therefore, given a scale of Banach spaces $(\cB_\beta)_{\beta\in[0,1]}$ and a path $y_t\in\cB_\beta$, we introduce for a non-negative parameter $\varepsilon\ge0$
the semi-norm
\[
|y|^{(\varepsilon)}_{0,\beta}:= \sup_{t\in[0,T]}t^\varepsilon|y_t|_{\beta}. 
\]
Similarly, for \( R=R_{t,s}\in \cB_\beta\) and \( \gamma\ge0\), we define the weighted $\gamma$-H\"older semi-norm of $R$ in $\cB_\beta$ as 
\[
[R]^{(\varepsilon)}_{\gamma,\beta}:= \sup_{(s,t)\in \Delta_2,\, s>0}s^{\varepsilon}\frac{|R_{t,s}|_{\beta}}{(t-s)^\gamma}
.\]
Furthermore, we denote by $\cC^{\gamma,\varepsilon}([0,T];\cB_\beta)$ the modified H\"older spaces, see~\cite{lunardi,HN19}. We drop the dependence on the time horizon when this is clear from the context. 

\subsection{Scales of Banach spaces and controlled paths}\label{scale}

Let \( J\subset \R \) be an interval.
In the sequel, we call \textit{scale} any family of separable Banach spaces $(\cB_\beta,|\cdot|_\beta)_{\beta \in J}$ such that \( \cB_{\alpha}\hookrightarrow \cB_{\beta} \) for each \( \alpha\le \beta \in J\).
Throughout this manuscript we will focus on the case \( J=[0,1] \) and \( \cB_1 \) is the generator domain of a given analytic semigroup on \( \cB_0 \).\\
More precisely, we assume that $(\cX,|\cdot|)$ is a Banach space and
denote by $(L,D(L))$ the generator of a strongly continuous semigroup of contractions $(S_t)_{t\ge0}$ on \( \cX.\)
For simplicity, we assume that the resolvent set $\rho(L)$ contains \( 0 \), i.e.\ \( L\colon D(L)\to \cX \) is invertible.
We then set \( (\cB_0,|\cdot|_0):=(\cX,|\cdot|)\) and \( (\cB_1,|\cdot|_1)=(D(L),|L\cdot|)\) (whose norm the norm is equivalent to the usual graph norm).
For the intermediate spaces \( \cB_\beta\) with $\beta\in(0,1)$, we make the following assumption.
\begin{assumption}
	\label{ass:intermediate}
 The scale \( (\cB_\beta)_{\beta\in [0,1]}\) satisfies for any \( \alpha\le \beta\le\gamma \in [0,1] \) with  \( \beta\in (0,1)\) the relation
\begin{equation}
\label{intermediate_spaces}
(\cB_\alpha,\cB_\gamma)_{\theta,1}\hookrightarrow
 \cB_\beta \hookrightarrow
(\cB_\alpha,\cB_\gamma)_{\theta,\infty},\quad \quad 
\text{where }\theta= \frac{\beta-\alpha}{\gamma-\alpha}.
\end{equation}
\end{assumption}

\begin{remark}
	Condition \eqref{intermediate_spaces} guarantees in particular that \( (\cB_\beta)_{\beta\in [0,1]} \) is a monotone family of interpolation spaces in the sense of \cite{gerasimovics2019hormander}. This means that
	for each $\alpha \leq \beta \leq \gamma \in [0,1]$ and $x \in \cB_{\alpha}\cap \cB_{\gamma}$, the following interpolation inequality
	\begin{equation} \label{interpolation}
	|x|^{\gamma-\alpha}_\beta \lesssim |x|^{\gamma-\beta}_\alpha |x|^{\beta-\alpha}_\gamma
	\end{equation}
	holds true.
	For a proof of this statement, we refer to \cite[Proposition 1.3.2]{lunardi}. Note that the inequality \eqref{interpolation} is equivalent to the first embedding in \eqref{intermediate_spaces}.
\end{remark}

\begin{example}
	\label{exa:scales} Relation \eqref{intermediate_spaces} is satisfied for instance when \( \cB_\beta:=F_\beta(\cB_0,\cB_1)\), where \( F_\beta\) is either the complex interpolation \( F_\beta:=[\cdot,\cdot]_\beta\), or any of the real interpolation functors \( (\cdot,\cdot)_{\beta,p}\) for \( p\in [1,\infty]\). We refer again to~\cite{lunardi} for further details regarding interpolation theory and point out important function spaces that fall into our framework.
In particular, when the underlying domain is the full space \( \dom=\R^n\), Assumption \ref{ass:intermediate} covers the following important cases:
\begin{itemize}
\item[1)] (Besov or Sobolev spaces) \( \cB_\beta :=(L^p(\cO),W^{2,p}(\cO))_{\beta,q}=\mathscr B^{2\beta}_{p,q}(\cO)\) for any  \( p\in [1,\infty]\).   
Note that the Sobolev scale corresponds to the choice \( p=q\) for which \( \cB_\beta =W^{2\beta,p}(\cO) \). 
\item[2)] (Bessel potential spaces) \( \cB_{\beta}:=[L^p(\cO),W^{2,p}(\cO)]_{\beta}=H^{2\beta,p}(\cO)\).
\item[3)] ``Sobolev tower'', see Example \ref{exa:sobolev_tower} below.
\end{itemize}
\end{example}

\begin{example}[Sobolev tower]
\label{exa:sobolev_tower} 
Assume that \( L \) is the generator of an analytic semigroup on a separable Banach space \( \cX\).
In this case we can introduce (see~\cite[Chap.~2]{Yagi1})) the \textit{fractional powers} of \( L \) as follows.
If \( z\in \C \) is such that \( \Re z >0\), we let 
\begin{equation}
\label{frac_power}
(-L)^{-z}=\frac{1}{2\pi i}\int_{\Gamma}\zeta^{-z}(\zeta + L)^{-1}d\zeta,
\end{equation}
where \( \Gamma \subset \mathbf C\) is any integral contour surrounding \( \sigma(-L) \) counterclockwise in \( (\mathbf C\setminus (-\infty,0])\cap \rho(-L) \).
As is well-known, the integral \eqref{frac_power} is convergent in \( \cL(\cX) \) and coincides with \( (-L)^{-n} = (-L^{-1})^n \) for any integer \( z=n \).
When \( \text{Re}(z) <0 \) we define \( D(L^z)=\text{Im} (L^{-z})\subset \cX \) and 
 \( L^{z} =(L^{-z})^{-1}\), which is unbounded. Keeping this definition in mind, we now state the following standard and useful properties of the fractional powers:
\begin{itemize}
\item \( (-L)^x \) are bounded operators on \( \cX \) for \( -\infty<x\le 0 \) and densely defined, closed linear operators for \( x>0 \).
\item \( D(-L^{x_2})\subset D(-L^{x_1}) \) if \( 0\le x_1\le x_2<\infty \);
\item \( (-L)^x(-L)^{x'} = (-L)^{x'}(-L)^{x} = (-L)^{x+x'} \) for any \( x,x'\in \R. \) 
\end{itemize}

With this definition at hand, we can build a scale \( (\cB_{\beta})_{\beta\in [0,1]}\) as follows:
\( \cB_0=\cX\), 
\(\cB_\beta=(D((-L)^\beta),|(-L)^\beta\cdot|)\) for \( \beta\in (0,1)\).
Moreover, condition \eqref{intermediate_spaces} is fulfilled if \( L\) has bounded imaginary powers (see \cite[Chap.~4]{lunardi}). This is particularly satisfied if the operator \( L\) is positive and self-adjoint.
\end{example}

We now provide some fundamental concepts from rough path theory starting with the definition of a $d$-dimensional $\gamma_0$-H\"older rough path for $\gamma_0\in(\frac{1}{3},\frac{1}{2})$.

\begin{definition}[$\gamma_0$-H\"older rough path]\label{hrp}
	Let $J \subset \mathbb{R}$ be  a compact interval. We call a pair $\mathbf{X}=(X,\XX)$ $\gamma$-H\"older rough path if $X\in C^{\gamma_0}(J, \mathbb{R}^d)$ and $\XX\in C^{2\gamma_0}(\Delta_{J}, \mathbb{R}^d\otimes \mathbb{R}^d)$, where $\Delta_{J} := \left\{\left(s,t \right) \in J^2 \colon  s \leq t \right\}$.   
	Furthermore $X$ and $\XX$ are connected via Chen's relation, meaning that
	\begin{align}\label{chen}
	\XX_{t,s} - \XX_{u,s} - \XX_{t,u} = \dd X_{u,s}\otimes \dd X _{t,u}
	,~~ \mbox{for } s,u,t \in J,~~ s \leq u \leq t .
	\end{align}
\end{definition}

\begin{definition}
	Let $J\subset\mathbb{R}$ be a compact interval and let $\mathbf{X}$ and $\mathbf{\tilde{X}}$ be two $\gamma$-H\"older rough paths. We introduce the $\gamma$-H\"older rough path (inhomogeneous) metric
	\begin{align}\label{rp:metric}
		\rho_{\gamma_0}(\mathbf{X},\mathbf{\tilde{X}} )
		:= \sup\limits_{(s,t)\in \Delta_J} \frac{|\dd X_{t,s}-\dd\tilde{X}_{t,s}|}{|t-s|^{\gamma_0}} 
		+ \sup\limits_{(s,t) \in \Delta_{J}}
		\frac{|\XX_{t,s}-\mathbb{\tilde X}_{t,s}|} {|t-s|^{2\gamma_0}}.
	\end{align}
	We set $\rho_{\gamma_0}(\mathbf{X}):=\rho_{\gamma_0,[0,T]}(\mathbf{X},0)$ and denote the set of $d$-dimensional $\gamma_0$-H\"older rough paths on the interval $[0,T]$ by $\cC^{\gamma_0}(0,T;\R^d)$.
\end{definition}

Our next task is to introduce a suitable space of \( X \)-controlled paths, in the spirit of \cite{gerasimovics2020non}.
Here we introduce additional parameters in order to quantify the space and time regularity of our paths in an optimal way. More precisely, we consider a scale \((\cB_\beta)_{\beta\in [0,1]}\) subject to Assumption \ref{ass:intermediate}, let $\gamma\in(\frac{1}{3},\frac{1}{2})$ denote the regularity of the random input and fix a parameter \( \alpha\in (\sigma,1)\) where $\sigma\in[0,\gamma]$. 

\begin{definition}[\( X\)-controlled rough paths according to a scale of Banach spaces]
	\label{def:controlled}
	Let	\( \X \equiv(X,\XX)\in \mathscr C^{\gamma}(0,T;\R^d)\) for some \( \gamma\in (\frac13,\frac12)\) and assume that \( \sigma\in [0,\gamma]\) is given. Consider two continuous paths \( y\colon [0,T]\to \cB_\alpha\) and \( y'\colon [0,T]\to \cL(\R^d;\cB_{\alpha-\sigma})\simeq \cB_{\alpha-\sigma}^d\).
	We say that the pair $(y,y')$ is \emph{controlled by $X$} (with spatial regularity loss \( \sigma\)) if there is a \( \gamma'\in (1-2\gamma,\gamma]\) such that
	\begin{enumerate}[label=(\roman*)]
		\item\label{D1}
		We have \(y\in \cC([0,T];\cB_\alpha)\cap \cC^\gamma((0,T];\cB_{\alpha-\sigma})\) while \(y'\in\cC((0,T];\cB_{\alpha-\sigma}^d)\cap\cC^{\gamma'}((0,T];\cB_{\alpha-\sigma-\gamma'}^d)\);
		\item\label{D2}
		The remainder $R^y$ defined as:
		\begin{equation*}
		R^y_{t,s} =\dd y_{t,s}-y'_t\cdot \dd X_{t,s} =  \dd y_{t,s} - \sum\nolimits_{i=1}^dy_t^{\prime,i}\dd X^i_{t,s}\text{ for each } 0\leq s\leq t \leq T ,
		\end{equation*}
		belongs to \( \cC_2^{\gamma+\gamma'}((0,T];\cB_{\alpha-\sigma-\gamma'}) \).
	\end{enumerate}
Keeping this in mind, we additionally introduce a suitable weight $\varepsilon\geq 0$ at the origin to describe a norm on the space of such paths, recall Subsection~\ref{weight}. We denote the space of all controlled rough paths by $\cD^{\gamma,\gamma'}_{X,\alpha ,\sigma}([0,T];\varepsilon)$, or simply $\cD^{\gamma,\gamma'}_{X,\alpha ,\sigma}$ when the dependence on the time horizon $T>0$ and weight \( \varepsilon\ge0\) is clear from the context.
	We endow this space with the norm
	\begin{equation}
	\label{new_norm}
	\|y,y'\|_{\cD^{\gamma,\gamma'}_{X,\alpha,\sigma}} := |y|_{0,\alpha} +[\dd y]^{(\varepsilon)}_{\gamma,\alpha-\sigma}+ |y'|_{0,\alpha-\sigma}^{(\varepsilon)} + [\dd y']^{(2\varepsilon)}_{\gamma',\alpha-\sigma-\gamma'} +  [R^y]^{(2\varepsilon)}_{\gamma+\gamma',\alpha-\sigma-\gamma'}\,.
	\end{equation}
Therefore $(\cD^{\gamma,\gamma'}_{X,\alpha,\sigma}, \|\cdot\|_{\cD^{\gamma,\gamma'}_{X,\alpha,\sigma}})$ is a Banach space. 
	\end{definition}

\begin{remark}
	\begin{itemize}
		\item [1)] If \( \gamma'=\gamma=\sigma\), the space $\cD^{\gamma,\gamma'}_{X,\alpha,\sigma}$ reduces to the space of controlled rough paths $\cD^{2\gamma}_{X,\alpha}$ introduced in~\cite{gerasimovics2020non}.
		\item [2)] As we will see later on, the natural choice of the weight is $\varepsilon=(\gamma-\sigma)_+$. This means that $\varepsilon$ reflects the interplay between the time regularity of the rough path and the loss of spatial regularity of the controlled rough path.
		\item [3)] A natural example of an element which belongs to $\cD^{\gamma,\gamma'}_{X,\alpha,\sigma}([0,T];(\gamma-\sigma)_+)$ is given by the pair $(S_tx,0)$, where $(S_t)_{t\geq 0}$ is an analytic semigroup and $x\in\cB_\alpha$, compare Lemma~\ref{lem:weight}.
	\end{itemize}
\end{remark}

A useful consequence of the interpolation inequality~\eqref{interpolation} is that the norm \eqref{new_norm} is equivalent to the apparently stronger one introduced in \cite[Sec.~4.2]{gerasimovics2020non}.	
\begin{lemma}
	\label{lem:interpolation}
	Let \( (\cB_{\beta})_{\beta\in [0,1]} \) be as in \eqref{intermediate_spaces} and fix \( \sigma\in [0,\gamma)\) and \( \alpha\in [2\sigma,1)\).
	The following interpolation estimates hold
	for each \( \theta \in [0,1]\):
	\begin{align}
		\label{apriori_dd_y_epsilon}
		[\dd y]^{(\theta\varepsilon)}_{\theta\gamma,\alpha-\theta\sigma} \lesssim _T (1\vee\rho_\gamma(\X))\|y,y'\|_{\cD_{X,\alpha,\sigma}^{\gamma,\gamma'}}\,,
		\\
		\label{apriori_R_gamma}
			[R^y]^{(\varepsilon(1+\theta))}_{\gamma+\theta\gamma',\alpha-(1+\theta)\sigma-\gamma'} 
			\lesssim (1\vee\rho_\gamma(\X))\|y,y'\|_{\cD_{X,\alpha,\sigma}^{\gamma,\gamma'}}\,,
			\\
			\label{apriori_dd_prime}
			[\dd y']^{(\varepsilon(1+\theta))}_{\theta\gamma',\alpha-(1+\theta)\sigma-\gamma'}\lesssim (1\vee\rho_\gamma(\X))\|y,y'\|_{\cD_{X,\alpha,\sigma}^{\gamma,\gamma'}}\,.
	\end{align}

As a further consequence, the norms \( \|y,y'\|_{\cD_{X,\alpha,\gamma}^{\gamma,\gamma}} \) and \(\|y,y'\|_{\cD_{X,\alpha}^{2\gamma}}= \|y,y'\|_{\cD_{X,\alpha,\gamma}^{\gamma,\gamma}} -[\dd y]_{\gamma,\alpha-\gamma}+ [R^y]_{\gamma,\alpha-\gamma} \) are equivalent under Assumption \ref{ass:intermediate}.
\end{lemma}
\begin{proof}
This is immediate from the property \eqref{interpolation} and the definition of the \( \cD_{X,\alpha,\sigma}^{\gamma,\gamma'}\)-norm.
\end{proof}

\subsection{Notion of solution and main result}
Let us fix a scale $(\cB_\alpha)_{\alpha\in[0,1]}$ subject to Assumption \ref{ass:intermediate}, $\gamma_0\in(\frac{1}{3},\frac{1}{2})$ and a $d$-dimensional \( \gamma \)-H\"older rough path \( \X\in \mathscr C^{\gamma_0}(0,T;\R^d)\). We now discuss a notion of mild solution for the quasilinear evolution equation
\begin{equation}
\label{rough_PDE_gene}
du_t=(L_t(u_t)u_t+ N_t(u_t))dt+\sum_{i=1}^dF^i(u_t)d\X^i_t\,,
\quad u_0=x\in\cB_\alpha.
\end{equation}
We firstly specify the assumptions on the quasi-linear term \(L(\cdot)\).
\begin{assumption}
	\label{ass:L_y}
Let $0<\eta\le\alpha \leq 1$.
We assume that
\begin{enumerate}[label=(Q\arabic*)]
	\item\label{Q1}
There is an open set $V$ of $\cB_\eta $ such that
for some $\varrho \in(0,1)$
\begin{equation}\label{regularity_L}
[(t,y)\mapsto L_t(y)]\in C^{\varrho ,1-}([0,T]\times V,\cL(\cB_1,\cB_0)).
\end{equation}
This means that there exists a constant $\ell>0$ such that for each $t\in[0,T]$ and every $x,y\in V$
\begin{equation}
	\label{lip_constant}
	|L_t(x)-L_s(y)|_{\cB_1\to \cB_0}\leq \ell [|t-s|^\varrho+|x-y|_{\eta}].
\end{equation}

\item \label{Q2}
For each point $y\in V,$ there exists a neighbourhood $ V_{y}\subset V$ in $\cB_\eta$ and constants $M>0$, $\lambda \in\R$ such that for all $t,y\in [0,T]\times V_{y}$ we have
$\lambda +\Sigma_0 \subset \rho (L_t(y))$ 
and 
\[
|(\zeta -L_t(y))^{-1}|_{\cB_0\to\cB_0}\leq \frac{M}{1+|\zeta -\lambda |}\,,
\]
for every $\zeta \in \lambda +\Sigma_0$,
where $L_t(y)$ is considered as a linear operator in $\cB_0$ with domain $\cB_1.$
\end{enumerate}
\end{assumption}
\begin{definition}
\label{def:var_sol}
Let \( N,F\) be nonlinear terms satisfying suitable assumptions.
We call $u\in \cC(0,\tau,\cB_\alpha)\cap \cC^{\gamma,\varepsilon}(0,\tau,\cB_{\alpha-\sigma})$ 
a \emph{mild solution} of \eqref{rough_PDE_gene} if
\begin{itemize}
\item it holds \( \int _0^T|N_r(u_r)|_{0}dr<\infty\);
\item the pair $(u,F(u))$ defines a controlled rough path in $\cD^{\gamma,\gamma'}_{X,\alpha,\sigma}(0,\tau;\varepsilon)$, where $\varepsilon=(\gamma-\sigma)_{+}$;
\item the path component $u$ satisfies the variation of constants formula
\begin{equation}\label{nota:solution}
	u_t = S_{t,0}u_0 + \int_0^t S_{t,r}N_r(u_r)dr+ \int_0^t S_{t,r} F(u_r) \cdot d\X_r \, , \quad t<\tau\,,
\end{equation}
where the last integral on the right hand side is well defined in the sense of controlled rough paths, see Theorem \ref{thm:integral} below.
\end{itemize}
Given \( x\in \cB_\alpha \) we call \( u \) a mild solution starting at \( x \) if it additionally holds that \( u_0=x \).
\end{definition}

We can now state our main result.
\begin{theorem}
	\label{thm:main}
Fix \( \frac13<\gamma<\gamma_0<\frac12 \), a $d$-dimensional $\gamma_0$-H\"older rough path \(\X=(X,\XX)\in \mathscr C^{\gamma_0}(0,T;\R^d)\) and parameters $\sigma\in [0,\gamma]$, \( \alpha\in (1-\gamma-\gamma_0+2\sigma,1-\gamma+\sigma]\).
Let $L_t(y)$ be as in Assumption \ref{ass:L_y} and consider a non-linearity \( F=(F_1,\dots,F_d)\) which is well-defined and three times continuously differentiable from \( \cB_\beta\to \cB_{\beta-\sigma}\) for each \( \beta\in [\beta_0,\alpha]\) where \( \beta_0 \in (\sigma,\alpha-\sigma-1 +\gamma_0+\gamma)\).
Similarly, let \( \delta\in [0,\alpha)\) and consider a non-linearity \( N\colon [0,T]\times \cB_{\beta}\to\cB_{\beta-\delta},\) \( \beta\in[\alpha, \alpha+\gamma-\sigma] \), such that \( N_t(\cdot) \) is Lipschitz with constant \( C_t \) for each \( t\in [0,T] \) and moreover \( \|N\|_1 =\sup_{t\in [0,T]}C_t<\infty.\) 
We have the following
\begin{itemize}
\item Suppose that \( \alpha<1. \)
Given $x\in \cB_\alpha,$ 
there exists a unique maximal solution to the problem \eqref{problem}.
Moreover, the solution map $(x,\X)\mapsto u$ is continuous from $\cB_{\alpha}\times \mathscr C^{\gamma_0}(0,T;\R^d)$ to $\cC([0,\tau);\cB_{\alpha})\cap \cC^{\gamma,\varepsilon}([0,\tau);\cB_{\alpha-\sigma})$ where \( \varepsilon=\gamma-\sigma\) and
$\tau$ denotes the maximal lifespan of the solution. If \( \tau<T\), we have the alternative that \(\limsup_{t\to\tau}|u_t|_\alpha=\infty\) or there is a \( y\in \partial V\) such that%
\footnote{In the case when the boundary \( \partial V\subset \cB_\eta\) is empty, it is understood as the statement that \( \limsup _{t\to \tau}|u_t|_{\eta}=\infty\).} 
\( u_t\to y\) in \( \cB_\eta\).
\item The same conclusion holds if \( \alpha=1 \), and provided we assume furthermore that
\begin{enumerate}[label=(\roman*)]
\item \( \dd L\in \cC^{\varrho}_2 \left (0,T;C^{1-}\Big(V;\cL(\cB_1,\cB_0)\Big)\right )\);
\item \( L\in \cC_1^0\Big(0,T;C^{2-}\Big(V;\cL(\cB_1,\cB_0\Big)\Big) \), by which we mean that for each \( t\in [0,T] \), \( L_t \colon V\to\cL(\cB_1,\cB_0)\) is continuously Fr\'echet differentiable 
 and \( DL_t \) is Lipschitz continuous, uniformly in \( t \in [0,T]\). 
\end{enumerate}
\end{itemize}
\end{theorem}
As a particular application of our results when \( d=3 \) and \( \X\) is a geometric rough path with the regularity of the Brownian motion, we obtain existence and uniqueness for the following Landau-Lifshitz-Gilbert equation with quadratic rough input
\[
\left \{\begin{aligned}
& du - (\Delta u + u\times \Delta u) dt = u|\nabla u|^2dt + u\times d\X -\epsilon u\times (u\times d\X),\quad \text{on } [0,T]\times \dom,
\\&
u_0\quad \text{given in }W^{2\alpha,p}(\dom;\R^3)\quad \text{with }|u_0(x)|=1\text{ a.e.,}
\end{aligned}\right .
\]
where \( \epsilon\in \{0,1\}\), the domain
\( \dom \subset \R^n\) is either bounded with smooth boundary or the whole space \( \R^n \), and the boundary conditions are either periodic, Neumann-homogeneous or void.
In this case the maximal lifespan \( \tau\in (0,T]\) is characterized by the property that \( t=T\) or \( \tau<T\) and \( \limsup_{t\to\tau} |u_t|_{W^{2\alpha,p}} =\infty\) for any \( \alpha >\frac{1}{2}+\frac{n}{2p}\).
In a random setting when \( \X \) is a rough path obtained from the Stratonovitch enhancement of a three-dimensional Brownian motion and for \( \dom=(-1,1)\) (with periodic boundary conditions), we will see in Section \ref{sec:stochastic_examples} how our results permit to obtain existence of a stochastic flow for the corresponding stochastic LLG equation with linear noise (i.e.\ when \( \epsilon=0\)), in the case when \( p=2 \) and \( \alpha=1 .\)

\subsection{A lemma on dyadic approximations}

Before closing this preliminary section, we state an important lemma on dyadic approximation which can be seen as a generalization of an argument in \cite{gerasimovics2019hormander} used to approximate certain Riemann-type integrals via dyadic partitions.

Herein, given \( (s,t)\in \Delta_2\) and \( n=0,1,\dots\), we introduce the corresponding \( n\)-th dyadic partition:
\begin{equation}\label{dyadic}
	\pi_n(s,t)=
	\{s\le s+2^{-n}(t-s)\le \dots \le t-2^{-n}(t-s) \le t \}\,.
\end{equation}

\begin{lemma}\label{lem:dyadic}
	Let \( T\le 1\) and 
	suppose that \( \mathscr J_{t,s}=\lim_{n\to \infty} \sum_{[u,v]\in\pi_n(s,t)}J_{u,v} \) for some \( J\in \cC_2(0,T;\cX)\). Let \( l\in \mathbf N\),
	assume the existence of real numbers \( \mu_i>1\), \( 0\le\lambda_i<\mu_i\), \( \nu_i \in [0,\mu_i]\) and \( \varepsilon_i \in [0,1)\), as well as constants \( A_1,\dots ,A_l\ge0\) for \( i=1,\dots ,l \), such that
	\begin{equation}
		\label{hyp:dyadic}
		|\dd J_{v,m,u}| _{\cX} 
		\le \sum_{i=1}^lA_i|t-u|^{-\lambda_i}|u-m|^{\nu_i}|m-v|^{\mu_i-\nu_i }m^{-\varepsilon_i}\,.
	\end{equation}
	for any \(0\le s\le u\le m\le v\le t\le T\) with \( m>0.\)
	Then, assuming that \( \omega\in[0,\vee_{i=1}^l \varepsilon_i]\) is chosen so that
	\( (\varepsilon_i-\omega)_+<\mu_i-\lambda_i\) for each \( i=1,\dots, l\), one has the estimate
	\[
	s^{\omega}|\mathscr J_{t,s}-J_{t,s}|_{\cX} \lesssim \sum_{i=1,\dots ,l}A_i |t-s|^{\mu_i-\lambda_i-(\varepsilon_i-\omega)_+}\,.
	\]
\end{lemma}

\begin{proof}
	Let \( n\ge0\), \( (s,t)\in \Delta_2\) and take any \( [u,v]\in \pi(s,t)\). By assumption, we have that
	\[
	\begin{aligned}
		\mathscr J_{t,s}
		&=\lim_{n\to \infty} \sum_{[u,v]\in \pi_n}J_{u,v}
		\\
		&=\sum_{n=0}^\infty (J^{n+1}_{t,s}-J^n_{t,s})\,,
	\end{aligned}
	\]
	where we introduce the partial sum \( J^n_{t,s}=\sum _{[u,v]\in \pi_n(s,t)}J_{u,v}.\)
	Next, one observes for each $n\ge0$ the algebraic identity
	\begin{equation}\label{algebraic_id_app}
		\begin{aligned}
			J^{n+1}-J^n
			&= 
			\sum_{\substack{[u,v]\in\pi_n\\ m=\frac{u+v}{2}}}\dd J_{v,m,u}\,.
		\end{aligned}
	\end{equation}
	In particular, the fact that the previous identity is linear in \( \dd J \) allows one to assume that \( l=1 \) without any loss of generality.
	In this case, if we denote by \( \mu=\mu_1,\lambda=\lambda_1 ,\nu=\nu_1\), \( \varepsilon=\varepsilon_1\), \( A=A_1 \),
	we observe thanks to \eqref{hyp:dyadic} and the triangle inequality, that
	\begin{equation}\label{sum}
		\begin{aligned}
			s^{\omega}|J^{n+1}-J^n|_{\cX}
			&\le 
			\sum_{\substack{[u,v]\in\pi_n\\ m=\frac{u+v}{2}}}
			\Big( A|t-u|^{-\lambda}|v-m|^{\nu}|m-u|^{\mu-\nu}m^{-\varepsilon}s^{\omega}\Big).
		\end{aligned}
	\end{equation}
	Note that \( |t-u|^{-\lambda}\leq |t-m|^{-\lambda}\) and that if \( \omega\in [0,\varepsilon]\) we have similarly \( m^{-\varepsilon}s^{\omega}\le |m-s|^{-(\varepsilon-\omega)_+}\). Moreover,
	if \( [u,v]\in \pi_n(s,t)\), then \( |m-u|=|v-m|=\frac12|u-v|=2^{-n-1}|t-s|\). Hence, we see that each
	term appearing in the summands of \eqref{sum} can be estimated from above by
	\[
	A(2^{-n-1}|t-s|)^{\mu-1}|t-m|^{-\lambda} |m-s|^{-(\varepsilon-\omega)_+}|u-v|\,.
	\]
	which is also bounded by
	\[
	A(2^{-n-1}|t-s|)^{\mu-1-\varkappa}|t-m|^{-\lambda+\varkappa}|m-s|^{-(\varepsilon-\omega)_+} |u-v|\,.
	\]
	for any \( \varkappa\in [0,\mu-1).\)
	Let us choose the latter so that \(\varkappa -\lambda>-1 \) (which is possible by the assumption that \( \lambda<\mu\)).
	Summing over $[u,v]\in\pi_n$ gives by comparison of the Riemann sum with its integral (the integrand is convex)
	\[
	\begin{aligned}
		s^{\omega}|J^{n+1}-J^n|_{\cX}
		&\leq A2^{-n(\mu-1-\varkappa)}|t-s|^{\mu-1-\varkappa}\sum _{[u,v]\in \pi_n}|t-m|^{\varkappa-\lambda}|m-s|^{-(\varepsilon -\omega)_+}|u-v|
		\\
		&\leq A2^{-n(\mu-1-\varkappa)}|t-s|^{\mu-1-\varkappa}\int_s^t|t-r|^{\varkappa-\lambda}|r-s|^{-(\varepsilon-\omega)_+}dr
		\\
		&\leq A2^{-n(\mu-1-\varkappa)}|t-s|^{\mu-\lambda-(\varepsilon-\omega)_+}\,\mathrm B(\varkappa-\lambda,1-(\varepsilon-\omega)_+)\,,
	\end{aligned}
	\]
	by the change of variable \( r=s+\tau (t-s) ,\) \( \tau\in [0,1] \), where we recall the definition of the Euler Beta function \(\mathrm B(x,y)=\int _0^1\theta^{x-1}(1-\theta)^{y-1} d\theta\), for each \( x,y>0.\)	 
	One gets the desired estimate by summing over $n\ge0.$
\end{proof}

\section{Operations on rough controlled paths}
\label{sec:rough_convolutions}

\subsection{Functional analytic framework}
Fix two Banach spaces $\cX_1\hookrightarrow\cX$. 
We recall the definition of an evolution family on $\cX$ in the sense of \cite{amann1986quasilinear,Yagi1} (or ``propagator'' \cite{gerasimovics2020non}).
\begin{definition}
	\label{def:propagator}
	Let $(L_t)_{t\in[0,T]}$ be a family of unbounded, closed operators with constant domain  $D(L_t)\equiv\mathcal X_1\hookrightarrow \cX$.
We call $S\colon \Delta _2\to \cL(\cX)$ the \emph{evolution family} associated to \( (L_t)_{t\in[0,T]} \) if and only if
	\begin{enumerate}[label=(P\arabic*)]
		\item\label{P1} $S\in \cC_2(0,T;\cL_s(\cX))$ and there exist constants $M ,\lambda>0$ such that for every $(t,s) \in \Delta_2$:
		\begin{equation}
			\label{growth_S}
			|S_{t,s}|_{\cX\to\cX}
			\leq M e^{\lambda(t-s)}\;.
		\end{equation}
		Here $\cL_s(\cX)$ denotes the strong operator topology on $\cX$.
		\item\label{P2} \( S \) is multiplicative, namely $S_{t,t}=\id$ and $S_{t,s}=S_{t,u }S_{u ,s}$ for $(t,u ,s)\in\Delta _3$. 
		
		\item\label{P3} For all $s,t \in [0,T]$ and $x \in \mathcal X_1$ we have
		\( \frac{\partial }{\partial t}S_{t,s}x = L_tS_{t,s}x \) and  \( \frac{\partial}{\partial s}S_{t,s}x=-S_{t,s}L_sx\)
		(where the differentiation is taken in the Banach space $\cX$).
		\item\label{P4} For every $(t,s)\in \Delta_2,$ $s\neq t$ we have that $S_{t,s} \cX \subset \mathcal X_1$ and for some constant $\|S\|_{(0,1)}>0$
		\begin{equation}
			\label{smoothing_S_prelim}
|S_{t,s}|_{\cX\to\cX_1} \leq  \|S\|_{(0,1)}|t-s|^{-1}.
		\end{equation}
	\end{enumerate}
\end{definition}
\begin{remark}
	\label{rem:constant_K}
	For $(L_t)_{t\in [0,T]}$  as above, the following additional property is well-known (see \cite{amann1986quasilinear}):
	\begin{enumerate}[label=(P4*)]
		\item\label{P4star} There exists a constant $\|S-I\|_{(1,0)}>0$ such that for every $(t,s)\in \Delta_2,$
		\begin{equation*}
			\label{anti_smoothing_S_raw}
			|S_{t,s}-\id|_{\cX_1\to\cX} \leq \|S-I\|_{(1,0)} |t-s|\,.
		\end{equation*}
	\end{enumerate}
\end{remark}
\begin{remark}
Consider a scale \( (\cB_\beta)_{\beta\in [0,1]} \) subject to Assumption \ref{ass:intermediate} and let \( (L_t)_{t\in [0,T]}\) be as above with \( \cX=\cB_0\) and \( \cX_1=\cB_1 .\)
By interpolation, there exists a constant \( K= K(M,T,\|S\|_{(0,1)}) \) such that for each \( 0\le \alpha,\beta\le 1 \)
\begin{equation}
\label{smoothing_S}
|S_{t,s}|_{\alpha\to\beta} \leq  K|t-s|^{-(\beta-\alpha)_+}\,.
\end{equation}
Similarly, there exists $\tilde K(M,\varrho,T,\|S-I\|_{(1,0)})>0$ such that for every $(t,s)\in \Delta_2,$
\begin{equation}
	\label{anti_smoothing_S}
	|S_{t,s}-\id|_{\alpha\to\beta} \leq \tilde K(M,\rho,T) |t-s|^{(\alpha-\beta)_+}\,.
\end{equation}
\end{remark}

If one is working with a constant family \( L_t\equiv L \), the corresponding propagator is the semigroup \( S_{t,s}=\exp((t-s)L) \), which in the analytic case is given for any \( \tau>0 \) by the contour integral formula
\begin{equation}
\label{contour_integral}
\exp(\tau L)=\frac{1}{2\pi i}\int _{\Gamma} e^{\zeta \tau}(\zeta - L)^{-1}d\zeta,
\end{equation}
where \( \Gamma\subset \C \) is as in \eqref{frac_power}.
The following well-known conditions guarantee that a given family of non-autonomous operators with constant domain generates a propagator (see Tanabe \cite{tanabe1960equations}).

\begin{assumption}\label{ass:L_t}
	Let $\cX_1\hookrightarrow \cX$ be Banach spaces. 
	Assume that $(L_t)_{t\in[0,T]}$ is a family of closed, densely defined linear operators on $\cX$ with constant domain $\cX_1$. Moreover, there exist constants $C,M>0$ (depending only on $T$) such that:
	\begin{enumerate}[label=(L\arabic*)]
		\item\label{L1} For each $t\in[0,T]$, the resolvent $\rho (L_t)$ contains $\Sigma_\vartheta$ and there exists a constant $M>0$ such that 
		\begin{equation}\label{resolvent}
		|(\zeta  - L_t)^{-1}|_{\cX\to \cX}\leq \frac{M}{1+|\zeta |},
		\quad \forall\, \zeta \in\Sigma_\vartheta
		\end{equation}
		
		\item\label{L3} there is a number \( \varrho>0 \) such that \( L\in C^\varrho([0,T];\cL(\cX_1,\cX)) \), namely
		\( |L_t-L_s|_{\cX_1\to\cX}\leq \ell|t-s|^\varrho \) for some constant \( \ell>0 \) independent of \( s\le t\in [0,T] .\)
	\end{enumerate}
\end{assumption}

\begin{remark}
	\label{rem:evolution}
Under Assumption \ref{ass:L_t},
the constants \( K,\tilde K\) in \eqref{smoothing_S}-\eqref{anti_smoothing_S} can be chosen to depend only on the parameters \( \ell,M,T,\rho \). This is indeed a well-known consequence of the representation formula \eqref{contour_integral} and the construction of the propagator \( S \) as the unique solution of the integral equations
\begin{equation}
\label{first_int}
S= a + S\star (\dd L. a)
\end{equation}
\begin{equation}
\label{second_int}
S = b - (b.\dd L)\star S\,,
\end{equation}
where for  \( (s,t)\in \Delta_2 \), \( \dd L_{t,s}=L_t-L_s \), \(a_{t,s}=e^{(t-s)L_s}\), \(b_{t,s}=e^{(t-s)L_t}\) and \( (f\star g)_{t,s}:=\int_s^tf_rg_rdr\).
We refer to \cite[Sec.~2]{amann1986quasilinear} for details about this construction.

Throughout this manuscript, we denote the evolution family \( S_{t,s}\), which is implicitly determined by \eqref{first_int}-\eqref{second_int} by \( \exp(\int _s^tL_rdr),\) \( (s,t)\in \Delta_2\). 
\end{remark}

\subsection{Rough integration with respect to an evolution family}
In the following we denote by $(\cB_\beta)_{\beta\in [0,1]}$ a scale of Banach spaces subject to Assumption \ref{ass:intermediate}. We firstly state important results regarding the existence and regularity of the rough convolution $\int_{s}^{t}S_{t,s}y_s\cdot d\X_s$, which is given as the limit in a suitable Banach space 
of compensated Riemann-sums of the form
\( I^{\pi}_{t,s}=\sum_{[u,v] \in \pi}S_{t,u}(y_{v}\cdot \dd X_{v,u} + y'_{v}:\xx_{v,u})\)
where we denote 
\[
\dd X^{\otimes2;i,j}_{v,u}= \dd X^i_{v,u}\dd X^j_{v,u}, \quad 
1\le i,j\le d,
\]
while \( \pi=\{[u_1,v_1],\dots, [u_n,v_n]\}\) is a generic partition of \( [s,t]\) with \( n=\#\pi<\infty. \)

\begin{remark}
	Comparing our framework with~\cite{gerasimovics2019hormander,gerasimovics2020non}, where the rough convolution was defined for a controlled rough path $(y,y')\in\cD^{2\gamma}_{X,\alpha}$ as
	\begin{align*}
	\int_0^t S_{t,s}y_s\cdot d\X_s=\lim\limits_{|\pi|\to 0} \sum\limits_{[u,v]\in\pi} S_{t,u}(y_u \cdot\delta X_{v,u} + y'_u : \mathbb{X}_{v,u} ),
	\end{align*}
	we immediately observe that the term $\xx$ is necessary, since we consider the end-point of the partition $v$ instead of $u$. This is natural since we additionally incorporate a weight in the space of controlled rough paths. 
\end{remark}
An important property which is satisfied by controlled paths is that rough integration (with respect to a given evolution family) is a well-defined operation. Moreover it improves spatial regularity by any number strictly less than the H\"older regularity of $\X$.

\begin{theoremdef}\label{thm:integral}
		Let $\X = (X,\XX) \in \mathscr C^{\gamma_0}(0,T;\R^d)$ for some $\gamma_0\in(\frac{1}{3},\frac{1}{2})$, and fix parameters \( \sigma\in[0,\gamma_0)\), \( \beta\in (\sigma,1] \), \( \gamma\in (\frac13,\gamma_0] ,\)  \( \gamma'\in (1-\gamma_0-\gamma,\gamma]\) and \( \varepsilon=(\gamma-\sigma)_+.\)
		Let \( (\zeta,\zeta^{\prime}) \in (\cD^{\gamma,\gamma'}_{X,\beta ,\sigma}([0,T];\varepsilon))^d\).
		 We call  \textit{rough convolution} the one-parameter quantity
		\begin{equation}\label{rough:int:map}
			\begin{aligned}
		t\mapsto\int_0^t S_{t,r}\zeta_r\cdot d\X_r 
		=\lim_{|\pi|\to0}\sum_{[u,v]\in \pi} S_{t,u}(\zeta_v\cdot \dd X_{v,u} + \zeta'_v:(\XX_{v,u}-\dd X^{\otimes2}_{v,u})),\quad t\in(0,T],
			\end{aligned}
		\end{equation}
where \( \pi=\pi(t)\) denotes a generic finite partition of \( [0,t]\) with mesh-size \( |\pi|=\max_{[u,v]\in \pi}|v-u|\), and the limit is taken in \(\cC^{\gamma,\varepsilon}(\cB_{\beta}).\)
 Furthermore, the rough integral is uniquely characterized by the following two properties:
		\begin{enumerate}[label=(I-\arabic*)]
			\item\label{itm:delta_S} (multiplicativity) we have \( \dd^{S}\int S_{\cdot,r}\zeta_r\cdot d\X_r\equiv 0 \), i.e.
			 \[ \int_s^t S_{t,r}\zeta_r\cdot d\X_r=\int_u^t S_{t,r}\zeta_r\cdot d\X_r + S_{t,u}\int_s^u S_{u,r}\zeta_r\cdot d\X_r,\quad \quad  \forall (s,u,t)\in \Delta_3;\]
			\item\label{itm:integ_remainder} (remainder estimate) introducing the \emph{integral remainder}
		\( \mathscr R^{S;\zeta}_{t,s}:= \int_s^tS_{t,u}\zeta_u\cdot d\X_u - S_{t,s}(\zeta_t \cdot \dd X_{t,s}+ \zeta'_t:(\XX_{t,s}-\dd X^{\otimes2}_{t,s}))\),
 for any $(s,t) \in \Delta_2$, the following estimate holds if \( \iota\in [0,2\varepsilon] \), \( \kappa+\iota\in [-\sigma,\gamma_0+\gamma-\sigma) \)  and \( T\le 1\):
			\begin{equation} \label{e:integration2}
			s^{2\varepsilon-\iota}\Big| \mathscr R^{S;\zeta}_{t,s} \Big|_{\beta+\kappa} \lesssim (1\vee K\vee\tilde K)^2 \rho_{\gamma_0}(\X) \|\zeta,\zeta'\|_{\cD^{\gamma,\gamma'}_{X,\beta,\sigma}(\varepsilon)}
			 |t-s|^{\gamma_0+\gamma -\sigma -\kappa-\iota}.
			\end{equation} 
		\end{enumerate} 	
\end{theoremdef}
\begin{remark}
\label{rem:uniqueness}
Uniqueness actually holds in the following broader sense. 
Suppose that a Banach space \( (\cZ , |\cdot|_{\cZ})\) exists such that \( \cB_0\hookrightarrow \cZ \) and let \( \mathscr{\bar  R}=\mathscr{ \bar R}_{s,t} \in \cB_0\) be such that  \( \dd ^S\mathscr{\bar R} \equiv \dd^{S}[S(\zeta\cdot \dd X + \zeta':\xx)]  \) while 
\[
|\mathscr{\bar R}_{t,s}|_{\cZ}=o(t-s)\,.
\]
Then we have necessarily \( \mathscr{\bar R}_{t,s}=\mathscr R^{S,\zeta}_{t,s}=\int_s^t \zeta\cdot d\X - S_{t,s}(\zeta_t\cdot \dd X_{t,s} + \zeta'_t:\xx_{t,s}) \) for each \( s\le t\in [0,T]. \)
\end{remark}

\begin{proof}[Proof of Theorem \ref{thm:integral}]
We focus on \eqref{e:integration2} since existence, multiplicativity and uniqueness follow from the same arguments as that of \cite[Thm 4.5]{gerasimovics2020non}. Moreover, it is clear that the rough integral is linear, hence we assume without loss of generality that \( \|\zeta,\zeta'\|_{\cD^{\gamma,\gamma'}_{X,\beta, \sigma}(\varepsilon)}\le 1. \) To show the claimed estimate, we rely on Lemma \ref{lem:dyadic}.
We consider the approximation term \( (s,t)\mapsto \xi_{t,s}:=\zeta_t\cdot \dd X_{t,s} + \zeta_t':\xx_{t,s} \) and apply Chen's relations~\eqref{chen}. These yield for any \( u\le m\le v\in [0,T]\) that
\begin{equation}\label{chen_plain}
\dd \xi_{v,m,u}= R_{v,m}^\zeta\cdot \dd X_{m,u} + \dd \zeta'_{v,m}\colon \xx_{m,u}.
\end{equation}
Fix \( (s,t)\in \Delta\) and write \( J_{v,u}=S_{t,u}\xi_{v,u}\). 
If \( s\le u\le m\le v\le t \in [0,T]\) the multiplicative structure of \( S\) together with \eqref{chen_plain} further entails
\[\begin{aligned}
	\dd J_{v,m,u}
	&=S _{t,u}\dd\xi _{v,m,u}+S_{t,m}(S _{m,u}-\id)\xi_{v,m}
	\\
	&=S _{t,u}R_{v,m}^\zeta\cdot \dd X_{m,u} + S _{t,u}\dd \zeta'_{v,m}\colon \xx_{m,u} 
	\\&\quad \quad 
	+ S_{t,m}(S _{m,u}-\id)\zeta_v\cdot\dd X_{v,m}+S_{t,m}(S _{m,u}-\id)\zeta'_v:\xx_{v,m}
	\\
	&=: \mathrm{I} + \mathrm{II} +\mathrm{III} +\mathrm{IV}\,.
\end{aligned}\]
Using the smoothing property \eqref{smoothing_S} together with the definition of the norm on $\cD^{\gamma,\gamma'}_{X,\beta,\sigma}$ given by~\eqref{new_norm}, we obtain the following estimates.
\[
\begin{aligned}
	|\mathrm{I}|_{\beta+\kappa}
	&\le \|S\|_{\beta-\sigma-\gamma'\to \beta+\kappa}[X]_{\gamma_0}[R^\zeta]^{(2\varepsilon)}_{\gamma+\gamma',\beta-\sigma-\gamma'}|m-u|^{\gamma_0}|v-m|^{\gamma+\gamma'}m^{-2\varepsilon}
	\\&
	\le K[X]_{\gamma_0}|t-m|^{-\sigma-\gamma'-\kappa}|m-u|^{\gamma_0}|v-m|^{\gamma+\gamma'}m^{-2\varepsilon}\,.
\end{aligned}
\]
Similarly
\[
\begin{aligned}
	|\mathrm{II}|_{\beta+\kappa}
	&\le \|S\|_{(\beta-\sigma-\gamma',\beta+\kappa)}[\xx]_{2\gamma_0}[\dd\zeta']^{(2\varepsilon)}_{\gamma',\beta-\sigma-\gamma'}|v-m|^{2\gamma_0}|v-m|^{\gamma'}m^{-2\varepsilon}\,.
	\\&
	\le K[\xx]_{2\gamma_0}|t-m|^{-\sigma-\gamma'-\kappa}|v-m|^{2\gamma_0}|v-m|^{\gamma'}m^{-2\varepsilon}\,.
\end{aligned}
\]
In order to estimate the third term,  we can fix any number \( \rho\in [1-\gamma, \beta)\) and infer that
\[\begin{aligned}
	|\mathrm{III}|_{\beta+\kappa}
	&\leq|S_{t,m}|_{\beta-\rho\to\beta+\kappa}|(S_{m,u}-\id)|_{\beta\to\beta-\rho}|\zeta_v\cdot \dd X_{v,m}|_{\beta}
	\\
	&\leq K\tilde K[X]_{\gamma _0}|t-m|^{-\rho-\kappa}|v-m|^{\gamma _0}|m-u|^{\rho}.
\end{aligned}
\]
Lastly, we fix any number \( \rho'\in (\gamma'-\sigma,\beta)\) and obtain
\[\begin{aligned}
	|\mathrm{IV}|_{\beta+\kappa}
	&\leq|S_{t,m}|_{\beta-\rho'\to\beta+\kappa}|(S_{m,u}-\id)|_{\beta-\sigma\to \beta-\rho'}|\zeta'_v:\xx_{v,m}|_{\beta-\sigma}
	\\&
	\leq [\xx]_{2\gamma _0}K\tilde K|t-m|^{-\rho'-\kappa}|v-m|^{2\gamma _0}|m-u|^{\rho'-\sigma}m^{-\varepsilon}\,.
\end{aligned}
\]
Now, we can apply Lemma \ref{lem:dyadic}, which yields the conclusion.
\end{proof}

A further consequence of Theorem \ref{thm:integral} is that rough convolutions are also controlled paths, as illustrated by the next result. We also collect bounds that quantify the gain of regularity due to integration.

\begin{corollary}\label{cor:improvereg}
Fix \( \frac13 <\gamma \le \gamma_0 <\frac12  \), 
a rough path \( \X\in\mathscr C^{\gamma_0}(\R^d)\), a spatial loss
\( \sigma\in [0,\gamma_0) \) and let \( \varepsilon= (\gamma-\sigma)_+.\)
Suppose that we are given parameters
\( \alpha\in (1-2\varepsilon-(\gamma_0-\gamma),1-\varepsilon]\)
while \( \gamma'\in(1-\gamma_0-\gamma, \alpha-\sigma)\), 
and finally pick any \( \theta\in (\frac{1-\gamma_0-\gamma}{\gamma'},\frac{\alpha-2\sigma}{\gamma'})\).
	 
	Let \( (\zeta,\zeta')\in \left (\cD^{\gamma,\theta\gamma'}_{X,\alpha-\sigma,\sigma}(0,T;\varepsilon)\right )^d\).
	The path \( z_t=\int_0^tS_{t,r}\zeta_r\cdot d\X_r\) satisfies the estimates
	\begin{equation}\label{est:integ}
		|z|^{(\varepsilon)}_{0,\alpha +\varepsilon}\lesssim (T^{\gamma_0 - \sigma } + T^{2\gamma_0-\sigma-\sigma\vee\gamma} + T^{\gamma_0 - \gamma})\|\zeta,\zeta'\|_{\cD^{\gamma,\theta\gamma'}_{X,\alpha-\sigma,\sigma}}
	\end{equation}
	\begin{equation}
		\label{est:holder_rho}
		[\dd z]_{\varrho;\alpha-\varrho} \lesssim (T^{\gamma_0 -\varrho -(\sigma-\varrho)_+} + T^{2\gamma_0-\varrho-(2\sigma-\varrho)_+}+ T^{\gamma_0-\gamma} )\|\zeta,\zeta'\|_{\cD^{\gamma,\theta\gamma'}_{X,\alpha-\sigma,\sigma}},
	\end{equation}
	for any \( \varrho\in (0,\gamma_0) \).
	
	Letting \(z'=\zeta \) and assuming further that \( \sigma\le \gamma\), \( \gamma'\le \gamma\), then the pair \( (z,z')\) is a well-defined controlled path, with spatial loss \( \sigma\) and \( (\gamma,\gamma')\)-H\"older regularity. Moreover, the following estimates are satisfied:
\begin{equation}
	\label{est:integ_dd}
	\begin{aligned}
	|z|_{0,\alpha} + [\dd z]^{(\varepsilon)}_{\gamma,\alpha-\sigma}
		\lesssim_{K,\tilde K}
		T^{\gamma_0-\gamma}\|\zeta,\zeta'\|_{\cD^{\gamma,\theta\gamma'}_{X,\alpha-\sigma,\sigma}}
	\end{aligned}
\end{equation}
\begin{equation}
\label{est:integ_R}
\begin{aligned}
[R^z]^{(2\varepsilon)}_{\gamma+\gamma',\alpha-\sigma-\gamma'}
\lesssim_{K,\tilde K,\rho_{\gamma}(\X)}
T^{\gamma_0-\gamma}\|\zeta,\zeta'\|_{\cD^{\gamma,\theta\gamma'}_{X,\alpha-\sigma,\sigma}}
\end{aligned}
\end{equation}
\begin{equation}
\label{est:integ_dd_prime}
\begin{aligned}
|z'|^{(\varepsilon)}_{0,\alpha-\sigma}+[\dd z']^{(2\varepsilon)}_{\gamma',\alpha-\sigma-\gamma'} 
\lesssim 
\left [T^\varepsilon+T^{2\varepsilon -\sigma(\frac{\gamma'}{\sigma}-1)_+}\right ]\|\zeta,\zeta'\|_{\cD^{\gamma,\theta\gamma'}_{X,\alpha-\sigma,\sigma}}\,.
\end{aligned}
\end{equation}
In particular, the rough integration is a well-defined, bounded operation from \\\( \left(\cD^{\gamma,\theta\gamma'}_{X,\alpha-\sigma,\sigma}(0,T;\varepsilon)\right )^d \) to \( \cD^{\gamma,\gamma'}_{X,\alpha,\sigma}(0,T;\varepsilon).\)
\end{corollary}
\begin{proof}
We assume without loss of generality that \[
 \|\zeta,\zeta'\|_{\cD^{\gamma,\theta\gamma'}_{X,\alpha-\sigma,\sigma}}\equiv |\zeta|_{0,\alpha-\sigma} +[\dd \zeta]^{(\varepsilon)}_{\gamma,\alpha-2\sigma}+ |\zeta'|^{(\varepsilon)}_{0,\alpha-2\sigma} + [\dd \zeta']^{(2\varepsilon)}_{\theta\gamma',\alpha-2\sigma-\theta\gamma'} +  [R^\zeta]^{(2\varepsilon)}_{\gamma+\theta\gamma',\alpha-2\sigma-\theta\gamma'}\le 1.
\]

Due to Theorem \ref{thm:integral} we obtain for the integral remainder \( \mathscr R^{S;\zeta} \) the estimate
\begin{equation}\label{integration_remainder}
	s^{2\varepsilon-\iota}|\mathscr R^{S,\zeta}_{t,s}|_{\alpha-\sigma + \kappa}
	\lesssim_{\rho_\gamma(\X),K,\tilde K}
	 |t-s|^{\gamma_0+\gamma -\sigma -\kappa-\iota}
\text{ for each }\kappa+\iota\in [-\sigma,\gamma_0+\gamma-\sigma),
\enskip\iota\in [0,2\varepsilon].
\end{equation}
Consequently, since
\[
z_t = \int_0^t \zeta_r \cdot d\X_r= S_{t,0}\zeta_t \cdot \dd X_{t,0} + S_{t,0}\zeta'_t :\xx_{t,0} +\mathscr R_{t,0}^{S,\zeta},
\]
we obtain setting \( \iota=\varepsilon\), \( \kappa=\gamma\) and recalling that $\varepsilon=(\gamma-\sigma)_+$
\[
\begin{aligned}
t^{\varepsilon}|z_t |_{\alpha+\varepsilon}
&\le 
t^{\varepsilon}\Big(| S_{t,0}|_{\alpha-\sigma\to \alpha+\varepsilon}|\zeta_t\cdot \dd X_{t,0}|_{\alpha-\sigma} + |S_{t,0}|_{\alpha-2\sigma\to\alpha+\varepsilon}|\zeta'_t: (\XX -\delta X^{\otimes 2})_{0,t}|_{\alpha-2\sigma} \\ &+|\mathscr R_{t,0}^{S,\zeta}|_{\alpha+\gamma-\sigma}\Big)
\\
&\lesssim_{K,\rho_\gamma(\X_0)}
t^{\gamma_0 - \sigma } + t^{2\gamma_0-2\sigma-\varepsilon} + t^{\gamma_0 - \gamma},
\end{aligned}
\]
which proves \eqref{est:integ} because \( \sigma + \varepsilon= \sigma\vee\gamma\).
As a consequence of Lemma \ref{lem:delta_S} and
\begin{multline*}
	|\dd^S z_{t,s}|_{\alpha-\varrho}
	\le |S_{t,s}|_{\alpha-\sigma\to \alpha-\varrho}|\zeta_t\cdot \dd X_{t,s}|_{\alpha-\sigma} \\
	+ |S_{t,s}|_{\alpha-2\sigma\to \alpha-\varrho}|\zeta'_t:(\XX_{t,s} -\delta X^{\otimes 2})_{t,s}|_{\alpha-2\sigma} 
	+|\mathscr R_{t,s}^{S,\zeta}|_{\alpha-\varrho} 
	\\
	\lesssim |t-s|^{\gamma_0 -(\sigma-\varrho)_+} + |t-s|^{2\gamma_0-(2\sigma-\varrho)_+}+ |t-s|^{\gamma_0-\gamma+\varrho} 
\end{multline*}
we also obtain \eqref{est:holder_rho}.

To show \eqref{est:integ_dd}, 
note that since \( \varepsilon=(\gamma-\sigma)_+\ge0\), we already infer from \eqref{est:integ} that \(|z|_{0,\alpha}\lesssim|z|_{0, \alpha +\varepsilon}\lesssim t^{\gamma_0 - \gamma } + t^{2\gamma_0-\gamma-\sigma}\), hence the first part of \eqref{est:integ_dd}. It remains to evaluate the increment \( \dd z_{t,s}=z_t-z_s.\) For that purpose, note first that
\begin{multline*}
	|\dd^S z_{t,s}|_{\alpha-\sigma}
	=|\int_s^t S_{t,r}\zeta_r\cdot d{\mathbf{X}}_r|_{\alpha-\sigma}
	\le |S_{t,s}|_{\alpha-\sigma\to \alpha-\sigma}|\zeta_t\cdot \dd X_{t,s}|_{\alpha-\sigma} \\
	+ |S_{t,s}|_{\alpha-2\sigma\to \alpha-\sigma}|\zeta'_t:(\XX_{t,s} -\delta X^{\otimes 2})_{t,s}|_{\alpha-2\sigma} 
	+|\mathscr R_{t,s}^{S,\zeta}|_{\alpha-\sigma} 
	\\
	\lesssim |t-s|^{\gamma_0} + |t-s|^{2\gamma_0-\sigma}+ s^{-\varepsilon}|t-s|^{\gamma}T^{\gamma_0-\sigma} , 
\end{multline*}
which shows boundedness of \( [\dd ^Sz]_{\gamma,\alpha-\sigma}^{(\varepsilon)}.\)
The fact that a similar bound is satisfied for the plain increment \( [\dd z]^{(\varepsilon)}_{\gamma,\alpha-\sigma}\)
follows by \eqref{est:integ} and Lemma \ref{lem:delta_S}.\\
The proof of \eqref{est:integ_R} is similar, noting that
\( R^{S,z}_{t,s}\equiv z_t-S_{t,s}z_s-\zeta_t\cdot\dd X_{t,s}\) satisfies 
\begin{equation}
\label{id_R_z}
\begin{aligned}
R^{S,z}_{t,s}&
= (S_{t,s}-I)\zeta_t\cdot\dd X_{t,s} + S_{t,s}\zeta_t':\xx_{t,s} + \mathscr R_{t,s}^{S;\zeta} .
\end{aligned}
\end{equation}
Using again \eqref{est:integ} and \eqref{integration_remainder} with \( \iota=0\) and \( \kappa=-\gamma'\ge-\sigma \) gives
\[
\begin{aligned}
	s^{2\varepsilon}|R^{S,z}_{t,s}|_{\alpha-\sigma-\gamma'} 
	&\lesssim
	 s^{2\varepsilon}\big(|S_{t,s}-I|_{\alpha-\sigma\to \alpha-\sigma-\gamma'}|\zeta_t\cdot \dd X_{t,s}|_{\alpha-\sigma}
	 \\&\quad \quad \quad \quad 
	 +|S_{t,s}|_{\alpha-2\sigma\to \alpha-\sigma-\gamma'}|\zeta'_t:(\XX_{t,s} -\delta X^{\otimes 2}_{t,s} )|_{\alpha-2\sigma}
	 +|t-s|^{\gamma_0+\gamma-\sigma+\gamma'}\big)
	 \\&\lesssim
	 s^{2\varepsilon}\big(|t-s|^{\gamma_0+\gamma'}
	 +|t-s|^{2\gamma_0 - \sigma +\gamma'}
	+|t-s|^{\gamma+\gamma'}T^{\gamma_0-\sigma}\big)
\end{aligned}
\]
as claimed.\\
Next, to evaluate the supremum of the Gubinelli derivative, we simply write
\[
t^{\varepsilon}|z'_t|_{\alpha-\sigma} = t^{\varepsilon}|\zeta_t|_{\alpha-\sigma} 
\leq T^\varepsilon\,.
\]
If \( (s,t)\in \Delta_2\), we also find
\begin{equation*}
s^{2\varepsilon}|\dd z'_{t,s}|_{\alpha-\sigma-\gamma'} 
=s^{2\varepsilon}|\dd \zeta_{t,s}|_{\alpha-\sigma - \frac{\gamma'}{\sigma}\sigma }
\lesssim 
\begin{cases}
	s^{\varepsilon} |t-s|^{\gamma} [\dd \zeta]^{(\varepsilon)}_{\gamma,\alpha-2\sigma}
	\quad \text{if}\quad \gamma'\ge\sigma
	\\
	s^{\varepsilon(2 - \frac{\gamma'}{\sigma})}|t-s|^{\frac{\gamma'}{\sigma}\gamma}
 [\dd \zeta]^{(\varepsilon\frac{\gamma'}{\sigma})}_{\frac{\gamma'}{\sigma}\gamma,\alpha-\sigma -\gamma'}
	\quad \text{otherwise}.
\end{cases}
\end{equation*}
The bound~\eqref{est:integ_dd_prime} follows by Lemma \ref{lem:interpolation}, observing that \( \varepsilon(2-\frac{\gamma'}{\sigma}) + \frac{\gamma\gamma'}{\sigma} =2\gamma-2\sigma +\gamma' \).\\
Finally, the fact that \( (z,z')\) is a controlled path in \( \cD^{\gamma,\gamma'}_{X,\alpha,\sigma}(0,T;\varepsilon)\) is a consequence of the estimates \eqref{est:integ_dd}, \eqref{est:integ_R} and \eqref{est:integ_dd_prime}.
\end{proof}

\subsection{Perturbation of the sewing map}
\label{ssec:perturb}
Given a scale \( (\cB_\beta)_{\beta\in [0,1]}\),
for $\beta,\beta'\in [0,1]$
we introduce the Banach space $\mathfrak K(\beta,\beta')$ of all $Q\in C(\dot \Delta_2; \cL(\cX_{\beta},\cX_{\beta'}))$ satisfying
\[
 \|Q\|_{(\beta,\beta')}:=\sup_{(s,t)\in\dot \Delta_2}
(t-s)^{\beta'-\beta}|Q_{t,s}|_{\cX_{\beta}\to\cX_{\beta'}}<\infty,
\]
where $\dot\Delta_2:=\Delta_2\setminus\{(t,t),t\in[0,T]\}.$
As is well-known, any evolution family \( S=S_{t,s}\) on \( \cX=\cB_0\), whose domain generator is constantly equal to \( \cB_1 \), defines an element of \( \mathfrak K(\beta,\beta') \) for any \( 0\le \beta\le\beta'\le 1. \)
The difference  \( Q=S^1-S^2\) of two such evolution families may also be evaluated in these spaces, in terms of the operator-norm of the difference of the generators (see the main perturbation results in Appendix \ref{app:perturbation}). 

We now aim to quantify the effect of a change of propagator on the corresponding sewing map, in order to investigate the rough convolution
\[
	\int_{s}^t ( S^{1}_{t,r} - S^2_{t,r}) \zeta_r \cdot d\X_r.
\]
when \( \X\in \mathscr C^{\gamma_0}(\R^d)\), \( \gamma_0\in (\frac13,\frac12).\)
The main result of this subsection reads as follows.
\begin{theorem}[Perturbation of the sewing map] 
	\label{thm:perturbation}
Fix \( \gamma_0,\gamma',\sigma,\beta,\kappa,\iota\) and \( \X \in \mathscr C^{\gamma_0}(\R^d)\) as in Theorem \ref{thm:integral}. Consider two evolution families \(S^1=\exp(\int L^1_rdr),S^2=\exp(\int L^2_rdr)\) on \( \cX=\cB_0\), both subject to \ref{P1}--\ref{P4star} with the same space \( \cX_1=\cB_1\) and the same constants \(M,K,\tilde K>0, \lambda\in \R \).
Then, the following uniform estimate holds for any \( 0< T\le1\), \( (t,s)\in \Delta_2([0,T])\) and \( (\zeta,\zeta')\in \mathcal D_{X,\beta,\sigma}^{\gamma,\gamma'}([0,T];\varepsilon)\):
\begin{equation}\label{perturbation:sewing}
s^{2\varepsilon-\iota}\Big| \mathscr R^{S^1;\zeta}_{t,s} - \mathscr R^{S^2;\zeta}_{t,s}\Big|_{\beta+\kappa}
\lesssim
\rho_{\gamma_0}(\X)(1\vee K\vee \tilde K) \Gamma_{\beta+\kappa}(L^1,L^2)\|\zeta,\zeta'\|_{\cD_{X,\beta,\sigma}^{\gamma,\gamma' }(\varepsilon)}|t-s|^{\gamma_0+\gamma-\sigma -\kappa-\iota},\,
\end{equation} 
where 
\begin{equation}
\label{nota:Gamma}
\Gamma_{\alpha}(L^1,L^2)=
\begin{cases}
|L^1-L^2|_{0,\cB_1\to\cB_0}\quad \hspace*{38 mm}\text{if } 0\le \alpha<1\\
|L^1-L^2|_{0,\cB_1\to\cB_0}+T^{\varrho}|L^1-L^2|_{\varrho,\cB_1\to\cB_0}\quad \text{if }\alpha=1\,.
\end{cases}
\end{equation}
Here we recall that the integral remainder is the two-parameter quantity \( \mathscr R^{S;\zeta}_{t,s}=\int\limits_s^t S_{t,r}\zeta_r\cdot d\X_r - S_{t,s}(\zeta_t\cdot\dd X_{t,s} +\zeta_t':\xx_{t,s}) \).
\end{theorem}
\begin{proof}
We assume without loss of generality that \( \|\zeta,\zeta'\|_{\cD_{X,\alpha,\sigma}^{\gamma,\gamma'}}\le 1\).
In order to evaluate $\mathscr R^{S^1;\zeta}-\mathscr R^{S^2;\zeta}$,
we rely on Lemma \ref{lem:dyadic}. As a consequence of Corollary \ref{thm:integral} and the existence of the sewing map, we have
 \[\begin{aligned}
\mathscr R^{S^1;\zeta}_{t,s}-\mathscr R^{S^2;\zeta}_{t,s}
&=\lim_{n\to \infty} \sum_{[u,v]\in \pi_n(s,t)}J_{v,u}
\end{aligned}
 \]
 where \( J_{v,u}=(S^1_{t,u}-S^2_{t,u})(\zeta_v\cdot\dd X_{v,u} + \zeta'_v:\xx_{v,u})\) and \( \pi_n(s,t)\) is the dyadic sequence as in \eqref{dyadic}.
Using Chen's relation, one observes the following identity for each \( s\le u\le m\le v\le t\):
\begin{equation}\label{algebraic_id}
\begin{aligned}
\dd J_{v,m,u}&= 
S^1_{t,u}(R^{\zeta}_{v,m}\cdot \dd X_{m,u} + \dd \zeta'_{v,m}:\xx_{m,u}) 
\\&\quad 
+ S^1_{t,u}(S^1_{v,m}-\id)(\zeta_v\cdot\dd X_{v,m} + \zeta_v':\xx_{v,m})
\\& \quad \quad 
- S^2_{t,u}(R^{\zeta}_{v,m}\cdot \dd X_{m,u} + \dd \zeta_{v,m}':\xx_{v,m}) 
\\&\quad \quad \quad 
- S^2_{t,u}(S^2_{v,m}-\id)(\zeta_v\cdot\dd X_{v,m} + \zeta'_v:\xx_{v,m})
\\&
=
(S^2-S^1)_{t,u} R^{\zeta}_{v,m}\cdot\dd X_{m,u}
+ (S^2-S^1)_{t,u} \dd \zeta'_{v,m}:\xx_{m,u}
\\
&\quad \quad \quad \quad 
+\big[S^1_{t,u}(S^1-S^2)_{v,m}+(S^1-S^2)_{t,u}(S^2_{v,m}-\id)\big]\zeta_v \cdot\dd X_{v,m} 
\\
&\quad \quad \quad \quad \quad \quad 
+\big[S^1_{t,u}(S^1-S^2)_{v,m}+(S^1-S^2)_{t,u}(S^2_{v,m}-\id)\big]\zeta'_v:\xx_{v,m} 
\\
&=\mathrm{I}+\mathrm{II}+\mathrm{III}+\mathrm{IV}.
\end{aligned}
\end{equation}
We now use \eqref{apriori_R_gamma} to estimate the first term as follows:
\[
\begin{aligned}
|\mathrm{I}|_{\beta+\kappa}
&\lesssim
|S^2_{t,u}-S^1_{t,u}|_{\beta-\sigma-\gamma'\to\beta+\kappa}|v-m|^{\gamma+\gamma'}|m-u|^{\gamma_0}m^{-2\varepsilon}[R^\zeta]^{(2\varepsilon)}_{\gamma+\gamma',\beta-\sigma-\gamma'}[X]_{\gamma_0}, 
\\
&\lesssim_{\rho_{\gamma_0}(\X)}
\|S^1-S^2\|_{(\beta-\sigma-\gamma',\beta+\kappa)}
|t-u|^{-\kappa-\sigma-\gamma'}|v-m|^{\gamma+\gamma'}|m-u|^{\gamma_0} m^{-2\varepsilon}\,.
\end{aligned}
\]
For the second term, we use \eqref{apriori_dd_prime}, which gives
\[
\begin{aligned}
|\mathrm{II}|_{\beta+\kappa}
&\lesssim
|S^2_{t,u}-S^1_{t,u}|_{\beta-\sigma-\gamma'\to\beta+\kappa}|v-m|^{\gamma'}|m-u|^{2\gamma_0}m^{-2\varepsilon}[\dd \zeta']^{(2\varepsilon)}_{\gamma',\beta-\sigma-\gamma'}[\xx]_{2\gamma_0}
\\
&
\lesssim_{\rho_{\gamma_0(\X)}}
\|S^2-S^1\|_{(\beta-\sigma-\gamma',\beta+\kappa)}|t-u|^{-\kappa-\sigma-\gamma'} |v-m|^{\gamma'}|m-u|^{2\gamma_0}m^{-2\varepsilon}\,.
\end{aligned}
\]
For the third term, we choose
 any \( \rho\in(1-\gamma,\beta)\) and regard that $|m-u|=|v-m|$ to obtain that
\[
\begin{aligned}
|\mathrm{III}|_{\beta+\kappa}
&\lesssim
\Big(|S^1_{t,u}|_{\beta-\rho\to\beta+\kappa}|S^1_{v,m}-S^2_{v,m}|_{\beta\to\beta-\rho}
\\
&\quad \quad \quad 
+|S^1_{t,u}-S^2_{t,u}|_{\beta-\rho\to\beta+\kappa}|S^2_{v,m}-\id|_{\beta\to\beta-\rho}
\Big)|m-u|^{\gamma_0}|\zeta|_{0,\beta}[X]_{\gamma_0}
\\
&
\lesssim_{\rho_{\gamma_0}(\X)}
\Big(\|S^1\|_{(\beta-\rho,\beta+\kappa)}\|S^1-S^2\|_{(\beta,\beta-\rho)}
\\
&\quad \quad 
+\|S^1-S^2\|_{(\beta-\rho,\beta+\kappa)}\|S^2-\id\|_{(\beta,\beta-\rho)}
\Big)|t-u|^{-\kappa-\rho}|m-u|^{\gamma_0+\rho}\,.
\end{aligned}
\]
Lastly, if \( \rho'\in (1-2\gamma+\sigma,\beta)\), then
\[
\begin{aligned}
|\mathrm{IV}|_{\beta+\kappa}
&\lesssim
\Big(|S^1_{t,u}|_{ \beta-\rho'\to\beta+\kappa}|S^1_{v,m}-S^2_{v,m}|_{\beta-\sigma\to\beta-\rho'}
\\
&\quad \quad 
+|S^1_{t,u}-S^2_{t,u}|_{ \beta-\rho'\to\beta+\kappa}|S^2_{v,m}-\id|_{\beta-\sigma\to\beta-\rho'}\Big)|m-u|^{2\gamma} |\zeta'|^{(\varepsilon)}_{0,\beta-\sigma}[\xx]_{2\gamma_0}
\\
&\lesssim_{\rho_{\gamma_0}(\X)}
\Big(\|S^1\|_{(\beta-\rho',\beta+\kappa)}\|S^1-S^2\|_{(\beta-\sigma,\beta-\rho')}
\\
&
\quad \quad 
+\|S^1-S^2\|_{(\beta-\rho',\beta+\kappa)}\|S^2-\id\|_{(\beta-\sigma,\beta-\rho')}
\Big)|t-u|^{-\kappa-\rho'}|m-u|^{2\gamma_0+\rho'-\sigma}m^{-\varepsilon}\,.
\end{aligned}
\]
To conclude, we note that Lemma \ref{lem:dyadic} is indeed applicable (by our choice of \( \rho,\rho' \) and by assumption on the parameters). It yields the bound
\[
s^{2\varepsilon-\iota}\Big| \mathscr R^{S^1;\zeta}_{t,s} - \mathscr R^{S^2;\zeta}_{t,s}\Big|_{\beta+\kappa}
\lesssim
\rho_{\gamma_0}(\X)(1\vee K\vee \tilde K) A|t-s|^{\gamma+\gamma_0 -\sigma-\kappa-\iota}\,,
\]
for the constant
\begin{equation}
	\label{constant_A}
	A=
	\|S^1-S^2\|_{(\beta-\sigma-\gamma',\beta+\kappa)}
	\vee\|S^1-S^2\|_{(\beta,\beta-\rho)}
	\vee\|S^1-S^2\|_{(\beta-\sigma,\beta-\rho')}
	\lesssim
\Gamma_{\beta+\kappa}(L^1,L^2)
\end{equation} 
by Lemma \ref{lem:K}.
This yields our conclusion.
\end{proof}

The next statement is the analogue of Theorem \ref{thm:integral} for the difference of two evolution families.

\begin{corollary}\label{c:perturbation}
	Let \( \gamma_0,\gamma,\gamma',\sigma,\alpha,\theta\), \( \X\in \mathscr C^{\gamma_0}(\R^d)\) be as in the hypotheses of Corollary \ref{cor:improvereg}, and pick \( \alpha\in (1-2\varepsilon-(\gamma_0-\gamma),1-\varepsilon]\) where \( \varepsilon=\gamma-\sigma\ge0. \)
	The map
	\begin{multline*}
	\left (\cD^{\gamma,\theta\gamma'}_{X,\alpha-\sigma,\sigma}([0,T];\varepsilon)\right )^d\to\cD^{\gamma,\gamma'}_{X,\alpha,\sigma}([0,T],\varepsilon) \,,
\\
(\zeta,\zeta')\mapsto (z,z'):=\left(\int\limits_0^{\cdot}(S^{1}_{\cdot,r} -S^2_{\cdot,r})\zeta_r \cdot d\X_r ,0\right)
	\end{multline*}
	is well-defined and bounded. It satisfies the following estimates (recall \eqref{nota:Gamma})
	\begin{equation}\label{est:perturb}
		|z|^{(\varepsilon)}_{0,\alpha +\varepsilon}\lesssim
		\Gamma_{\alpha+\varepsilon}(L^1, L^2) T^{\gamma_0-\gamma}\|\zeta,\zeta'\|_{\cD^{\gamma,\theta\gamma'}_{X,\alpha-\sigma,\sigma}}
	\end{equation}
	\begin{equation}
		\label{est:holder_rho_diff}
		[\dd z]_{\varrho;\alpha-\varrho} \lesssim \Gamma_{\alpha}(L^1, L^2)\left [T^{\gamma_0 -\varrho -(\sigma-\varrho)_+} + T^{2\gamma_0-\varrho-(2\sigma-\varrho)_+}+ T^{\gamma_0-\gamma}\right ] \|\zeta,\zeta'\|_{\cD^{\gamma,\theta\gamma'}_{X,\alpha-\sigma,\sigma}},
	\end{equation}
for any \(\varrho\in (0,\gamma_0)\) while
	\begin{equation}
		\label{est:perturb_dd}
		\begin{aligned}
			|z|_{0,\alpha}+\vee_{i=0,1}[\dd z]^{((i+1)\varepsilon)}_{\gamma+i\gamma',\alpha-\sigma-i\gamma'}
			\lesssim_{K,\tilde K,\rho_{\gamma_0}(\X)}
			\Gamma_{\alpha}(L^1,L^2)
			 T^{\gamma_0-\gamma}\|\zeta,\zeta'\|_{\cD^{\gamma,\theta\gamma'}_{X,\alpha-\sigma,\sigma}}.
		\end{aligned}
	\end{equation}
	In particular
	\begin{align}
	\|z,z'\|_{\cD^{\gamma,\gamma'}_{X,\alpha,\sigma}([0,T],\varepsilon)} \lesssim 
	\rho_{\gamma_0}(\X)\Gamma_\alpha(L^1,L^2)\|\zeta,\zeta'\|_{\cD^{\gamma,\theta\gamma'}_{X,\alpha,\sigma}([0,T],\varepsilon)}\,.
	\end{align}
\end{corollary}
\begin{proof}
We assume without loss of generality that \( \|\zeta,\zeta'\|_{\cD^{\gamma,\theta\gamma'}_{X,\alpha,\sigma}}\le 1.\)
Because \( z'\equiv0 \), we have that \( R^z=\dd z\), thus boundedness will follow from the estimates \eqref{est:perturb}-\eqref{est:perturb_dd}.
To show \eqref{est:perturb}, write
\[
z_t = (\mathscr R^{S^1,\zeta}-\mathscr R^{S^2,\zeta})_{t,0}
+(S^1-S^2)_{t,0}\zeta_t\cdot \dd X_{t,0} + (S^1-S^2)_{t,0}\zeta'_t:\xx_{t,0}\,.
\]
Then, Theorem \ref{thm:perturbation} with \( \beta=\alpha-\sigma\), \( \kappa=\gamma\) and \( \iota=\varepsilon\) yields
\begin{multline*}
	t^{\varepsilon}|z_t|_{\alpha+\gamma-\sigma} 
	\lesssim  \Gamma_{\alpha+\gamma-\sigma}(L^1,L^2)\,t^{\gamma_0-\gamma}
	\\
+\|S^1-S^2\|_{(\alpha-\sigma,\alpha+\gamma-\sigma)}t^{\gamma_0-\gamma} + \|S^1-S^2\|_{(\alpha-2\sigma,\alpha-\sigma+\gamma)}t^{2\gamma_0-\gamma-\sigma}
\end{multline*}
and the claimed estimate follows by Lemma \ref{lem:K}.

To estimate the plain increment \( \dd z\), we investigate first the reduced increment \( \dd ^Sz\) and then make use of Lemma \ref{lem:delta_S}. We have
	\begin{align*}
	\dd^S z_{t,s}
	& = (\mathscr R^{S^1,\zeta}_{t,s}-\mathscr R^{S^2,\zeta}_{t,s}) + (S^1_{t,s} -S^2_{t,s} )\zeta_t\cdot\dd X_{t,s} + (S^1_{t,s} -S^2_{t,s} )\zeta_t'\cdot\xx_{t,s} 
	\end{align*}
Let \( i\in \{0,1\}\).
	Using~\eqref{perturbation:sewing} for $T\leq 1$, we obtain for the first term
	\begin{align*}
	s^{\varepsilon}\Big|\mathscr R^{S^1,\zeta}_{t,s}-\mathscr R^{S^2,\zeta}_{t,s}\Big|_{\alpha-\sigma-i\gamma'}
\lesssim _{K,\tilde K,\rho_{\gamma_0}(\X)} \Gamma_{\alpha-\sigma-i\gamma'}(L^1,L^2)(t-s)^{\gamma_0 + \gamma -\sigma+i\gamma' - \varepsilon} ,
	\end{align*}
while for the second and third, we have from the definition of the spaces $\mathfrak K$:
\begin{multline*}
|(S^1-S^2)_{t,s}\dd X_{t,s}\cdot \zeta_t|_{\alpha-\sigma-i\gamma'} + |(S^1-S^2)\XX_{t,s}:\zeta'_t|_{\alpha-\sigma-i\gamma'}  
\\
\lesssim \rho_{\gamma_0}(\X)(
\|S^1-S^2\|_{(\alpha-\sigma,\alpha-\sigma-i\gamma')}|t-s|^{\gamma_0+i\gamma'}
+\|S^1-S^2\|_{(\alpha-2\sigma,\alpha-\sigma-i\gamma')}|t-s|^{2\gamma_0+i\gamma'-\sigma}
).
\end{multline*}
Part two of Lemma \ref{lem:delta_S} then shows the desired estimate.
Putting these estimates together and making use of Lemma \ref{lem:K} entails \eqref{est:perturb_dd}.
\end{proof}

\subsection{Composition of a controlled rough path with a regular enough nonlinearity}

Before closing this section, we need to collect some estimates on the composition of controlled path with a certain class of non-linearities.
As outlined in \cite{gerasimovics2020non}, when the scale considered is of the form \( (\cB_{\beta})_{\beta\in\R}\), the composition of a controlled path with any map
which is \( C^2_b\) from \(\cB_{\beta-2\gamma} \) to \(\cB_{\beta-2\gamma +\theta},\) for all \(\theta\ge0\), is again a controlled path.
In our setting, we need to be slightly more careful since by assumption the spatial regularity indices are only allowed to vary in \( J=[0,1].\)

\begin{notation} 
Given a scale \( (\cB_\beta)_{\beta\in [0,1]}\) subject to Assumption \ref{ass:intermediate},
if  \(F\colon \cB_\alpha\to \cB_\beta\) is a \( k\)-times continuously differentiable mapping, we make use of the following notation for
indices \( \beta'\in [0,1]\) and \( \alpha_i\in [0,1]\) with \( \max_{i=1\dots k}\alpha_i=\alpha\):
	\[
	|D^kF|_{(\alpha_1,\dots, \alpha_k)\to\beta'}:=\sup_{x\in \cB_{\alpha}} |D^kF(x)|_{(\alpha_1,\dots ,\alpha_k)\to\beta'}\in [0,\infty]\,,
	\]
where we recall the notation \( |D^kF(x)|_{(\alpha_1,\dots,\alpha_k)\to\beta'} =|D^kF(x)|_{\cB_{\alpha_1}\otimes\dots \otimes\cB_{\alpha_k}\to\cB_{\beta'}}\) for fixed \( x\in \cB_\alpha \).
\end{notation}
In the following, we consider the parameters
\( \frac13 <\gamma \le \gamma_0 <\frac12  \), 
\( \sigma\in [0,\gamma] \),
\( \alpha\in (1-\gamma_0-\gamma+2\sigma,1-\gamma+\sigma)\),
\( \gamma'\in(1-\gamma_0-\gamma, \alpha-\sigma)\), 
and we fix a path \( X\in C^{\gamma_0}([0,T];\R^d).\) We prove that the composition of a controlled rough path with a smooth nonlinear functions is a well-defined operation.
\begin{lemma} \label{lem:composition}
Pick any number \( \theta\in [0,1]\) such that
 \begin{equation}\label{theta_F}
\theta\in\left (\frac{1-\gamma_0-\gamma}{\gamma'},\frac{\alpha-2\sigma}{\gamma'}\right ),
\end{equation}
and consider a continuous and bounded mapping \(F\colon \cB_{\beta}\to\cB_{\beta-\sigma}\) for every \( \beta\in [\alpha-\sigma-\theta\gamma',\alpha].\)
Let \( \varepsilon= \gamma-\sigma\), take any
\((y,y') \in \cD^{\gamma,\gamma'}_{X,\alpha,\sigma}([0,T];\varepsilon)\),
 and introduce
\[
(\zeta_t,\zeta'_t) := (F(y_t),DF(y_t)\circ y'_t) , \quad t\in [0,T]. 
\]
	\begin{enumerate}[label=(\Roman*)]
		\item \label{part_I}
		Suppose that \( F \) is twice continuously differentiable from \(\cB_{\beta}\to\cB_{\beta-\sigma}\) for each \( \beta\in [\alpha-\sigma-\theta\gamma',\alpha],\) and that
		\( \|F\|_{(2)}:=\sup_{\beta\in[\alpha-\sigma-\theta\gamma',\alpha]}|F|_{C^2_b(\cB_{\beta},\cB_{\beta-\sigma})}<\infty.\)
		Then, \( (\zeta,\zeta')\) belongs to  \(\cD^{\gamma,\theta\gamma'}_{X,\alpha-\sigma,\sigma}\), and moreover
\begin{equation}
	\label{est:remainder_F}
		\begin{aligned}
			[R^\zeta]^{(2\varepsilon)}_{\gamma+\theta\gamma',\alpha-2\sigma-\theta\gamma'} 
			&\le |D^2F|_{(\alpha-\sigma,\alpha-\sigma)\to\alpha-2\sigma-\theta\gamma'}([\dd y]^{(\varepsilon)}_{\gamma,\alpha-\sigma})^2
			\\&\quad \quad 
			+|DF|_{(\alpha-\sigma-\theta\gamma')\to\alpha-2\sigma-\theta\gamma'}[R^y]^{(\varepsilon(1+\theta))}_{\gamma + \theta\gamma',\alpha-\sigma-\theta\gamma'}
		\end{aligned}
\end{equation}
\begin{equation}
	\label{est:derivative_F}
		\begin{aligned}
			[\dd \zeta']^{(2\varepsilon)}_{\theta\gamma',\alpha-2\sigma -\theta\gamma'}
			&\le |DF|_{(\alpha-\sigma-\theta\gamma')\to\alpha-2\sigma -\theta\gamma'}
			[\dd y']^{(\varepsilon(1+\theta))}_{\theta\gamma',\alpha-\sigma-\theta\gamma'} 
			\\&\quad \quad 
			+|D^2F|_{(\alpha-\sigma,\alpha-\sigma)\to\alpha-2\sigma-\theta\gamma'}
			[\dd y]^{(\varepsilon)}_{\gamma,\alpha-\sigma}|y'|^{(\varepsilon)}_{0,\alpha-\sigma}
			\quad .
		\end{aligned}
\end{equation}
		Consequently
		\begin{equation}\label{consequently_1}
		\|\zeta,\zeta'\|_{\cD^{\gamma,\theta\gamma'}_{X,\alpha-\sigma,\sigma}(\varepsilon)}
		\lesssim \|F\|_{(2)}(1+\|y,y'\|_{\cD^{\gamma,\gamma'}_{X,\alpha,\sigma}})\|y,y'\|_{\cD^{\gamma,\gamma'}_{X,\alpha,\sigma}(\varepsilon)}\,.
		\end{equation}
		\item\label{part_II}
		Suppose that \(F\) is three times continuously differentiable from \(\cB_{\beta}\to\cB_{\beta-\sigma}\) for each \( \beta\in [\alpha-\sigma-\theta\gamma',\alpha]\) and that \( \|F\|_{(3)}:=\sup_{\beta\in[\alpha-\sigma-\theta\gamma',\alpha]}|F|_{C^3_b(\cB_\beta,\cB_{\beta-\sigma})}<\infty.\)
		If \( (\bar \zeta,\bar \zeta'):=(F(\bar y),DF(\bar y)\circ \bar y')\) for another such \( (\bar y,\bar y')\in\mathcal D_{X,\alpha,\sigma}^{\gamma ,\gamma'},\) then
\begin{equation}
	\label{est:remainder_F_diff}
		\begin{aligned}
			&[R^\zeta-R^{\bar \zeta}]^{(2\varepsilon)}_{\gamma+\theta\gamma',\alpha-2\sigma-\theta\gamma'} 
			\le |D^3F|_{(\alpha,\alpha-\sigma,\alpha-\sigma)\to\alpha-2\sigma-\theta\gamma'}
			|y-\bar y|_{0,\alpha}([\dd y]^{(\varepsilon)}_{\gamma,\alpha-\sigma})^2
			\\&\quad\quad \quad
			 +|D^2F|_{(\alpha-\sigma,\alpha-\sigma)\to\alpha-2\sigma-\theta\gamma'}[\dd y-\dd \bar y]^{(\varepsilon)}_{\gamma,\alpha-\sigma}[\dd \bar y]^{(\varepsilon)}_{\gamma,\alpha-\sigma}
			\\&\quad \quad\quad \quad 
			+|D^2F|_{(\alpha-\sigma-\theta\gamma',\alpha)\to\alpha-2\sigma-\theta\gamma'}
			|y-\bar y|_{0,\alpha}[R^y]^{(2\varepsilon)}_{\gamma+\theta\gamma',\alpha-\sigma-\theta\gamma'}
			\\
			&\quad \quad \quad\quad \quad  
			+|DF|_{(\alpha-\sigma-\theta\gamma')\to\alpha-2\sigma-\theta\gamma'}
			[R^y-R^{\bar y}]^{(2\varepsilon)}_{\gamma+\theta\gamma',\alpha-\sigma-\theta\gamma'}
			\quad ,
		\end{aligned}
\end{equation}
\begin{equation}
		\label{est:derivative_F_diff}
			\begin{aligned}
			&[\dd \zeta'-\dd\bar \zeta']_{\theta\gamma',\alpha-2\sigma-\theta\gamma'}
			\le
			|D^3F|_{(\alpha,\alpha-\sigma,\alpha-\sigma)\to\alpha-2\sigma-\theta\gamma'}
			|y-\bar y|_{0,\alpha}[\dd y]_{\gamma,\alpha-\sigma} |y'|_{0,\alpha-\sigma} 
			\\
			&\quad \quad \quad 
			+|D^2F|_{(\alpha-\sigma,\alpha-\sigma)\to\alpha-2\sigma-\theta\gamma'}
			([\dd y-\dd \bar y]_{\gamma,\alpha -\sigma}|y'|^{(\varepsilon)}_{0,\alpha-\sigma} + [\dd\bar y]^{(\varepsilon)}_{\gamma,\alpha-\sigma}|y'-\bar y'|^{(\varepsilon)}_{0,\alpha-\sigma})
			\\&\quad \quad \quad \quad 
			+|D^2F|_{(\alpha,\alpha-\sigma-\theta\gamma')\to\alpha-2\sigma-\theta\gamma'}
			 |y-\bar y|_{0,\alpha}[\dd y']^{(\varepsilon(1+\theta))}_{\theta\gamma,\alpha-\sigma-\theta\gamma'}
			\\&\quad \quad \quad \quad \quad 
			+|DF|_{(\alpha-\sigma-\theta\gamma')\to\alpha-2\sigma-\theta\gamma'}
			[\dd y'-\dd\bar y']^{(\varepsilon(1+\theta))}_{\theta\gamma',\alpha-\sigma-\theta\gamma'}
		\end{aligned}
\end{equation}
	and consequently
	\begin{align}\label{consequently_2}
	&\|\zeta-\bar\zeta,\zeta'-\bar\zeta'\|_{\cD^{\gamma,\theta\gamma'}_{X,\alpha-\sigma,\sigma}(\varepsilon)} \\
	&\leq \|F\|_{(3)}(1+\|y,y'\|_{\cD^{\gamma,\gamma'}_{X,\alpha,\sigma}(\varepsilon)}+\|\bar y,\bar y'\|_{\cD^{\gamma,\gamma'}_{X,\alpha,\sigma}(\varepsilon)})\|y-\bar y,y'-\bar y'\|_{\cD^{\gamma,\gamma'}_{X,\alpha,\sigma}(\varepsilon)}.
	\end{align}
	\end{enumerate}
\end{lemma} 
\begin{proof}[Proof of Lemma \ref{lem:composition}]
	This is a slight refinement of  \cite[Lemma 4.7]{gerasimovics2020non}, hence we only sketch the argument.
	
	\textit{Part \ref{part_I}.}
	That \( (\zeta,\zeta') \) is a controlled path is a clear consequence of the estimates.
	To estimate the remainder, we note the Taylor-type identity
	\[
	R^\zeta_{t,s}=(\iint_{0\le \tau' \le\tau\le 1}D^2F(y_s + \tau'\dd y_{t,s})d\tau' d\tau )\circ (\dd y_{t,s})^{\otimes 2}
	+ DF(y_s)\circ R^{y}_{t,s}
	\]
and \eqref{est:remainder_F} follows immediately.
As for the Gubinelli derivative, we rely on the decomposition \( \dd \zeta'_{t,s} =DF(y_s)\circ \dd y'_{t,s} + (DF(y_t)-DF(y_s))\circ y_t'\), from which \eqref{est:derivative_F} follows.
To conclude that \eqref{consequently_1} holds, it is sufficient to combine Lemma \ref{lem:interpolation} with the obvious embedding
\( \cB_{\alpha-\sigma}^{\otimes2}\hookrightarrow \cB_{\alpha-\sigma-\theta\gamma'}^{\otimes2} \)
and use the twice continuously Fr\'echet-differentiability of \( F\).\smallskip
	
	\textit{Part \ref{part_II}.}
	For the first bound, we observe the following algebraic identity:
	\begin{multline}
		\label{id:R_zeta}
		R^{\zeta}_{t,s}-R^{\bar \zeta}_{t,s}
		=
		\iint_{0\le \tau'\le\tau\le 1}D^2F(y_s + \tau'\dd y_{t,s})d\tau'd\tau\circ (\dd y_{t,s})^{\otimes2}
		\\
		-\iint_{0\le \tau'\le\tau\le 1}D^2F(\bar y_s + \tau'\dd \bar y_{t,s})d\tau'd\tau\circ (\dd \bar y_{t,s})^{\otimes2}
		\\
		+(DF(y_s)-DF(\bar y_s))\circ R^y_{t,s}
		+DF(\bar y_s)\circ (R^y_{t,s}-R^{\bar y}_{t,s}),
	\end{multline}
	from which \eqref{est:remainder_F_diff} follows.
	Similarly
\begin{equation}\label{id:zeta_prime}
	\begin{aligned}
		\dd \zeta'_{t,s} - \dd\bar \zeta'_{t,s}
		&=
		\int_0^1(D^2F(y_s+\tau\dd y_{t,s})-D^2F(\bar y_s+\tau\dd \bar y_{t,s})d\tau)\circ (\dd y_{t,s}\otimes y'_t)
		\\&\quad 
		+ \int _0^1D^2F(\bar y_s+\tau\dd \bar y_{t,s})d\tau \circ [(\dd y_{t,s}-\dd \bar y_{t,s})\otimes y'_t + \dd \bar y_{t,s}\otimes (y'_t-\bar y'_t)]
		\\&\quad 
		+(DF(y_s)-DF(\bar y_s))\circ \dd y'_{t,s} + DF(\bar y_s)\circ (\dd y_{t,s}'-\dd\bar y'_{t,s})\,,
	\end{aligned}
\end{equation}
	which yields \eqref{est:derivative_F_diff} after evaluation in \(\cB_{\alpha-2\sigma-\theta\gamma'} .\)
	As for \ref{part_I}, we infer \eqref{consequently_2} from Lemma \ref{lem:interpolation} and our hypotheses on \( F.\)
\end{proof}

\begin{remark}
	\label{rem:other_RP}
The identities \eqref{id:R_zeta} and \eqref{id:zeta_prime} are still meaningful if the pair \( (\bar y,\bar y')\) is replaced by any element of \( \cD_{\bar X,\alpha,\sigma}^{\gamma,\gamma'}\), for another such rough path \( \bar X\in \mathscr C^{\gamma_0}(\R^d).\) 
In this case we obtain similar estimates as in \eqref{est:remainder_F_diff} through \eqref{consequently_2}.
\end{remark}

\section{Solution to subcritical quasilinear equations}
\label{sec:evolution}

In this section we work in the setting of Corollary \ref{cor:improvereg} and Lemma \ref{lem:composition}, in the sense that we are given parameters
\begin{equation}
\frac13 <\gamma < \gamma_0 <\frac12 ,
\quad 
\sigma\in [0,\gamma],
\quad 
\alpha\in (1-\gamma_0-\gamma+2\sigma,1-\gamma+\sigma],
\quad 
\gamma'\in(1-\gamma_0-\gamma, \alpha-\sigma)
\end{equation}
and we let \( \X\in C^{\gamma_0}([0,T];\R^d).\)
We rely on the controlled rough paths framework to prove existence and uniqueness for the rough PDE
	\begin{equation} \label{e:rpde2}
	du_t = (L_t(u_t)u_t +N_t(u_t))dt + \sum\nolimits_{i=1}^dF^i(u_t)d\X^i_t\quad \text{and}\quad  u_0 = x \in \cB_\alpha
	\end{equation}
for a choice of rough non-linearity \( F\colon \cB_\alpha\to \cB_{\alpha-\sigma}\). Similarly, \( N\colon[0,T]\times \cB_{\alpha}\to\cB_{\alpha-\delta}\) is a given drift non-linearity for some \( \delta<1\).\\
We recall that $V$ is the open subset of $\cB_{\eta}$ as introduced in Assumption \ref{ass:L_y}, where $\eta<\alpha$ is fixed and $V_\alpha:= V\cap \cB_\alpha$.
\begin{theorem}[Solutions of quasilinear RPDEs with bounded coefficients] \label{thm:RPDE}
Let \( \gamma_0,\gamma,\gamma'\), \(\sigma,\alpha,\X\) as above with \( \gamma'\le\gamma \) and \( \alpha<1 \).
Let moreover Assumption \ref{ass:L_y} hold true for \( L(\cdot) \), for some \( \eta \in[0,\alpha)\) with \( V =V_{y_0}=\cB_\eta\)~ for all \( y_0\in \cB_\eta \).
Fix a number \( \beta_0\in(\sigma, \alpha-\sigma-1+\gamma+\gamma_0) \) and a bounded, three-times continuously differentiable mapping \( F\colon \cB_\beta\to \cB_{\beta-\sigma} \), \( \beta\in [\beta_0,\alpha]\), such that all the constants appearing in Lemma \ref{lem:composition} are finite.
 In addition, fix \( \delta \in [0,\alpha)\) and consider for each \( t\in [0,T]\) a bounded and \( C_t\)-Lipschitz map
 \( N_t\colon \cB_{\beta}\to \cB_{\beta-\delta} \) for \( \beta\in [\alpha,\alpha+\gamma-\sigma]\) with
  \( \|N\|_1=\sup_{t\in [0,T]}C_t\vee |N_t|_{C^0_b}<\infty.\)
 
	For every $x\in\cB_\alpha$
	there exists a unique continuous path \( u\colon [0,T]\to \cB_{\alpha}\) with \( u\in \cC^{0,\gamma-\sigma}(\cB_{\alpha+\gamma-\sigma})\) such that \((u,F(u))\) belongs to  \(\cD^{\gamma,\gamma'}_{X,\alpha ,\sigma}([0,T];\gamma-\sigma)\) and
	\begin{equation} \label{e:rpde3}
		u_t = S^u_{t,0}x + \int_0^t S^u_{t,r}N_r(u_r)dr+ \int_0^t S^u_{t,r}F(u_r)\cdot d\X_r\,,
		\quad \text{for all }t\in [0,T]\,,
	\end{equation}
	where $S^u=\exp(\int L_r(u_r)dr)$ is the evolution family generated by $t\mapsto L_t(u_t)$, recall Remark \ref{rem:evolution}.
\end{theorem}

The next subsections are devoted to the proofs of Theorem~\ref{thm:RPDE} and Theorem~\ref{thm:main}. We begin with Theorem~\ref{thm:RPDE} and then drop the boundedness assumptions on the coefficients in Subsection~\ref{ssec:cast_aside_bd}.

\subsection{Proof of the main results}
\label{ssec:prf_thm_RPDE}

First of all we quantify the effect of the initial condition \( x\in\cB_{\alpha}\) in the fixed point argument and state the following result.
\begin{lemma}\label{lem:weight}
	\begin{itemize}
	\item [1)]
	Let \( S=S_{t,s}=\exp(\int_s^tL_rdr)\) for \( L \) as in Assumption \ref{ass:L_t}, let \( x\in \cB_{\alpha}\) and set \( (f^x_t,f^{x,\prime}_t):=(S_{t,0}x,0).\)
	Then \( f^x\) belongs to \( \cD_{X,\alpha,\sigma}^{\gamma,\gamma'}(0,T;\gamma-\sigma)\)
	and 
	\[
	\|f^x,0\|_{\cD^{\gamma,\gamma'}_{X,\alpha,\sigma}([0,T];\gamma-\sigma)}\le K(1+2\tilde K)|x|_{\alpha}.
	\]
	\item [2)] Letting \( \bar f^x_t=\bar S_{t,0}x \), for another evolution family \( \bar{S}=\exp(\int \bar L_rdr) \), we have the stability estimate (recall \eqref{nota:Gamma})
	\[
	\|f^x-\bar f^x,0\|_{\cD^{\gamma,\gamma'}_{X,\alpha,\sigma}([0,T];\gamma-\sigma)}\le C\Gamma_\alpha(L_t,\bar L_t)\,.
\]
	\end{itemize}
\end{lemma}

\begin{proof}
The first statement is obvious from the properties of $(y,y')$ combined with the fact that $y_t = S_{t,0} y_t + (\id - S_{t,0})y_t$
and the standard estimates
 $|S_{t,0} y_t|_{\alpha} \leq |S_{t,0}|_{\alpha+\varepsilon\to\alpha}|y_t|_{\alpha+\varepsilon}\lesssim_K T^\varepsilon |y|_{0,\alpha+\varepsilon}$.\\
Concerning the second statement, let \( f=f^x\) and \( \varepsilon=\gamma-\sigma\).
Obviously $|f_t|_{\alpha} \leq K |x|_\alpha$ and
	\[
	|f_t|_{\alpha+\varepsilon}\le K t^{-\varepsilon}|x|_{\alpha},
	\]
	thus \( t^\varepsilon|f|_{\alpha+\varepsilon}\le K|x|_{\alpha}<\infty.\)\\
	Moreover, if \( i=0,1\) we have
	\begin{align*}
	|\dd f_{t,s}|_{\alpha-\sigma-i\gamma'} =
	|(S_{t,s}-I)S_{s,0}x|_{\alpha-\sigma-i\gamma'}
&	\lesssim  |t-s|^{i\gamma'+\gamma}\|S-I\|_{(\alpha+\varepsilon,\alpha-\sigma-\gamma')}|S_{s,0}x|_{\alpha+\varepsilon}\\
&	\lesssim K\tilde K|t-s|^{i\gamma'+\gamma}s^{-\varepsilon}|x|_{\alpha},
	\end{align*}
	which proves the second claim.\\
The third assertion is similar, noting the identity \( \dd(f^x-\bar f^x)_{t,s} = (S_{t,s}-\bar S_{t,s})S_{0,s}x + (\bar S_{t,s}-I)(S_{s,0}-\bar S_{s,0})x\). Details are omitted.
\end{proof}

Now, given an evolution family \( S \) and a controlled path \( (y,y')\in\cD^{\gamma,\gamma'}_{X,\alpha,\sigma} \), we introduce the map
\begin{equation}\label{soln_map}
\psi_{T;N,F}(S;y, y')_t :=\Big(S_{t,0}x+\int_0^tS_{t,s}N_s(y_s)ds+\int_0^tS_{t,s}F(y_s)\cdot d\X_s, F(y_t)\Big)\,.
\end{equation}
The proof is based on a fixed point argument for the map \( (y,y')\mapsto \psi_{T;N,F}(S^y;y,y') ,\) where \( S^y_{t,s}=\exp(\int_s^tL_r(y_r)dr)\).
We assume without loss of generality that \( \varepsilon:=\gamma-\sigma>0\) (otherwise introduce \( \tilde \gamma:=\frac{\gamma+\gamma_0}{2}\) and observe that the corresponding statement implies the conclusion). 
We further fix two parameters
\( \gamma '\in(1-2\gamma,\sigma]\), \( \varrho \in (0,(\alpha-\eta)\wedge \delta) \) and a constant
\( R\ge K(1+2\tilde K)|x|_\alpha + |F|_{C^0_b}\).
Let us now introduce
\begin{multline*}
	B_T(x)= \Big\{(y,y') \in \cD^{\gamma,\gamma'}_{X,\alpha ,\sigma}(0,T), \enskip \text{with}\enskip y_0 = x,\; y'_0 = F(x),\,\forall t\in [0,T], 
	\text{ such that:}
	\\
	\text{\textbullet }\quad  y_t\in V_\alpha\quad \text{(resp.\ }y_t\in V_{\alpha+\varepsilon}\text{)},\quad \,\forall t\in [0,T]\quad\text{(resp.\ }\forall t\in (0,T] \text{);}
		\\
	\text{\textbullet }\quad |y|_{0,\alpha+\varepsilon}^{(\varepsilon)} + [\dd y]_{\varrho,\alpha-\varrho}
\le R +1
	\\
	\text{\textbullet } \quad  \|y,y'\|_{\cD_{X,\alpha,\sigma}^{\gamma,\gamma'}(\varepsilon)}\equiv|y|_{0,\alpha} +[\dd y]^{(\varepsilon)}_{\gamma,\alpha-\sigma} + |y'|^{(\varepsilon)}_{0,\alpha-\sigma}+[\dd y']^{(2\varepsilon)}_{\gamma ',\alpha-\sigma-\gamma'} + [R^y]^{(2\varepsilon)}_{\gamma+\gamma ',\alpha -\sigma-\gamma'} 
	\leq R + 1 \Big\}\,.
\end{multline*}
It is clear that \( B_T(x)\neq \emptyset\) since Lemma \ref{lem:weight} asserts in particular that \( t\mapsto(f^x_t,0)=(S^x_{t,0}x,0)\) satisfies the needed requirements. Moreover, letting \( \nn{y,y'}:=|y|^{(\varepsilon)}_{0,\alpha+\varepsilon}+[\dd y]_{\varrho,\alpha-\varrho}+\|y,y'\|_{\cD_{X,\alpha,\sigma}^{\gamma,\gamma'}(\varepsilon)}\), it is plain to check that
 \( (B_T(x),\nn{\cdot})\) is a complete space (we omit the proof).

	\subsubsection{Invariance of $B_T(x)$}
	
 We assume without loss of generality that \( T\le 1.\)
		Let $(y,y') \in B_T(x)$, denote  \( \psi:=\psi_{T;N,F}\) and write
		\[\begin{aligned}
		 \psi (S^y;y,y')_t
		&=\Big(S^y_{t,0}x + \int_0^t S^y_{t,r} N_r(y_r)dr + \int_0^t S^y_{t,r}F(y_r)\cdot d\X_r\,,\, F(y_t)\Big)
		\\&
		=: (f^x_t,0) + (D_t,0) + (z_t, z'_t)
		\end{aligned}
		\]
		for \( f^x:=(S_{t,0}^yx,0)\) and \( (z ,z '):=\Big(\int_0^t S^y_{t,r}\zeta_r \cdot d\X_r\,,\,\zeta_t \Big)\), where \( (\zeta,\zeta')=(F(y_t), DF(y_t)\circ y'_t) \).\\
			We start by estimating the drift term \( D_t\). Using \eqref{smoothing_S} we obtain 
		that
		\begin{align*}
		t^{\varepsilon}|D_t|_{\alpha+\varepsilon}
		=\Big|\int_0^tS^y_{t,r}N_r(y_r)dr\Big|_{\alpha+\gamma-\sigma}
	&	\lesssim \|N\|_{1}\|S^y\|_{(\alpha+\varepsilon-\delta,\alpha+\varepsilon)}(1+R)t^{\varepsilon}\int _0^t(t-r)^{-\delta }r^{-\varepsilon}dr\\
	&	\lesssim 
		T^{1-\delta}\,.
		\end{align*}	
		As a further consequence, Lemma \ref{lem:delta_S} is applicable and we may now focus on the evaluation of reduced increment \( \dd^{S^y} D_{t,s}=D_t-S^y_{t,s}D_s\). For \( i=0,1\), we have
		\[
		\begin{aligned}
		|\dd^{S^y}D_{t,s}|_{\alpha-\sigma-i\gamma'}=\Big|\int _s ^tS^y_{t,r}N_r(y_r) dr\Big|_{\alpha-\sigma-i\gamma'}
		&\le 
		\int_s^t |S^y_{t,r}|_{\alpha-\delta\to \alpha-\sigma-i\gamma'} |N_r(y_r)|_{\alpha-\delta}dr
		\\&
		\lesssim _{K,\|N\|_{1}} (t-s)^{1-(\delta-\sigma-i\gamma')_+}.
		\end{aligned}
		\]
	Observing similarly that \( |\dd^{S^y}D_{t,s}|_{\alpha-\varrho}\lesssim (t-s)^{1+\varrho-\delta} \), we therefore	
	obtain from Lemma \ref{lem:delta_S} that
\begin{multline}
\label{drift_term}
\nn{D,0}=|D|_{0,\alpha+\varepsilon}^{(\varepsilon)} + [\dd D]_{\varrho,\alpha-\varrho}+\|(D,0)\|_{\mathcal D_{X,\alpha ,\sigma}^{\gamma,\gamma '}(\varepsilon)}
\\
\lesssim (|D|^{(\varepsilon)}_{0,\alpha+\varepsilon} + T^\varepsilon\vee_{i=0,1}[\dd^{S} D]_{\gamma+i\gamma',\alpha-\sigma-i\gamma'})
\\
\lesssim 
_{\|N\|_{1},K,\tilde K,|x|_{\alpha}} T^{\varepsilon}+T^{1 - (\delta -\sigma-\gamma')_+ +\varepsilon } + T^{1-\delta}\,.
\end{multline}
Next, pick any \( \theta\in(\frac{1-\gamma_0-\gamma}{\gamma'},\frac{\alpha-\sigma-\beta_0}{\gamma'}] \) which is always possible from the upper bound on \( \beta_0\), and note that for such choice we have \( [1-\sigma - \theta\gamma',\alpha]\subset [\beta_0,\alpha] \). Applying Corollary \ref{cor:improvereg}, Lemma \ref{lem:composition} and using our hypotheses on \( F \), we infer that
\begin{multline*}
\nn{z,z'}=|z|^{(\varepsilon)}_{0,\alpha+\varepsilon}+[\dd z]_{\varrho,\alpha-\varrho}+\|z,z'\|_{\cD_{X,\alpha,\sigma}^{\gamma,\gamma'}(\varepsilon)}
\\
\lesssim _{\rho_\gamma(\X),K,\tilde K} 
[T^{\gamma_0-\gamma} + T^{\varepsilon} + T^{2\varepsilon-\sigma(\frac{\gamma'}{\sigma}-1)_+}] \|\zeta,\zeta'\|_{\cD^{\gamma,\theta\gamma'}_{X,\alpha-\sigma,\sigma}(\varepsilon)} 
\\
\leq C_0(\rho_\gamma(\X),K,\tilde K,\|F\|_{(2)})
[T^{\gamma_0-\gamma} + T^{\varepsilon} + T^{2\varepsilon-\sigma(\frac{\gamma'}{\sigma}-1)_+}] 
\|y,y'\|_{\cD^{\gamma,\gamma'}_{X,\alpha,\sigma}(\varepsilon)} ,
\end{multline*}
where the constant \( \|F\|_{(2)}\) is the one appearing in Lemma \ref{lem:composition}-\ref{part_I}. 
Putting all these estimates together, we observe thanks to Lemma \ref{lem:weight}, triangle inequality and the fact that \( \gamma'\le\gamma \) implies \( 2\varepsilon-\sigma(\frac{\gamma'}{\sigma}-1)_+ \ge \varepsilon\)
		\begin{equation}
			\begin{aligned}
		\label{estim:z}
	\nn{\psi (S^y;y,y'),\psi '(S^y;y,y')}
	    &\le \nn{f^x,0}
	    + \nn{D_t,0}
	    + \nn{z,z'}
		\\&
		\le K(1+2\tilde K)|x|_{\alpha} + C_1 T^{\varkappa}\,,
			\end{aligned}
		\end{equation}
	for a constant \( C_1=C_1(\|F\|_{(2)},K,\tilde K,\rho_\gamma (\X),|x|_{\alpha})>0\) and where
\( \varkappa=\varepsilon\wedge (\gamma_0-\gamma)\wedge (1-\delta).\)
Choosing $T$ small enough shows the claimed stability.

\subsubsection{Contraction property}
We now choose another pair \( (\bar y,\bar y')\in B_T(x)\) and set \( S:=S^y\), \( \bar S:=S^{\bar y}\), \( \bar f^x:=\bar S_{\cdot,0}x\), 
\( \bar D_t= \int _{0}^t\bar S_{t,r}N_r(y_r)dr\) and \( \bar z_t= \int _0^t\bar S_{t,r}F(y_r)\cdot d\mathbf{X}_r\).
By the definition of \( \psi,\) we observe
\begin{multline}
\label{id_tilde}
\psi (S ;y ,y^{\prime})_t - \psi (\bar S ;\bar y ,\bar y ^{\prime})_t
=\psi (\bar S ;y ,y')_t- \psi (\bar S ;\bar y ,\bar y ^{\prime})_t
\\
+\Big(S_{t,0}x-\bar S_{t,0}x + \int_0^t(S _{t,r}-\bar S _{t,r})N_r(y _r)dr + \int_0^t(S _{t,r}-\bar S _{t,r})F(y _r)\cdot d\X_r\,,\, 0\Big)
\\
=\psi (\bar S ;y ,y')_t- \psi (\bar S ;\bar y ,\bar y ^{\prime})_t
+ (f^x_t-\bar f^x_t,0)
+(D_t-\bar D_t ,0)+ (z _t-\bar z_t,0)\,.
\end{multline}
Repeating the same steps as for the stability property (using Lemma \ref{lem:composition}-(II)), we obtain the following estimate for the first two terms
\begin{multline}\label{est:GHN_diff}
\nn{\psi (\bar S ;y ,y')_t- \psi (\bar S ;\bar y ,\bar y ^{\prime})_t}
\\
=\nn{\int_0^t\bar S _{t,r}[N_r(y _r)-N_r(\bar y _r)]dr + \int_0^t\bar S _{t,r}[F(y _r)-F(\bar y _r)]\cdot d\X_r	\,, \, F(y)-F(\bar y)}
\\
\le C_2T^{\varkappa}\|y-\bar y,y'-\bar y'\|_{\cD^{\gamma,\gamma'}_{X,\alpha,\sigma}(\varepsilon)}
\end{multline}
for another constant \( C_2=C_2(\|F\|_{(3)},K,\tilde K,\|N\|_{1}) >0\).\\
To estimate the drift term in \eqref{id_tilde}, we first evaluate its norm in \( \cC^0(\cB_{\alpha+\gamma-\sigma})\).
We have
\begin{multline}\label{improve_D_diff}
t^{\varepsilon}|D_t-\bar D_t|_{\alpha+\varepsilon}
= t^{\varepsilon}|\int_0^t(S-\bar S)_{t,r}N_r(y_r)|_{\alpha+\varepsilon}dr
\\
\lesssim  \|N\|_{1}\|S-\bar S\|_{(\alpha+\varepsilon,\alpha+\varepsilon-\delta )}t^{\varepsilon}\int_0^t(t-r)^{-\delta}r^{-\varepsilon}dr
\lesssim T^{1-\delta}.
\end{multline}
As for the reduced increment, we note that
\begin{align*}
\dd^S D_{t,s} -\dd^{\bar S}\bar D_{t,s}
& = \int\limits_{s}^{t} [S _{t,r} -\bar S _{t,r}] N_r( y _r)dr ,
\end{align*}
which needs to be evaluated in \( \cC^\gamma_2(B_{\alpha-\sigma})\cap\cC_2^{\gamma+\gamma'}(B_{\alpha-\sigma-\gamma'})\).
We observe for \( i=0,1 \)
\begin{align*}
|\dd^S D_{t,s} -\dd^{\bar S}\bar D_{t,s}|_{\alpha-\sigma-i\gamma'}
&\leq \int\limits_{s}^{t} |S _{t,r}-\bar S _{t,r}|_{\alpha-\delta\to\alpha-\sigma-i\gamma'} |N_r(y_r)|_{\alpha-\delta}dr
\\& \lesssim_{\|N\|_{1}} \|S -\bar S \|_{(\alpha-\delta,\alpha-\sigma-i\gamma')}\int\limits_{s}^{t} (t-r)^{\sigma+i\gamma'-\delta}~dr 
\\& 
\lesssim_{\|N\|_{1}}  \sup_{r\in [s,t]}|L_r-\bar L_r|_{1\to0}  (t-s)^{\gamma+i\gamma'}T^{1-\delta+\sigma}.
\end{align*}
We conclude thanks to \eqref{improve_D_diff} and Lemma \ref{lem:delta_S} that a similar bound holds for the plain increment \( \dd(D-\bar D)_{t,s}\).\\
As for the rough convolution, we infer from Corollary \ref{c:perturbation} and Lemma \ref{lem:composition} that
\begin{align*}
\nn{z-\bar z,z'-\bar z'}&
\lesssim _{\rho_{\gamma_0}(\X)}
\sup_{t\in [0,T]}|L^1_t-L^2_t|_{1\to0}T^{\gamma_0-\gamma}\|F(y),DF(y)\circ y'\|_{\cD^{\gamma,\theta\gamma'}_{X,\alpha-\sigma,\sigma}(\varepsilon)}\\
& \lesssim _{\rho_{\gamma_0}(\X)} \sup_{t\in [0,T]}|L^1_t-L^2_t|_{1\to0} T^{\gamma_0-\gamma}(R+1),
\end{align*}
for the same choice of \( \theta \) as before.
Now, the triangle inequality, Lemma \ref{lem:weight} and the previous bounds, imply
\begin{multline*}
\nn{\psi (S ;y ,y')_t- \psi (\bar S ;\bar y ,\bar y ^{\prime})_t}
\\
\le\nn{\psi (\bar S ;y ,y')_t- \psi (\bar S ;\bar y ,\bar y ^{\prime})_t}
	+\nn{f^x_t-\bar f^x_t,0}
	\\
	+\nn{D_t-\bar D_t ,0}
	+ \nn{z _t-\bar z_t,0}
	\\
	\le  C_3\bigg(\sup_{t\in [0,T]}|L_t-\bar L_t|_{1\to 0} + T^{\varkappa}\|y-\bar{y}, y'-\bar{y}'\|_{\cD^{\gamma,\gamma'}_{X,\alpha,\sigma}}\bigg) .
\end{multline*}
To conclude, we regard that \( y_0=\bar y_0=x\) and interpolate \( \cB_{\eta} \) between \( \cB_{\alpha-\varrho} \) and \( \cB_{\alpha} \) which gives for \( \vartheta=\frac{\alpha-\eta}{\varrho} \)
\[
|y_t-\bar y_t|_{\eta}
\le C t^{\vartheta\varrho}[\dd y-\dd \bar y]_{\varrho,\alpha-\varrho}^{\vartheta}|y-\bar y|^{1-\vartheta}_{0,\alpha},\quad t\in [0,T].
\]
Making use of Assumption \ref{ass:L_y} and \( T\le1 \) now leads to
\begin{align*}
\sup_{t\in [0,T]}|L_t-\bar L_t|_{1\to0}
&\le  2\ell C_3\sup_{t\in [0,T]}|y_t-\bar y_t|_{\eta} \\
&\le 2\ell CC_3 T^{\vartheta \varrho}\nn{y-\bar y,y'-\bar y'},
\end{align*}
hence the claimed contraction for \( 0<T\ll 1\). \\
The existence of a unique fixed point for the map \( (y,y)\mapsto \psi(S^y;y,y)\) in \( B_T(x) \) is a consequence of Picard Theorem, which finishes the proof of Theorem \ref{thm:RPDE}.

\subsubsection{Solutions starting from the domain of the generator}

Herein, we complete the proof of Theorem \ref{thm:RPDE} when \( \alpha=1 \), i.e.\ when the initial datum \( x \in D(L_t(x))=\cB_1 .\)
We now show the following result, whose proof is again based on a fixed point argument.
\begin{theorem}[The case when \( x\in \cB_1 \)]
	\label{thm:RPDE_critical}
Fix \( \alpha=1 \) and let \( \gamma_0,\gamma,\gamma',\sigma,\delta,\eta,\X,L,N,F\) be as in Theorem \ref{thm:RPDE}. Suppose moreover that 
\( L_t(\cdot) \) is continuously Fréchet differentiable on \( V \) (equipped with the \( \cB_\eta \)-topology) for each \( t\in [0,T] \) and that two constants \( \ell_0,\ell_1 \) exist such that
\begin{equation}\label{lip_constant_0_RPDE}
|\dd DL_{t,s}-\dd DL_{t,s}|_{\cB_\eta\otimes \cB_1\to\cB_0} \le \ell_0 |t-s|^\varrho\,.
\end{equation}
while
\begin{equation}
\label{lip_constant_1_RPDE}
\sup_{t\in [0,T]}|DL_t(x)-DL_t(y)|_{\cB_\eta\otimes \cB_1\to\cB_0}
\le \ell_1|x-y|_{\eta}, \quad \forall s\le t\in [0,T].
\end{equation}
Then, the same conclusions as in the previous theorem hold. Namely: given $x\in\cB_1$, there is a unique mild solution \((u,F(u))\in \cD^{\gamma,\gamma'}_{X,1 ,\gamma}([0,T];0)\) of \eqref{e:rpde2} and it depends continuously on \( (x,\X)\in \cB_1 \times \mathscr C^{\gamma_0}([0,T];\R^d)\).
\end{theorem}
\begin{proof}
We employ the same notations as in Section \ref{ssec:prf_thm_RPDE}.
The stability part of the fixed point map \( \psi_{T;N,F} \) is unchanged, hence we focus on the contraction property.
This time, Corollary \ref{c:perturbation} gives the inequality
\begin{multline}
	\label{pre_contraction_1}
\|\psi (S ;y ,y')_t- \psi (\bar S ;\bar y ,\bar y ^{\prime})_t\|_{\cD^{\gamma,\gamma'}_{X,\alpha,\gamma}}
\\
	\lesssim |L-\bar L|_{0,1\to 0} + T^{\varrho}[L-\bar L]_{\varrho;1\to 0} + T^{(\gamma-\sigma)\wedge(\gamma_0-\gamma)}\|y-\bar{y}, y'-\bar{y}'\|_{\cD^{\gamma,\gamma'}_{X,\alpha,\sigma}}
\\
\lesssim T^\varrho[\dd L-\dd \bar L]_{\varrho;1\to 0} + T^{(\gamma-\sigma)\wedge(\gamma_0-\gamma)}\|y-\bar{y}, y'-\bar{y}'\|_{\cD^{\gamma,\gamma'}_{X,\alpha,\sigma}}
\end{multline}
where we regard that \( L_0(y_0)=L_0(x)=L_0(\bar y_0) .\)
On the other hand, we have from the mean-value Theorem:
\begin{multline*}
 L_t(y_t) - L_s(y_s) - L_t(\bar y_t) + L_s(\bar y_s)
\\
 =(\dd L_{t,s}(y_t)- \dd L_{t,s}(\bar y_t))
+ (\int _0^1 DL_s(\tau y_t +(1-\tau)y_s )d\tau) \circ \dd y_{t,s} 
\\
-
(\int _0^1 DL_s(\tau \bar y_t +(1-\tau)\bar y_s )d\tau) \circ \dd \bar y_{t,s}  \,.
\end{multline*}
In particular, \eqref{lip_constant_0_RPDE} and \eqref{lip_constant_1_RPDE} yield
\[
|\dd L_{t,s}- \dd \bar L_{t,s} |_{1\to 0}
\le
|t-s|^\varrho\Big(\ell_0|y-\bar y|_{0,\eta}
+\ell_1|y-\bar y|_{0,\eta} [\dd y]_{\varrho;\alpha-\varrho}
+C(1+\ell_1)|\bar y|_{0,\eta}[\dd y-\dd \bar y]_{\varrho,\alpha-\varrho})\,.
\]
Inserting this bound into \eqref{pre_contraction_1} and proceeding similarly as in Section \ref{ssec:prf_thm_RPDE} yields
\begin{multline*}
\nnn{\psi (S ;y ,y')_t- \psi (\bar S ;\bar y ,\bar y ^{\prime})_t}
\\
\lesssim _{\ell_0,\ell_1,R} T^{\varrho} (|y-\bar y|_{0,\eta} + |\dd y-\dd \bar y|_{\varrho,\alpha-\varrho})
+T^{\varkappa}\nn{y-\bar y,y'-\bar y'}
\\
\lesssim T^{\varrho\wedge\varkappa}\nnn{y-\bar y;y'-\bar y'}.
\end{multline*}
This shows the claimed contraction property for \( 0<T\ll 1\).
\end{proof}

\subsubsection{Casting aside boundedness}
\label{ssec:cast_aside_bd}

To complete the proof of Theorem~\ref{thm:main}, we rely on Theorems \ref{thm:RPDE}, \ref{thm:RPDE_critical} and a localization argument.
More precisely, we truncate the nonlinear terms $F$ and $N$ with a smooth cut-off function $\chi:\R\to [0,1]$ such that
\begin{align*}
	\begin{cases}
		\chi(s) = 1, & |s|\leq 1\\
		\chi(s) =0, & |s|\geq 2.
	\end{cases}
\end{align*}
For each \(R>0\) we introduce \( N_{\cdot ,R}(y):= N_{\cdot}(y)\chi _R(|y|_\alpha) \), \( F_R(y):= F(y)\chi_R(|y|_{\alpha}) \) where \( \chi_R(\cdot)=\chi(\cdot/R). \) One can easily show that $F_R$ and $N_R$ satisfy the growth conditions imposed in Theorem~\ref{thm:RPDE}. 
Similarly if \( R'\in (0,\mathrm{dist}(x,\partial V))\), we define the quasilinear operator  
\[
V\mapsto \cL(\cB_1,\cB),\quad 
z\mapsto L^{R'}(z):=L\Big(z + 1_{|z-x|>R'}\frac{z -x}{|z-x|_{\alpha}}(R'-|z-x|_{\alpha})\Big)
\] 
which coincides with \( L(\cdot)\) on \( \bar B(x,R')\subset V\) and is easily seen to be Lipschitz.
With this at hand, we introduce the evolution family 
\( S_{t,s}^{y,R'}=\exp(\int_s ^t L^{R'}(y_r)dr),\) for \( (s,t)\in \Delta_2.\)
Applying Picard's fixed point Theorem, we see that for \( T(R,R')>0 \) chosen small enough, there exists a unique fixed point $(u,u')\in \cD_{X,\alpha,\sigma }^{\gamma,\gamma'}$ for $\psi_{T;N_R,F_R}(S^{\cdot,R'};\cdot,\cdot)$ and it is clearly a local solution of \eqref{e:rpde3}. 
Now, letting \( \tau_{R',\eta}:=\inf\{t\in (0,T], |y_t-x|_{\eta}\ge R'\}\)
and repeating the argument with $T^1:=\tau_{R,\alpha}\wedge\tau_{R',\eta}$ and \(x^1:= u_{T^1}\) instead of the initial pair $(0,x)$, we find the existence of a sequence $(T^n,x^n)$ such that by construction $\tau :=\sup_{n\in\N}T^n< T$ implies $\limsup_{n\to\infty}|x^n|_{\alpha }=\infty$ or \( x^n\) leaves the open set \( V\in \cB_\eta\). Exploiting local uniqueness, we can then construct a maximal solution by a standard concatenation procedure, see~\cite[Lemma 5.9]{HN20}.
\hfill\qed

\subsection{Continuity of the solution map}
Let \( \bar\X \in \cC^{\bar{\gamma_0}}(\R^d)\) be another rough path. 
For each \( (y,y')\in \cD^{\gamma,\gamma'}_{X,\alpha,\sigma} \) and \( (\bar y,\bar y')\in \cD^{\gamma,\gamma'}_{\bar X,\alpha,\sigma} ,\)
we extend the previously introduced metric as follows%
\footnote{This is no longer a metric unless \( X=\bar X\).}
\begin{multline}
\label{nota:distance}
\nn{y,y';\bar y,\bar y'}:=|y-\bar y|_{0,\alpha+\varepsilon}^{(\varepsilon)}+ [y-\bar y]_{\varrho,\alpha-\varrho}
\\
+|y-\bar y|_{0,\alpha } + |y'-\bar y'|_{0,\alpha-\sigma}^{(\varepsilon)} + [\dd y'-\dd \bar y']^{(2\varepsilon)}_{\gamma',\alpha-\sigma-\gamma'} + [R^y-\bar R^{\bar y}]^{(2\varepsilon)}_{\gamma+\gamma',\alpha-\sigma-\gamma'},
\end{multline}
where it is understood that \( R_{t,s}^y=\dd y_{t,s}-y'_t\cdot\dd X_{t,s} \) while \( \bar R_{t,s}^{\bar y} =\dd \bar y_{t,s}-\bar y'_t\cdot\dd\bar X_{t,s} .\)
Given \(L,N,F \) as in Theorem \ref{thm:main}, we define two controlled rough paths via the relations
\[
dy - L_t(y)ydt = N_t(y)dt + F(y)\cdot d\X ,\quad \quad y_0=x
\]
with \( y'=F(y)\),
and
\[
d\bar y - L_t(\bar y)\bar ydt = N_t(\bar y)dt + F(\bar y)\cdot d\bar \X ,\quad \quad \bar y_0=\bar x
\]
where \( \bar y'=F(\bar y)\) and \( \bar x\in \cB_\alpha\) is another choice of starting point.

\begin{claim}
	Let \( C:=\|y,F(y)\|_{\cD^{\gamma,\gamma'}_{X,\alpha,\sigma}}\vee \|\bar y,F(\bar y)\|_{\cD^{\gamma,\gamma'}_{\bar X,\alpha,\sigma}}\).
	If \( \mathscr R_{t,s}:=\mathscr R^{S,y}_{t,s}-\mathscr R^{\bar S,\bar y}_{t,s}\), then
	we have the estimate
	\begin{equation}\label{remainder_apriori}
	s^{2\varepsilon-\iota}|\mathscr R_{t,s}|_{\alpha-\sigma + \kappa}
	\lesssim_{\rho_\gamma(\X),K,\tilde K,C}
	|t-s|^{\gamma_0+\gamma -\sigma -\kappa-\iota}[\nn{y,y';\bar y,\bar y'}+\rho_{\gamma_0}(\X,\bar\X)]
	\end{equation}
	for each \( \kappa+\iota\in [-\sigma,\gamma_0+\gamma-\sigma),\)
	\( \iota\in [0,2\varepsilon].\)
\end{claim}
\begin{proof}
	Introduce \( (\zeta,\zeta')=(F(y),DF\circ F(y))\) and define \( (\bar \zeta,\bar \zeta')\) similarly. For \( u\le v\), let \( \xi_{v,u}=\zeta_v\cdot \dd X_{v,u} + \zeta'_v:\xx_{v,u} \). Define \( \bar\xi_{v,u}\) in a similar way and observe that
	\begin{equation}\label{diff_xi}
	\begin{aligned}
	(\xi - \bar \xi)_{v,u} 
	&=(\zeta_v-\bar\zeta_v)\cdot \dd X_{v,u} + (\zeta'_v-\bar\zeta'_v):\xx_{v,u}
	\\&\quad 
	+\zeta_v\cdot \dd (X-\bar X)_{v,u} + \bar\zeta_v:(\xx -\xxx)_{v,u}\,.
	\end{aligned}
	\end{equation}
	If an additional point \( m\in [u,v]\) is given, then Chen's relations also yield 
	\begin{equation}\label{diff_dd_xi}
	\begin{aligned}
	(\dd \xi-\dd\bar\xi)_{v,m,u}=
	&=(R^{\zeta}-\bar R^{\bar \zeta})_{v,m}\cdot \dd X_{m,u} + (\dd\zeta' - \dd \bar\zeta)_{v,m}:\xx_{m,u}
	\\&\quad 
	+\bar R_{v,m}\cdot \dd (X-\bar X)_{m,u} + \dd\bar\zeta_{v,m}:(\xx -\xxx)_{m,u}\,.	
	\end{aligned}	
	\end{equation}
	On the other hand, if \( s\le u\le m\le v\le t\), we see that
	\begin{equation}\label{decomp:continuity}
	\begin{aligned}
	\delta \mathscr R_{v,m,u}
	&= S_{t,u}(\dd\xi-\dd\bar\xi)_{v,m,u} 
	+ (S-\bar S)_{t,u}\dd\bar\xi_{v,m,u}
	\\&
	+S_{t,u}(I-S_{m,u})(\xi-\bar\xi)_{v,m}
	+ (S-\bar S)_{t,u}(I-S_{m,u})\bar\xi_{v,m}
	+\bar S_{t,u}(\bar S-S)_{m,u}\bar \xi_{v,m}\,.
	\end{aligned}
	\end{equation}
	Combining this identity with \eqref{diff_xi} and \eqref{diff_dd_xi}, we see thanks to similar arguments as for the fixed point that each term in \( \dd \mathscr R_{v,m,u}\) can be estimated above by a constant times an expression of the form
	\[
	A|t-u|^{-\lambda}|u-m|^{\nu}|m-v|^{\mu-\nu}m^{-\epsilon}
	\]
	where \( A\lesssim \nn{y,y';\bar y,\bar y'} + \rho_{\gamma_0}(\X,\bar \X)\).
	For instance, Remark \ref{rem:other_RP} and the smoothing properties of \( S\) yield
	\[
	\begin{aligned}
	|S_{t,u}(R^{\zeta}-\bar R^{\bar \zeta})_{v,m,u} \cdot \dd X_{m,u}|_{\alpha-\sigma+\kappa}
	&\lesssim 
	(t-u)^{-\kappa-\sigma-\gamma'}(v-m)^{\gamma+\gamma'}
	\\&\quad \quad 
	\times(m-u)^{\gamma_0}m^{-2\varepsilon}\nn{y,y';\bar y,\bar y'},
	\\
	|S_{t,u}(\dd\zeta'-\dd\bar\zeta')_{v,m} \cdot \xx_{m,u}|_{\alpha-\sigma+\kappa}
	&\lesssim 
	(t-u)^{-\kappa-\sigma-\gamma'}(v-m)^{\gamma'}
	\\&\quad \quad 
	\times(m-u)^{2\gamma_0}m^{-2\varepsilon}\nn{y,y';\bar y,\bar y'},
	\\
	|S_{t,u}\bar R^{\bar \zeta}_{v,m} \cdot \dd (X-\bar X)_{m,u}|_{\alpha-\sigma+\kappa}
	&\lesssim 
	(t-u)^{-\kappa-\sigma-\gamma'}(v-m)^{\gamma+\gamma'}
	\\&\quad \quad 
	\times(m-u)^{\gamma_0}m^{-2\varepsilon}[X-\bar X]_{\gamma_0},
	\end{aligned}
	\]
	and 
	\[
	\begin{aligned}
	&|S_{t,u}\dd\bar\zeta_{v,m} \cdot (\xx-\xxx)_{m,u}|_{\alpha-\sigma+\kappa}
	\\
	&\lesssim 
	(t-u)^{-\kappa-\sigma-\gamma'}(v-m)^{\gamma'}(m-u)^{2\gamma_0}m^{-2\varepsilon}
	([\XX-\mathbb{\bar X}]_{2\gamma_0} + [\dd (X-\bar X)]_{\gamma_0}[\dd(X+\bar X)]_{\gamma_0}),
	\end{aligned}
	\]
	which gives the corresponding estimate for the first term in \eqref{decomp:continuity}.
	Using Assumption \ref{ass:L_y} to deal with occurences of \( S-\bar S\), the other terms in \eqref{decomp:continuity} are treated by similar arguments.
	The conclusion follows by an application of Lemma \ref{lem:dyadic}.
\end{proof}
\begin{corollary}
	The solution map \[
	\mathcal S\colon \mathscr C^{\gamma_0}(\R^d)\times \cB_{\alpha}\to \cD_{X,\alpha,\sigma}^{\gamma,\gamma'}\]
	which to every pair
	\((\X,x)\) assigns the solution \( (u,F(u))=\mathcal S(\X,x)\) of \eqref{e:rpde2}, is continuous. 
	More precisely, we have the explicit bound 
	\[
	\nn{\mathcal S(\X,x)-\mathcal S(\mathbf{\bar X},\bar x)}
	\lesssim |x-\bar x|_{\alpha} + \rho_{\gamma_0}(\X,\mathbf{\bar X})\,.
	\]
	where the implied constant only depends on \( \rho_{\gamma_0}(\X),\rho_{\gamma_0}(\mathbf{\bar X}), |x|_\alpha\) and \( |\bar x|_{\alpha}\).
\end{corollary}
\begin{proof}
	We assume for notational simplicity that \( N=0\). First, observe that
	\begin{multline*}
	y_t-\bar y_t
	= (S-\bar S)_{t,0}x + \bar S(x-\bar x)
	+ (S-\bar S)_{t,0}(\bar\zeta_t \cdot X_t + \bar\zeta'_t:\xx_{t,0}) 
	\\
	+ \bar S_{t,0}((\zeta-\bar\zeta)_t \cdot X_t + (\zeta'-\bar\zeta')_t:\xx_{t,0}) 
	\\
	+ \bar S_{t,0}(\bar\zeta_t \cdot (X-\bar X)_t + \bar\zeta'_t:(\xx-\xxx)_{t,0}) 
	+\mathscr R_{t,0}\,.
	\end{multline*}
	From the previous claim and Assumption \ref{ass:L_y}, it follows that 
	\begin{multline*}
	t^\varepsilon|y_t-\bar y_t|_{\alpha+\varepsilon}
	\lesssim
	|y-\bar y|_{0,\eta}[|x|_\alpha + C\rho_{\gamma_0}(\X)t^{\gamma_0-\gamma}]
	\\
	+|x-\bar x|_\alpha + (t^{\gamma_0-\sigma} + t^{2(\gamma_0-\gamma)})[\nn{y,y';\bar y,\bar y'}+C\rho_{\gamma_0}(\X,\bar\X)]
	\end{multline*}
	Proceeding with similar arguments as in the proof of the contraction property for the previous fixed point theorem, each term in \eqref{nota:distance} is estimated in the same way, and one ends up with the relation
	\[
	\nn{y,y';\bar y,\bar y'}
	\lesssim_{K,\tilde K,C} \rho_{\gamma_0}(\X,\bar\X) + |x-\bar x|_{\alpha}
	+
	o(t)\nn{y,y';\bar y,\bar y'}
	\]
	where \( o(t)\to 0\) if \( t\to0.\)
	Regarding this, we conclude that if \( T \) is chosen small enough (depending only on the implicit quantities above), then
	\( \nn{y,y';\bar y,\bar y'} \lesssim
	\rho_\gamma(\X,\bar\X) + |x-\bar x|_{\alpha}
	\)
	which proves the desired continuity, locally in time. 
	To obtain continuity up to the common maximal existence time, we can repeat the argument starting with \( u_{T},\bar u_{T} \) instead of \( x,\bar x \), and so on. The proof is then completed by an obvious induction.
\end{proof}

\section{Quasilinear parabolic systems}
\label{sec:quasilinear_systems}

Throughout this section we let $m,n,n_1\in\mathbf N$, \( k\in \mathbf N_0\) with \( k\le 2m-1\) and we fix a domain $\dom\subset \R^n.$
We denote by $G$ an open set of $\R^{n_1}\times \R^{n_1\times n}\dots \R^{n_1\times n^k}\simeq \R^{{\mu}(k)}$ where, given an integer \( j \), we adopt the notation \[
{\mu}(j):={n_1}\sum_{|\beta|\leq j}1.
\]
We consider a family of linear differential operators of order $2m:$
\begin{equation}\label{generic_A}
\mathscr L(t,x,y)u:=(-1)^{m-1}\sum_{|\beta |\leq 2m} a_\beta (t,x,y )D^\beta u
\end{equation}
for each \( t\in [0,T],\) \( x\in \dom\) and \( y\in G\) with
\begin{equation}
	\label{coef_regularity}
a_{\beta} \in C^{\varrho ,0,1-}([0,T]\times \dom\times G,\cL(\R^{n_1}))\,
\end{equation}
and we suppose that $\mathscr L$ is strongly parabolic. \\
Denoting by $a(t,x,y ;\xi ):=\sum_{|\beta |=2m}a_\beta (t,x,y )\xi ^{\beta } \in \cL(\R^{n_1})$, this means that
\begin{equation}\label{coercive}
\mathrm{Re}\left\langle a(t,x,y;\xi )\zeta \,,\, \zeta \right\rangle>0
\end{equation}
 for any 
\( (t,x,y,\xi,\zeta)\in[0,T]\times\mathcal{\bar O}\times G\times(\R^n\setminus 0)\times (\C^{n_1}\setminus 0). \)\\
We now fix \( p\in [1,\infty] \) and let \( \cX=L^p(\R^n,\R^{n_1})\) for some \( n,{n_1}\in \mathbf N\), while \( L\) is as in \eqref{generic_A}--\eqref{coercive}.
In the special case when \( m=2 \) and \( \cX_1=W^{2,p}(\R^n,\R^{n_1}) \), then the Sobolev scale \(\cB_\alpha:= W^{2\alpha,p}(\R^n;\R^{n_1})\) satisfies Assumption \ref{ass:intermediate} (see Example \ref{exa:scales}). In general if \( m> 2\), we can let \( \cB_\alpha := (\cX,\cX_1)_{\alpha,p}\), which coincides with \( W^{2\alpha m,p}(\R^n,\R^{n_1})\) if and only if \( 2m\alpha\notin \mathbb N.\)

\subsection{The Dirichlet case}

In the setting of \eqref{generic_A}--\eqref{coercive} with \(k_0:=k\in\{0,\dots, 2m-1\}\),
let us fix two additional integers \( 0\le k_1\le 2m-1 \) and \( 0\le k_2 \le \lfloor 2\gamma_0 m \rfloor \)
and consider Nemyitskii non-linearities \( g(x;\cdot)\in C^1(\R^{\mu(k_1)};\R^{n_1}) \) and \( f(x;\cdot )\in C^3(\R^{\mu(k_2)}; \R^{n_1}\otimes\R^d) \) for each \( x\in \dom \). 
We assume that the dependency on \( x\) is regular enough in the following sense: for every smooth and compactly supported \( \varphi=\varphi(y)\) on \( G\), it holds \( |\varphi g|_{W^{2\alpha,p}_{\mathscr D}(\dom;C_b^1)} , |\varphi f|_{W^{2\alpha,p}_{\mathscr D}(\dom;C_b^3)}<\infty.\)\\
We specialize our main well-posedness result (Theorem~\ref{thm:RPDE}) to the case of a quasilinear parabolic evolution system driven by a rough path \( \X=(X,\XX)\in \mathscr C^{\gamma_0}(\R^d) \) with \( \gamma_0\in (\frac13,\frac12) \), of the form
\begin{equation}
	\label{quasilinear_system}
	\left \{\begin{aligned}
		&du - \mathscr L(t,x,u,\dots,D^{k_0}u)udt 
		\\
		&\quad \quad \quad 
		= g(t,x,u,\dots ,D^{k_1}u)dt
		+f(x,u,\dots ,D^{k_2}u)\cdot d\X_t, 
		\quad \text{ in }(0,T]\times\dom
		\\
		&\mathscr Du=0\quad \text{ on }(0,T]\times \partial \dom,
		\\
		&u_0(\cdot)=u^0\quad \text{ on }\dom,
	\end{aligned}\right .
\end{equation}
whose unknown \( u =(u^j)_{j=1,\dots ,n_1}\) is a continuous path from \( [0,T]\to \cB_\alpha=W^{2m\alpha,p}_{\mathscr D}(\dom;\R^{n_1}) ,\) where $\alpha\in(0,1)$ is suitably chosen.
Herein $\dom$ is a bounded smooth domain in $\R^n,$ and $\mathscr D$ denotes the Dirichlet boundary operator
\( \mathscr Du:= \big(u\big|_{\partial\dom},\dots , \frac{\partial^{m-1}u}{\partial\nu^{m-1}}\big|_{\partial\dom}\big) \),
where $\nu$ is the outward unit vector at $\partial\dom$. \\
Firstly we introduce the notation
\begin{equation}
	\label{nota_V}
	V^{\eta,p}(\cO):=\{u\in W^{2m\eta}_{\mathscr D}: u(\dom),\dots ,D^{k_0}u(\dom)\subset G\}\,.
\end{equation}
Because of the Sobolev embedding theorem, it is plain to check that  \( V^{\eta,p} \) is open in \(  \cB_{\eta} \) if \( \eta>\frac{1}{2m}(k_0+n/p) \).
For \( \beta \) large enough and $u\in\cB_\beta$ we define the Nemyitskii operators
\[
N_t(u):= g(t,\cdot,u(\cdot),\dots ,D^{k_1}u(\cdot)),
\quad 
F(u):= f(\cdot,u(\cdot),\dots , D^{k_2}u(\cdot)).
\]
Similarly if \( v\in V^{\beta,p}(\cO) \) and \( u\in \mathcal B_1 \), we define 
\[
L_t(v)u:=\mathscr L(t,\cdot,v(\cdot),\dots ,D^{k_0}v(\cdot))u\,.
\]
With this at hand, our main result in this section states as follows.

\begin{theorem}
	\label{thm:system}
	Let \( n< p < \infty \), fix a real number $\alpha\notin \frac{1}{2m}(\frac1p+\mathbf N_0)$ such that
	\[ \alpha_0:=\frac{1}{2m}\left (\max_{0\le i\le2}(k_i)+\frac{n}{p}\right ) \vee (1-\gamma_0-\gamma+2\sigma) < \alpha <1\]
	and finally let \( \sigma\in (\frac{k_2}{2m},\gamma_0) \). 
	For each \( u^0\in V^{\alpha,p}(\cO) \) there is a \( \tau\in(0,T] \), such that the problem \eqref{quasilinear_system} has a maximal solution
	\[
	u\in C([0,\tau);W^{2m\alpha,p}_{\mathscr D}(\cO))\cap C^{\gamma,\varepsilon}([0,\tau); W^{2m(\alpha-\sigma),p}_{\mathscr D}(\cO)).
	\]
	This solution is unique amongst controlled paths \( (u,u') \) such that \( u'_t=F(u_t) \) for all \( t\in [0,\tau) \) and
 \( (u,u')\in \cD^{\gamma,\gamma}_{X,\alpha,\sigma}([0,\tau);\gamma) \).
\end{theorem}

\begin{proof}
	Let us first check that Assumption \ref{ass:L_y} is satisfied on \( L \) for any $\eta\in (0,1)$ such that \( 2m\eta\in(k+\frac np ,2m\alpha). \)
	The property \ref{Q1} is a consequence of the regularity \eqref{coef_regularity} on the coefficients and Sobolev embeddings: observe indeed that for any \( u\in \cB_1 \) and each \( v,y\in W^{\beta,p} \)
	\[
	\begin{aligned}
		|L(v)u-L(w)u|_{0} \le \sum_{|\beta |\leq 2m} |a_\beta (t,x, v,\dots,D^kv )-a_\beta (t,x, w,\dots,D^kw )|_{L^\infty}|D^\beta u|_{L^p}
		\\
		\lesssim _{a}|v-w|_{W^{k,\infty}}|u|_1,
	\end{aligned}
	\]
	which is bounded by a constant times \( |v-w|_{\eta}\) provided \( 2m\eta > k+ \frac np \).
	As for \ref{Q2}, it is a well-known consequence of \eqref{coercive} (see, e.g., \cite{amann1985global}), the continuous embedding \( W^{2m\eta,p}(\cO)\hookrightarrow C^k(\cO) \) and the definition of \( V^{s,p}(\cO) \) which ensures that $(u(\cdot),\dots ,Du(\cdot))$ is \( G \)-valued.\\
	Furthermore, for the right hand side of \eqref{quasilinear_system}, we know due to~\cite[Lemma 10.1]{amann1986quasilinear} that \( N \) belongs to \(C^{1-}(V^{\alpha,p}(\cO), W^{2m(\alpha-\delta),p}(\cO))\) for every \(\delta> \frac{1}{2m} (k_1+\frac np)\).
	Similarly, fixing \( \sigma\in (\frac{k_2}{2m},\gamma_0) \) and
	\( \beta_0> \frac{1}{2m}(\frac np +k_2) \),
	we see that for any \( \beta\in [\beta_0,\alpha] \) the map \( F \colon \cB_{\beta}\to\cB_{\beta-\sigma} \) is well-defined and three-times continuously Fr\'echet-differentiable.\\
	Consequently, Theorem~\ref{thm:RPDE} entails the existence of a unique maximal mild solution in $\cD^{\gamma,\gamma}_{X,\alpha,\sigma}([0,\tau);(\gamma-\sigma)_+)$ as described above. 
\end{proof}

\subsection{Other types of boundary conditions for second order differential operators and polynomial nonlinearities}\label{neumann}
Herein, we restrict to \( m=1=k_0=k_1\), \(k_2=0\). We assume for concreteness that the nonlinear terms are polynomials of the form 
\begin{equation}\label{polynomial_g}
	 g(u,Du)(x)= \sum_{i,j,k}g_{i,k}(x)u^i(x)^{\mu_{ijk}}D_ku^j(x)^{\nu_{ijk}}
\end{equation}
and similarly 
\begin{equation}\label{polynomial_f}
f(u)(x)=\sum_{i}f_i(x)u^i(x)^{q_i},
\end{equation}
 for coefficients \( f_i,g_{i,k}\) belonging to a well-chosen functional space, and given numbers \( \mu_{ijk},\nu_{ijk},q_i\in \{0,1,\dots \}\).

Our purpose now is to discuss more general types of boundary conditions in this setting.
We consider the ansatz 
\begin{equation}
	\label{system_B}
	\left \{\begin{aligned}
		&du -\mathscr L(u,Du) u dt = g(u,D u)+ f(u)\cdot d\X\,\quad 
		\text{ on }[0,T]\times\dom\,,
		\\
		&\mathscr Bu_t=0\text{ for each }t\in [0,T],
		\\
		& u_0\text{ given in } W^{2\alpha,p}_{\mathscr B}(\dom;\R^{n_1}),
	\end{aligned}\right .
\end{equation}
for some boundary operator \( \mathscr B \), and suppose that we are either in one of the following cases
\begin{itemize}
	\item (periodic) \( \dom=(-1,1)^n \) in which case we introduce \( \mathscr Bu(x)= u(x)-u(Rx) \) where \(Rx\) is the vector obtained from \( x \) by replacing each coordinate of length \( 1 \) with their opposite value;
	\item (Neumann, resp.\ Dirichlet homogeneous) \( \dom \) is bounded with smooth boundary and \( \mathscr Bu(x)= \nu(x)\cdot \nabla u(x)\), resp.\ \( \mathscr Bu(x)=u(x)\), \( x\in \partial \dom \);
	\item (whole space) \( \dom =\R^n \) and \( \mathscr B=0. \)
\end{itemize}
With this at hand we introduce the spaces \( W^{\beta,p}_{\mathscr B}= W^{\beta,p}(\dom;\R^{n_1})\cap \mathrm{Ker}\mathscr B \) for each \( \beta\in [0,2] \) and \( p\in [1,\infty] .\)
The main result of this paragraph, whose proof is left to the reader, states as follows.
\begin{theorem}\label{thm:system_B}
	Let \( n<p<\infty \) such that \( \left (\frac12+\frac{n}{2p}\right )\vee (1-2\gamma_0)< \alpha \le1 \) and fix \( x\in V^{\alpha,p}(\cO).\) If \( \alpha=1\), suppose in addition that
	\begin{multline}
		\label{coef_regularity_critical}
		\dd a \in \cC_2^{\varrho }\Big(0,T;,C^{0,1-}(\dom\times G,\cL(\R^{n_1}))\Big)
		\\ \text{and that}\quad 
		D_y a \in \cC^{0}\Big(0,T;C^{0,1-}(\dom\times G,\cL(\R^{n_1}))\Big).
	\end{multline}
Let the non-linearities \( f(u),g(u,Du)\) be polynomials in the unknown, with coefficients \( f_i(x),g_{ijk}(x)\) belonging to \( \cB_\alpha\). 
Then there exist \( \tau\in(0,T) \) and a unique solution
	\[
	u\in C^0([0,\tau);\cB_{\alpha})\cap C^{\gamma,(\gamma-\sigma)_+}([0,\tau);\cB_{\alpha-\sigma})
	\]
	of the system \eqref{system_B} such that
	\begin{itemize}
		\item \( (u,f(u))\in \cD^{\gamma,\gamma}_{X,\alpha,0}([0,\tau);\gamma) \) if \( \alpha<1-\gamma\)
		\item \((u,f(u))\in \cD^{\gamma,\gamma}_{X,\alpha,\sigma}([0,\tau);\gamma-\sigma) \) if \( 1-\gamma \le \alpha \le 1-\gamma+\sigma\) for some \( \sigma\in(0,\gamma]\).
	\end{itemize}
\end{theorem}

\begin{remark}
	Obviously, it is possible to consider other types of right hand sides for \eqref{quasilinear_system} which are neither polynomials nor Nemyitskii operators and obtain similar results.
	For instance, the fractional flux-type nonlinearity \( f(u) =(-\Delta)^\sigma(u^q)\) would work choosing \( q\ge1\) suitably small. 
\end{remark}


\section{Stochastic examples}
\label{sec:stochastic_examples}
Herein we fix a complete, filtered probability space \( (\Omega,\mathcal F,(\mathcal F_t)_{t\in[0,T]},\mathbb P) \) with a right-continuous filtration $(\cF_t)_{t\in[0,T]}$.  
Basic examples of rough paths \( \X =(X,\mathbb X)\) are provided by the theory of stochastic processes and It\^o (or Stratonovich) integration. 
For instance if \( \gamma_0<\frac12 \), a \( d \)-dimensional Wiener process \( W \colon \Omega\times[0,T]\to\R^d\), gives rise to the following \( \mathscr C^{\gamma_0}(0,T;\R^d) \)-valued random variables 
\[
\mathbf W_{s,t}=(W_t-W_s,\mathbb W_{s,t}),\quad 
\mathbf{\tilde W}_{s,t}=(W_t-W_s,\mathbb {\tilde W}_{s,t}),\quad 0\le s\le t\le T,
\]
where for \( 1\le i,j\le d \) the term \( \mathbb W^{i,j}_{t,s}=\int_s^t \dd W_{r,s}^idW^j_r \) is understood in the It\^o sense, while \( \mathbb{\tilde W}^{i,j}_{t,s}= \int_s^t \dd W^i_{r,s}\circ dW^j_r = \mathbb W_{t,s} + \frac12(t-s)\mathbf 1_{i=j}\) corresponds to the Stratonovich integration.

\subsection{The stochastic Shigesada-Kawasaki-Teramoto population model}\label{skt}
Let $\dom\subset\R$ be an open bounded domain. We fix parameters $\alpha_1, \alpha_2, \delta_{11}, \delta_{21} > 0$. We are interested in studying a cross-diffusion SPDE, which was originally introduced by Shigesada, Kawasaki and Teramato (SKT) in the deterministic setting, to analyze population segregation by induced cross-diffusion in a two-species model. Note that the nonlinear {drift} term correspond to those arising in the classical Lokta-Volterra competition model. The stochastic SKT system is given by 
\begin{equation}
	\label{eq:SKT}
	\begin{split}
		&du_{1}=\left(\Delta\left(\alpha_{1} u_{1}+\gamma_{1} u_{1} u_{2}+\beta_{1} u_{1}^{2}\right)+\delta_{11} u_{1}-\theta_{11} u_{1}^{2}-\theta_{12} u_{1} u_{2}\right) dt+F_{1}(u_{1}, u_{2}) d  W^{1}_t, \\
		&du_{2}=\left(\Delta\left(\alpha_{2} u_{2}+\gamma_{2} u_{1} u_{2}+\beta_{2} u_{2}^{2}\right)+\delta_{21} u_{2}-\theta_{21} u_{1} u_{2}-\theta_{22} u_{2}^{2}\right) d t+F_2(u_{1}, u_{2}) dW^{2}_t,
	\end{split}
\end{equation}
for $t \in[0,T]$ and $x \in \dom$ and is supplemented with the following boundary and initial conditions:
\begin{align*}
	&\frac{\partial}{\partial n} u_{1}(t,x) = \frac{\partial}{\partial n} u_{2}(t,x) = 0, & t > 0, \; x \in \partial \dom,\\
	&u_{1}(x, 0) = u_{1}^0(x) \geq 0,\quad  u_{2}(x, 0) = u_{2}^0(x) \geq 0, &x \in \dom.
\end{align*}
$W = (W^1, W^2)$ is a two-dimensional Wiener process and the equation~\eqref{skt} is understood in the It\^o sense. Its solution $u:=(u_1,u_2)$, where $u_1=u_1(x, t)$ and $u_2 = u_2(x, t)$,
represents the densities of two competing species $S_1$ and $S_2$ at certain location $x \in \dom$, at time $t$.
The coefficients $\theta_{11}, \theta_{22} > 0$ denote the intraspecies competition rates in $S_1$, respectively in $S_2$
and $\theta_{12}, \theta_{21} > 0$ stand for the interspecies competition rates between $S_1$ and $S_2$. Furthermore, the terms $\Delta(\beta_1 u_1^2)$ and $\Delta(\beta_2 u_2^2)$ represent the self-diffusions of $S_1$ and $S_2$ with rates $\beta_1, \beta_2 \ge 0$, and $\Delta(\gamma_1 u_1 u_2)$, $\Delta(\gamma_2 u_1 u_2)$ represent the cross-diffusions of $S_1$ and $S_2$ with rates $\gamma_1, \gamma_2 \ge 0$. The nonlinear terms $F_1$ and $F_2$ are Nemytskii operators as above. Further details on this model are available in~\cite[Chapter 15]{Yagi1}. The SKT system \eqref{eq:SKT} can be rewritten as an abstract quasilinear SPDE on the product space $\mathbb{L}^2(\cO):=L^2(\cO)\times L^2(\cO)$
\begin{equation}
	\label{eq:SKT_abs}
	\begin{cases}
		&du=[L(u) u+N(u)]~ d t+F(u)\cdot d{\bf W}_t, \quad t \in[0, T] \\
		&u(0)=u_{0},
	\end{cases}
\end{equation}
where 
\[
L(u) u:=\textnormal{div}(B(u)\nabla u)-\Gamma u,
\]
with 
\[B(u)=\begin{pmatrix}
	\alpha_{1}+2\beta_1u_{1} + \gamma_1 u_{2} & \gamma_1 u_{1} \\
	\gamma_2 u_{2} & \alpha_{2}+2 \beta_2 u_{2} +\gamma_2 u_{1}
\end{pmatrix}, \qquad 
\Gamma(u)=\begin{pmatrix}
	\delta_{11} & 0 \\
	0 & \delta_{21}
\end{pmatrix}.\]
Here ${\bf W}(\omega)$ is the It\^o lift of a two-dimensional Brownian motion, which belongs to $\cC^{\gamma_0}(\R^2)$ for every $\gamma_0<\frac{1}{2}$ and each $\omega\in \Omega^{\gamma_0}$, where $\Omega^{\gamma_0}\subset\Omega$ of full probability. 
Furthermore, the nonlinear term $N$ corresponds to the Lotka-Volterra type competition model
\[
N(u)=\begin{pmatrix}
	2\delta_{11} u_{1}-\theta_{11} u_{1}^{2}-\theta_{12} u_{1} u_{2} \\
	2\delta_{21} u_{2}-\theta_{21} u_{1} u_{2}-\theta_{22} u_{2}^{2}
\end{pmatrix}.\]
\begin{theorem}
 Let $\alpha\in(\frac{3}{4},1]$ and let $u_0$ belong to an open set of $\cB_\alpha$. Then there exists a stopping time $\tau\in(0,T)$ and a unique solution  $(u,F(u))\in \cD^{\gamma,\gamma}_{W,\alpha,0}(\gamma)$ of~\eqref{skt} such that the path component $$u\in C^0([0,\tau);\cB_\alpha)\cap C^{\gamma,\gamma}([0,\tau);\cB_{\alpha-\sigma}).$$ 
\end{theorem}
\begin{proof}
	In order to ensure the positive definiteness of the matrix $B$, the following restriction on the parameters is necessary:
	\begin{align}\label{pos:def:skt}
		\gamma_{1}^{2}<8\alpha_{1}\beta_{1}\quad\mbox{and}\quad\gamma_{2}^{2}<8\alpha_{2}\beta_{2}.
	\end{align}
	This assumption is required in order to show that $L(u)$ satisfies Assumption~\ref{ass:L_y} (with $\eta=\frac{1+\varepsilon_1}{2}$, for  $\varepsilon_1\in(0,\frac{1}{2})$) for $u$ belonging to an open set of $\mathbb{H}^{1+\varepsilon_1}(\cO):=\mathbb{H}^{1+\varepsilon_1}(\cO)\times \mathbb{H}^{1+\varepsilon_1}(\cO) $ for some $0<\varepsilon_1<\frac{1}{2}$, see~\cite[Proposition 15.1]{Yagi1}. Moreover, for $u\in \mathbb{H}^{1+\varepsilon_1}(\dom)$, one can show that $D(L(u)^\theta)=[\mathbb{L}^2(\dom),\mathbb{H}_N(\dom)]_\theta$ for any $0\leq \theta\leq 1$, where $[\cdot,\cdot]_\theta$ stands for complex interpolation.
	This further leads to~\cite[Proposition 15.3]{Yagi1} 
	\begin{align*}
		\begin{cases}
			D(L(u)^\theta) = \mathbb{H}^{2\theta}(\dom),~ \mbox{ for } 0\leq\theta<\frac{3}{4}\\
			D(L(u)^\theta)= \mathbb{H}^{2\theta}_N(\dom),~ \mbox{ for } \frac{3}{4}<\theta\leq 1.
		\end{cases} 
	\end{align*}
	The set of initial conditions is given by $\cK =\{ u_0 \in H^{1+\varepsilon_2}(\cO)\times H^{1+\varepsilon_2}(\cO), u_0\geq 0, v_0\geq 0 \}$ for some $0<\varepsilon_1<\varepsilon_2$ and the nonlinear drift term $N$ maps $\mathbb{H}^{1+\varepsilon_1}(\dom)$ into $\mathbb{L}^2(\dom)$. This means that we can work with the Sobolev tower built from the fractional powers of the operator $L(u)$ for $u\in\cB_\eta$. Therefore we obtain the spaces $\cB_\alpha= \mathbb{H}^{2\alpha}(\dom)$, for $\alpha=\frac{1+\varepsilon_2}{2}$. 
	Putting this together for  $\alpha=\frac{1+\varepsilon_2}{2}$, we observe that we are in the setting of Theorem~\ref{thm:system_B}. Therefore the conclusion follows.
\end{proof}
\begin{remark}
	The existence of a {\em pathwise mild solution} $u\in C([0,\tau];\mathbb{H}^{1+\varepsilon_2}(\dom))$ for~\eqref{skt} was obtained in~\cite{kuehn2018pathwise}. The equivalence of these two solution concepts (pathwise mild and controlled rough path) will be discussed in a forthcoming work.
\end{remark}

\subsection{Stochastic Landau-Lifshitz-Gilbert equation}
\label{subsec:LLG}
For any given \( \phi\in W^{1,\infty}_{\mathscr B}(\cO) \),
we aim to derive the existence of pathwise local mild solutions of the stochastic Landau-Lifshitz-Gilbert (LLG) system. 
 This can be written as the following vector-valued stochastic equation in the Stratonovich sense
\begin{equation}
\label{LLG}
\left \{\begin{aligned}
	&du = (\Delta u + u\times\Delta u+ u|\nabla u|^2)dt + u\times \circ dW_\phi -\epsilon u\times(u\times \circ dW_\phi)\,\quad 
\text{ on }[0,T]\times\dom\,,
\\
&\mathscr Bu_t=0\text{ for each }t\in [0,T], 
\\
& u_0\text{ given in } W^{2\alpha,p}_{\mathscr B}(\dom;\mathbb S^2)\,~ \text{ with } |u_0(x)|=1 \text{ a.e.}
\end{aligned}\right .
\end{equation}
where \( \epsilon\in \{0,1\}\),
 \( W_{\phi,t}(x)= \phi(x)W_t\), for a 3-dimensional Wiener process \( W _t(\omega)=(W^1_t,W^2_t,W^3_t)(\omega),\) while \( \times \) denotes vector product. This equation describes the magnetization of a ferromagnetic material occupying a domain $\cO$ and the noise term is introduced in order to model thermal fluctuations. Based on a Doss-Sussman transformation weak martingale solutions have been investigated in~\cite{br1}. A similar transformation was considered in~\cite{hausenblas} in order to treat the LLG equation driven by a geometric rough path.\\
Here we justify that~\eqref{LLG} fits into our abstract functional analytic setting and derive the existence of mild solutions in the sense of Definition~\ref{def:var_sol}. Note that the quasilinear term has the form \eqref{generic_A} with $m=2$, \( k=0\), $n_1=3$ and for each \( y\in \R^3 \)
\[
\mathscr L(y)= \sum_{1\le i,j\leq n}a_{ij}(y)D_iD_j
\,,\quad \text{for }
a_{ij}(y):=\mathbf1_{i=j}(\id_{3}+y\times\cdot)\,.
\]
Observe that for each $y\in \R^3$, \(\zeta\in \R^n,\) and \(\xi\in\C^3\)
\[
\Re \langle a(y;\xi)\zeta,\zeta\rangle =|\xi|^2\Re\langle \zeta-y\times\zeta,\zeta\rangle=|\xi|^2|\zeta|^2
\,,
\]
hence the property \eqref{coercive} is clear.
Moreover, since the previous right hand side is independent of \( y \), we may take \( G=\R^{\mu(0)}=\R^3 \) in \eqref{nota_V}, which implies that \( V^{\alpha,p}(\cO)\) is the whole space \(W^{\alpha,p}(\cO) .\)

\begin{theorem}
Let \( n<p< \infty \) and \( \alpha \in (\frac12 + \frac{n}{2p},1] \).
Then there exists a stopping time \( \tau>0 \), a random variable \( (u,u')\colon \Omega\to \cD^{\gamma,\gamma}_{X,\alpha,0}(0,\tau;\gamma) \)
with the property that \( \mathbb P\)-almost surely \( u'_t(\omega,x)=f_\phi(x,u_t(\omega,x)), \text{ for all } t\in[0,\tau)\), where
\begin{multline}
	\label{f_LLG}
f_\phi(x,y):=
\phi(x)\begin{pmatrix}
0 & - y^3 & y^2\\
y^3 & 0 & -y^1\\
-y^2 & y^1 & 0
\end{pmatrix}
\\
+
\epsilon\phi(x)\begin{pmatrix}
	(y^2)^2+(y^3)^2 &  y^1y^2 & y^1y^3\\
	y^2y^1 & (y^1)^2 +(y^3)^2 & y^2y^3\\
	y^3y^1 & y^3y^2 & (y^1)^2+(y^2)^2
\end{pmatrix},
\end{multline}
for \( (x,y)\in \dom\times \R^3\),
and a set \( \Omega_0\in \mathcal F \) of full measure such that for each \( \omega\in \Omega_0 \), the pair \( (u(\omega),u'(\omega)) \) solves \eqref{LLG}, understood here as the following RPDE with random coefficients
\[
	du = (\Delta u + u\times\Delta u+ u|\nabla u|^2)dt + u' \cdot d\mathbf{\tilde W} \,,
\]
in the sense of Definition \ref{def:var_sol}. This solution is unique amongst the previous class.
\end{theorem}

\begin{proof}
	It is well-known that the Stratonovitch lift \( \mathbf{\tilde W}(\omega)\) as defined above is in \( \mathscr C^{\gamma_0}(\R^3)\) for every \( \gamma_0<\frac12\) and each \( \omega\) in a set  \(\Omega^{\gamma_0} \subset \Omega\) with full-probability.
If we fix such \( \omega\in \Omega^{\gamma_0}\), the equation \eqref{LLG} thus falls into the framework of Theorem \ref{thm:system_B} and the conclusion follows. (Observe that since the map \( v\mapsto L(v)\) is affine linear, the condition \eqref{coef_regularity_critical} is void, hence the case \( \alpha=1\) is allowed.)
\end{proof}

We point out the following notion of solution for~\eqref{LLG} which was investigated in~\cite{gussetti2021} for linear noise (in which case it is possible to obtain a priori estimates for \eqref{LLG}, in suitable Sobolev spaces). This will turn out to be equivalent to the concept introduced in Definition~\ref{def:var_sol} as justified in Section~\ref{rds}.
\begin{remark}[Alternative notion of solution in 1D]
\label{rem:LLG}
Fix \( \epsilon=0\), \( n=1 \), \( \gamma\in (\frac13,\frac12)\), let \( \dom=\mathbf T \)
and take \( k\in \N, \phi\in H^{k+1}\).
In \cite{gussetti2021}, a solution of \eqref{LLG} starting at \( u_0 \in H^k\)
is defined as a continuous stochastic process $u\colon \Omega\times [0,T]\to H^{k-1}(\T;\R^3)$ satisfying \( u_0=u^0 \),
			\(\mathbb P(|u_t(x)|=1,\; \forall(t,x)\in[0,T]\times\T)=1=\mathbb P(u\in L^\infty (0,T;H^k)\cap L^2(0,T;H^{k+1}))\)
 and there exists a random variable $u^{\natural}\in L^0(\Omega; \cC_2^{3\gamma}(0,T;H^{k-1}))$ such that
			\begin{equation}
			\label{rLLG_def}
			\delta u_{s,t}-\int_s^t(\partial_{xx}u_r +u_r|\partial_x u_r|^2 +u_r\times\partial_x^2 u_r)  d r
			=G_{s,t}u_s + \mathbb G_{s,t}u_s + u^\natural_{s,t}\,,
			\end{equation}
			$\mathbb P$-a.s.\ as an equality in $L^2(\T ;\R^3)$ for every $s\leq t\in[0,T] $, and where \( (G_{t,s},\mathbb G_{t,s})\) is the pair of linear maps given for each \( v\in \R^3\) by
			\begin{equation}
				\label{G}
			G_{t,s}(x)v:= 
			\phi(x)\begin{pmatrix}
				0 & \dd W^3_{t,s} & -\dd W^2_{t,s}\\
				-\dd W^3_{t,s} & 0 & \dd W^1_{t,s}\\
				\dd W^2_{t,s} & -\dd W^1_{t,s} & 0
			\end{pmatrix}
		 \begin{pmatrix}
		 v^1\\
		 v^2\\
		 v^3
		 \end{pmatrix}
		 = f(x,v)\dd W_{t,s}
			\end{equation}
			while
			\begin{equation}
			\label{GG}
		\begin{aligned}
			\mathbb G_{t,s}(x)v&:= 
			\phi(x)^2\begin{pmatrix}
				-\mathbb{\tilde W}^{3,3}_{t,s} - \mathbb{\tilde W}^{2,2}_{t,s} &  \mathbb{\tilde W}^{1,2}_{t,s} & -\mathbb{\tilde W}^{1,3}_{t,s}\\
				-\mathbb{\tilde W}^{2,1}_{t,s} & -\mathbb{\tilde W}^{3,3}_{t,s} - \mathbb{\tilde W}^{1,1}_{t,s} & -\mathbb{\tilde W}^{2,3}_{t,s}\\
				\mathbb{\tilde W}^{3,1}_{t,s} & -\mathbb{\tilde W}^{3,2}_{t,s} & -\mathbb{\tilde W}^{2,2}_{t,s} - \mathbb{\tilde W}^{1,1}_{t,s}
			\end{pmatrix}
		\begin{pmatrix}
			v^1\\
			v^2\\
			v^3
			\end{pmatrix}
		\\&
		= D_2f(x,v) \otimes f(x,v)[\mathbb{\widetilde W}_{t,s}]\,.
		\end{aligned}
			\end{equation}
\end{remark}

\subsection{Random dynamical systems}\label{rds}
The main goal of this subsection is to establish the existence of flows for~\eqref{problem}. To this aim we assume that the solution constructed in Theorem~\ref{thm:main} is global-in-time. We emphasize that global well-posedness results for rough partial differential equations are difficult to obtain in general. See however \cite{HN20,HN21} for a treatment of the semilinear case under some linear growth conditions. Global well-posedness results for rough differential equations with a dissipative drift term have also been obtained in~\cite{Weberglobal}.\\
In order to construct a random dynamical system corresponding to~\eqref{problem}, we firstly introduce some concepts from the theory of random dynamical systems~\cite{Arnold}. The following definition describes a model of the driving noise.

\begin{definition}\label{mds} 
	Let $(\Omega,\mathcal{F},\mathbb{P})$ stand for a probability space and 
	$\theta:\mathbb{R}\times\Omega\rightarrow\Omega$ be a family of 
	$\mathbb{P}$-preserving transformations (i.e.,~$\theta_{t}\mathbb{P}=
	\mathbb{P}$ for $t\in\mathbb{R}$) having the following properties:
	\begin{description}
		\item[(i)] The mapping $(t,\omega)\mapsto\theta_{t}\omega$ is 
		$(\mathcal{B}(\mathbb{R})\otimes\mathcal{F},\mathcal{F})$-measurable, where 
		$\mathcal{B}(\cdot)$ denotes the Borel sigma-algebra;
		\item[(ii)] $\theta_{0}=\textnormal{Id}_{\Omega}$;
		\item[(iii)] $\theta_{t+s}=\theta_{t}\circ\theta_{s}$ for all 
		$t,s,\in\mathbb{R}$.
	\end{description}
	Then the quadrupel $(\Omega,\mathcal{F},\mathbb{P},(\theta_{t})_{t\in\mathbb{R}})$ 
	is called a metric dynamical system.
\end{definition}

\begin{definition}
	\label{rdsy} 
	A continuous random dynamical system on a separable Banach space $\cX$ over a metric dynamical 
	system $(\Omega,\mathcal{F},\mathbb{P},(\theta_{t})_{t\in\mathbb{R}})$ 
	is a mapping $$\varphi:[0,\infty)\times\Omega\times \cX\to \cX,
	\mbox{  } (t,\omega,x)\mapsto \varphi(t,\omega,x), $$
	which is $(\mathcal{B}([0,\infty))\otimes\mathcal{F}\otimes
	\mathcal{B}(\cX),\mathcal{B}(\cX))$-measurable and satisfies:
	\begin{description}
		\item[(i)] $\varphi(0,\omega,\cdot{})=\textnormal{Id}_{\cX}$ 
		for all $\omega\in\Omega$;
		\item[(ii)]$ \varphi(t+\tau,\omega,x)=
		\varphi(t,\theta_{\tau}\omega,\varphi(\tau,\omega,x)), 
		\mbox{ for all } x\in \cX, ~t,\tau\in[0,\infty),~\omega\in\Omega;$
		\item[(iii)] $\varphi(t,\omega,\cdot{}):\cX\to \cX$ is 
		continuous for all $t\in[0,\infty)$ and all $\omega\in\Omega$.
	\end{description}
\end{definition}

The second property in Definition~\ref{rdsy} is referred to as the 
cocycle property. The generation of a random dynamical system from an It\^{o}-type stochastic partial differential equation (SPDE) has been a long-standing open problem, since Kolmogorov's theorem breaks down for random fields parametrized by infinite-dimensional Banach spaces.~As a consequence it is not obvious how to obtain a random dynamical system from an SPDE, since its solution is defined almost surely, which contradicts the cocycle property. Particularly, this means that there are exceptional sets which depend on the initial condition and it is not clear how to define a random dynamical system if more than countably many exceptional sets occur. This issue does not occur in a pathwise approach. Under suitable assumptions on the coefficients, rough path driven equations generate random
dynamical systems provided that the driving rough path forms a rough path cocycle, as established in~\cite{BRiedelScheutzow}.\\
Let $(\Omega,\mathcal{F},\mathbb{P},(\theta_{t})_{t\in\mathbb{R}})$ be a metric dynamical system as in Definition~\ref{mds}. We say that 
\begin{align*}
\mathbf{X}=(X,\XX):\Omega\to C^{\gamma_{0}}_{\text{loc}}([0,\infty);\mathbb{R}^d) \times C^{2\gamma_0}_{\text{loc}}([0,\infty);\mathbb{R}^{d\times d}) 
\end{align*}
is a continuous ($\gamma_0$-H\"older) rough path cocycle if $\mathbf{X}|_{[0,T]}$ is a continuous $\gamma_0$-H\"older rough path for every $T>0$ and for every $\omega\in\Omega$ and the following cocycle property holds true for every $s,t\in[0,\infty)$ and $\omega\in\Omega$
\begin{align*}
&X_{s,s+t}(\omega)= X_t(\theta_s\omega)\\
&\XX_{s,s+t}(\omega)=\XX_{0,t}(\theta_s\omega).
\end{align*}
According to~\cite[Section 2]{BRiedelScheutzow} rough path lifts of various stochastic processes define cocycles.
These include Gaussian processes with stationary increments under certain
assumption on the covariance function~\cite[Chapter 10]{friz2014course} and particulary apply to the fractional Brownian motion with Hurst index $H>\frac{1}{4}$. Recall that here we fixed the $\gamma_0$-H\"older regularity of the rough path $\gamma_0\in(\frac{1}{3},\frac{1}{2})$, consequently the results obtained apply to fractional Brownian motion for $H\in(\frac{1}{3},\frac{1}{2}]$.

\begin{theorem}\label{thm:rds}
	Let $u_0\in\cB_\alpha$ and assume that the solution constructed in Theorem~\ref{thm:main} is global-in-time. Then its solution operator generates a continuous random dynamical system on $\cB_\alpha$. 
\end{theorem}

A particular application we have in mind is the following.

\begin{corollary}[Random dynamical system for the Landau-Lifshitz-Gilbert equation with linear noise in 1D]
	Let \( n=1\), take \( \dom=\mathbb T\) and fix some initial datum \( u_0\in H^2(\mathbf T;\R^3)\) such that \( |u_0(x)|=1,\forall x\).
	Suppose moreover that \( \epsilon=0\) in \eqref{f_LLG}.
	 The rough solution of \eqref{LLG} constructed in section \ref{subsec:LLG} is also a pathwise mild solution in \( \cD^{2\gamma}_{W(\omega),1}\), for any \( \gamma\in (\frac13,\frac12)\) and \( \omega\in \Omega^{\gamma}\). Moreover, it is global in time and generates a continuous random dynamical system on $\cX=H^2$ in the sense of Definition \ref{rdsy}.
\end{corollary}
\begin{proof}
	Let \( T>0\) be arbitrary and take \( \gamma\in (\frac13,\frac12)\). We fix the sample parameter \( \omega\in \Omega^\gamma\) all throughout and omit it in the notations.
	In this setting, the existence of a solution in \(L^\infty(0,T;H^2)\cap L^2(0,T;H^3)\), in the sense of Remark \ref{rem:LLG}, has been established in \cite{gussetti2021}.
	By Theorem \ref{thm:rds} we only need to show that this solution, denoted by \( u\) in the sequel, is also a mild solution in $\cB_1=H^2$.
	
	As was observed in \cite{gussetti2021}, \( \dd u \) belongs to \(\cC^{\gamma}_2(0,T;\cB_{\alpha-\gamma})\), and moreover
	\begin{equation}\label{fwd_rest}
		|\cR_{t,s}|_{H^1} \lesssim |t-s|^{2\gamma},
	\end{equation}
	where \( \cR_{t,s}:=\dd u_{t,s}- G_{t,s}u_s \). In particular, if we denote by \( L:=L(u_\cdot)\) then by \eqref{lip_constant} we have that
	\( S:=\exp(\int L_rdr)\) is a well-defined evolution family. Moreover, it holds \( (u,F(u))\in \cD^{2\gamma}_{W,\alpha}\).
	Next, using that \( x-S_{t,s}x=-\int_s^tS_{t,r}L_rxdr \) for each \( x\in H^2\) (Bochner sense, in \( L^2\)), we find thanks to \eqref{rLLG_def} that 
	\[
	\begin{aligned}
	u_t - S_{t,s}u_s 
	&= (I-S_{t,s})u_t + S_{t,s}\dd u_{t,s}
	\\& 
	=-\int_s^t S_{t,r}L_ru_tdr + S_{t,s}[\int _s^t (L_ru_r+N(u_r))dr + (G_{t,s}+ \mathbb G_{t,s})u_s + u^{\natural}_{t,s}]
	\\&
	=\int _s^t S_{t,r}N(u_r)dr + S_{t,s}[G_{t,s} + \mathbb G_{t,s}-(G_{t,s})^2]u_t + \mathscr{\bar R}_{t,s}
	\end{aligned}
	\]
where
	\begin{align*}
	\mathscr{\bar R}_{t,s}
	= \int_s^t S_{t,r}(S_{r,s}-I)L_ru_rdr -\int _s^tS_{t,r}L_r \dd u_{t,r}dr 
	+\int _s ^tS_{t,r}(S_{r,s}-I)N(u_r)dr
	\\
	- S_{t,s} G_{t,s}(\dd u_{t,s}-G_{t,s}u_t)  -\mathbb{G}_{t,s} \dd u_{t,s} + S_{t,s}u^\natural_{t,s}.
	\end{align*}
	Is is plain to check that every term above belongs to \( \cC_2^{z}([0,T];\cZ)\) for some \( z>1\) and \( \cZ\hookleftarrow L^2\). For instance, regarding that \( r\mapsto L_ru_r\in L^2(0,T;H^1)\), we have for the first term
	\[
		 |\int_s^t S_{t,r}(S_{r,s}-I)L_ru_rdr|_0
		 \lesssim \int _s^t (r-s)^{1/2}|Lu_r|_{H^1}
		 \lesssim (t-s)(\int _s^t|u_r|_{H^3}^2dr)^{1/2}
	\]
	and we obtain a similar bound for the third term.
	For the second term, we use the H\"older regularity of the solution to get
	\[
	|\int _s^tS_{t,r}L_r \dd u_{t,r}dr |_{H^{-1}}
	\lesssim 
	\int _s ^t (t-r)^{\gamma}dr [\dd u]_{\gamma;H^1}\lesssim (t-s)^{1+\gamma}.
	\]
	Next, using \eqref{fwd_rest}
	 we observe that
	\[
	|S_{t,s} G_{t,s}(\dd u_{t,s}-G_{t,s}u_t)|_0
	\le 
	|S_{t,s}G_{t,s}\cR_{t,s}|_0 + |S_{t,s}G_{t,s}^2\dd u_{t,s}|_0
	\lesssim 
	|t-s|^{3\gamma } 
	\]
	and similarly for the fifth and last terms.\\
	Therefore, using the identities \eqref{G}-\eqref{GG} and the uniqueness statement of Remark \ref{rem:uniqueness} (with the choice \( \cZ=H^{-1}\)), we find that
	\[
	\begin{aligned}
		u_t - S_{t,s}u_s &
		=\int _s^t S_{t,r}N(u_r)dr + S_{t,s}(\zeta_t\cdot \dd W_{t,s}+ \zeta'_{t}:\mathbb{\tilde W}_{t,s}) + \mathscr R^{S,\zeta}_{t,s}
	\end{aligned}
	\]
where \( (\zeta,\zeta')=(F(u), DF(u)\circ F(u))\)
and \(\mathscr R^{S,\zeta}_{t,s}\) is the integral remainder as defined in Theorem \ref{thm:integral}. This shows in particular that \( u \) is a mild solution, and it is necessarily the same as the one supplied by Theorem \ref{thm:system_B} when \( \alpha=1\) (in particular \( \tau(\omega)= T\)). This concludes the proof.
\end{proof}

\appendix
\section{Some technical results}

\subsection{Perturbation results for evolution families }\label{app:perturbation}

We provide some estimates of the difference of two evolution families, which are crucial for the perturbation results of the sewing map established in Section~\ref{ssec:perturb}.
Recall that the Euler Beta function is the map \( \mathrm B(x,y)=\int _0^1\theta^{x-1}(1-\theta)^{y-1} d\theta\), for \( x,y>0. \)
\begin{lemma}
	\label{lem:K}
Consider two evolution families $S^1,S^2$ on a Banach space \( \cX=\cX_0 \), whose respective infinitesimal generators $L^1_\cdot,L^2_\cdot$ are such that $D(L^1_t)=D(L^2_t)=\cX_{1}\hookrightarrow\cX$ for all \( t\in [0,T] \).
Suppose moreover that \( (\cX_\beta)_{\beta\in [0,1]} \) is a scale of Banach spaces subject to Assumption \ref{ass:intermediate}.
\begin{itemize}
\item
Let $0\leq\beta\le \beta'\leq 1$ be such that $|\beta-\beta'|<1.$ 
Then
\[
S^1-S^2 \in \mathfrak K(\beta,\beta').
\]
If $\beta>0$ and $\beta'<1$, we have the estimate
\[
\|S^1-S^2\|_{(\beta,\beta')}
\leq
\mathrm B(1-\beta',\beta)\|S^1\|_{(0,\beta')}\|S^2\|_{(\beta,1)}\sup_{t\in [0,T]}|L^1-L^2|_{\cX_1\to \cX}.
\]
\item
If either \( \beta =1\) or \( \beta'=0\) and provided \( L^1-L^2\in C^\varrho([0,T];\cL(\cX_1,\cX))\) with \( \varrho\in (0,\beta\wedge (1-\beta')) \) we have instead, for \( T=T_{K,\varrho}>0 \) chosen small enough
\begin{equation}\label{perturb_critical}
\|S^1-S^2\|_{(\beta,\beta')}
\lesssim _{K,\tilde K} \sup_{t\in [0,T]}|L^1_t-L^2_t|_{\cX_1\to \cX}+T^{\varrho}|L^1-L^2|_{\varrho;\cX_{1}\to\cX}.
\end{equation}
\end{itemize}
\end{lemma}
\begin{proof}
 Using~\cite[Lemma 5.1.4]{amann95quasilinear} we have that
\begin{align}\label{representation1}
S^1_{t,s} - S^2_{t,s}=\int\limits_s^t S^1_{t,\tau} [L^1_\tau -L^2_\tau] S^2_{\tau, s}d\tau.
\end{align}
Then we estimate the norm $|\cdot|_{\cX_\beta\to\cX_{\beta'}}$ as follows assuming that $\beta'\neq 1$ and $\beta\neq 0$.
\begin{align*}
|S^1_{t,s} - S^2_{t,s}|_{\cX_\beta\to\cX_{\beta'}} &\leq \int\limits_s^t |S^1_{t,\tau} (L^1_\tau -L^2_\tau) S^2_{\tau, s}|_{\cX_{\beta}\to\cX_{\beta'}}d\tau \\
& \le \int\limits_s^t |S^1_{t,\tau}|_{\cX\to\cX_\beta'} |L^1_\tau -L^2_\tau|_{\cX_1\to\cX} |S^2_{\tau,s}|_{\cX_\beta\to \cX_1}d\tau\\
& \le \|S^1\|_{(0,\beta')}\|S^2\|_{(\beta,1)} \int\limits_s^t (t-\tau)^{-\beta'} (\tau -s)^{\beta-1} d\tau \sup_{r\in [s,t]}|L^1_r-L^2_r|_{\cX_1\to\cX}\\
& \le\mathrm B(1-\beta',\beta) \|S^1\|_{(0,\beta')}\|S^2\|_{(\beta,1)}\sup_{r\in [s,t]}|L^1_r-L^2_r|_{\cX_1\to\cX} (t-s)^{\beta-\beta'}.
\end{align*}

We now adress the proof of \eqref{perturb_critical} when $\beta>0$ and $\beta'=1$, the other case being similar (hence omitted).
For that purpose we need to use the H\"older continuity of the evolution families, i.e.\ the fact that $L^1-L^2\in C^{\rho} (0,T;\cL(\cX_{1},\cX)) $. This step is technically more involved and can be justified from the integral formulas \eqref{first_int}-\eqref{second_int}.
In the notations of Remark \ref{rem:evolution}, we have for \( j=1,2 \)
\[
S^j= b^j - h^j\star S^j,
\]
where \( h^j_{t,s}=\dd L^j_{t,s}.b^j_{t,s} \) and we recall that \( b^j_{t,s}=e^{(t-s)L_t} \).
It is plain to see that
\begin{equation*}
\|h^1-h^2\|_{(1,1;1-\rho)}\lesssim |L^1-L^2|_{\varrho; \cX_1\to\cX }
\end{equation*}
where we introduce the notation
\(\|h\|_{(\alpha,\beta;\kappa)}=\sup_{s<t\in [0,T]}(t-s)^{\kappa}|h^1_{t,s}-h^2_{t,s}|_{\alpha\to \beta}\).
A basic change of variable then yields
\begin{equation}
\label{est:diff_h}
\begin{aligned}
\|(h^1-h^2)\star S\|_{(\beta,1)}
&\le \mathrm B(\varrho,\beta-\varrho)\|h^1-h^2\|_{(1,1;1-\varrho)}\|S\|_{(\beta,1;1-\beta+\varrho)}
\\&
\lesssim_{K,\varrho,\beta} T^\varrho|L^1-L^2|_{\varrho; \cX_1\to\cX }\,.
\end{aligned}
\end{equation}
(we refer to \cite[Lemma 5.3]{amann1986quasilinear}) for details).

Consequently
\begin{equation}
\label{pre_neumann}
(I - h^2\star )(S^1-S^2) = b^1-b^2 + (h^1-h^2)\star S^1\,.
\end{equation}
Observing that the linear operator \( h^2\star \in \cL \big(\mathfrak K(\beta,1) \big)\) has zero spectral radius, we can revert \eqref{pre_neumann} and obtain in particular for \( 0<T\ll1 \) (using \eqref{est:diff_h})
\[
\begin{aligned}
\|S^1-S^2\|_{(\beta,1)} 
&\le \frac{1}{1-|h^2\star|_{\mathfrak K(\beta,1)\to\mathfrak K(\beta,1)}} [\|b^1-b^2\|_{(\beta,1)} + \|(h^1-h^2)\star S^1\|_{(\beta,1)}]
\\&
\lesssim 
\frac{1}{1-cT^\varrho}[|L^1-L^2|_{0;\cX_1\to\cX } + T^\varrho |L^1-L^2|_{\varrho;\cX_1\to\cX_0}]
\end{aligned}
\]
which gives the desired conclusion.
\end{proof}

\subsection{Reduced versus plain increments}\label{app:plain}

Recall that, for a Banach space \( \cX\), a multiplicative \( S=S_{t,s}\in \mathscr L(\cX)\) and a path \( y\colon [0,T]\to\cX\),
the \textit{reduced increment} of \( y\) (w.r.t.\ \( S\)) is the 2-parameter quantity \[
\dd^{S}y_{t,s}=y_t-S_{t,s}y_s,\quad \quad 0\le s\le t\le T.
\]
If \( (y,y')\) is a controlled path, we may also introduce the reduced remainder
\[
R^{S,y}_{t,s}:= \dd^Sy_{t,s}-S_{t,s}y'_t\cdot\dd X_{t,s}\,.
\]
The main idea of the next result is that the H\"older norm of certain controlled rough paths (satisfying an additional regularity property) can be estimated in terms of the corresponding ``reduced'' quantity.
\begin{lemma}
	\label{lem:delta_S}
	Let \( S=\exp(\int L_rdr)\) be as in Assumption \ref{ass:L_t}. 
	Fix \( X\in \mathscr C^{\gamma}(\R^d)\) with \( \gamma\in (\frac13,\frac13)\) and \( \alpha\in (0,1)\).
Furthermore let \( (y,y')\in \cD_{X,\alpha,\sigma}^{\gamma,\gamma'}(0,T;\varepsilon)\) where \( \varepsilon=\gamma-\sigma\) and additionally suppose that \( y\in C^{0,\varepsilon}(\cB_{\alpha+\varepsilon})\).
	\begin{itemize}
	\item [1)] We have the norm equivalence
	\begin{equation}
		\label{delta_S_equiv}
	|y|_{0,\alpha+\varepsilon}^{(\varepsilon)}+[\dd^{S}y]^{(\varepsilon)}_{\gamma,\alpha-\sigma}
	\simeq
	|y|_{0,\alpha+\varepsilon}^{(\varepsilon)}+[\dd y]^{(\varepsilon)}_{\gamma,\alpha-\sigma}
	\end{equation}
and similarly for the remainders
	\begin{equation}
	\label{R_S_equiv}
	|y|_{0,\alpha+\varepsilon}^{(\varepsilon)}+ |y'|_{0,\alpha-\sigma}^{(\varepsilon)}+[R^{S,y}]^{(2\varepsilon)}_{\gamma+\gamma',\alpha-\sigma-\gamma'}
	\simeq	|y|^{(\varepsilon)}_{0,\alpha+\varepsilon}+|y'|_{0,\alpha-\sigma}^{(\varepsilon)}+[R^y]^{(2\varepsilon)}_{\gamma,\alpha-\sigma-\gamma'},
\end{equation}
	for constants depending only on \( K,\tilde K\) and \(\rho_{\gamma}(\X).\)
	\item [2)]
	Fix another evolution family \( \bar S=\exp(\int \bar L_rdr)\) and \((\bar y,\bar y')\in \cD_{X,\alpha,\sigma}^{\gamma,\gamma'}(0,T;\varepsilon)\) such that \( \bar y\in C^{0,\varepsilon}(\cB_{\alpha+\varepsilon}).\)
	Let \( N:=|y|^{(\varepsilon)}_{0,\alpha+\varepsilon}\vee |\bar y|^{(\varepsilon)}_{0,\alpha+\varepsilon}\). We have the norm equivalences
	\begin{multline}\label{delta_S_equiv_diff}
|y-\bar y|^{(\varepsilon)}_{0,\alpha+\varepsilon}+[\dd^{S}y-\dd^{\bar S}\bar y]^{(\varepsilon)}_{\gamma,\alpha-\sigma} +\sup_{t\in[0,T]}|L_t-\bar L_{t}|_{1\to0}
	\\
	\simeq
	|y-\bar y|^{(\varepsilon)}_{0,\alpha+\varepsilon}+[\dd y-\dd \bar y]^{(\varepsilon)}_{\gamma,\alpha-\sigma}+\sup_{t\in[0,T]}|L_t-\bar L_{t}|_{1\to0}
	\end{multline}
and
\begin{multline}\label{R_S_equiv_diff}
	|y-\bar y|^{(\varepsilon)}_{0,\alpha+\varepsilon}+|y'-\bar y'|^{(\varepsilon)}_{0,\alpha-\sigma}+[R^{S,y}-R^{\bar S,y}]^{(2\varepsilon)}_{\gamma+\gamma',\alpha-\sigma-\gamma'} +\sup_{t\in[0,T]}|L_t-\bar L_{t}|_{1\to0}
	\\
	\simeq
	|y-\bar y|^{(\varepsilon)}_{0,\alpha+\varepsilon}+|y'-\bar y'|^{(\varepsilon)}_{0,\alpha-\sigma}+[R^{y}-R^{\bar y}]^{(2\varepsilon)}_{\gamma+\gamma',\alpha-\sigma-\gamma'}+\sup_{t\in[0,T]}|L_t-\bar L_{t}|_{1\to0},
\end{multline}
for implicit constants depending on \( K,\tilde K,\rho_{\gamma}(\X)\) and \( N.\)
\end{itemize}
\end{lemma}

\begin{proof}
	The lower and upper bounds follow by similar arguments. We begin with the first statement and observe due to~\eqref{id_delta_S} for $s>0$ that
	\begin{align*}
		s^\varepsilon \dd^{S} y_{t,s} =s^\varepsilon(\dd y_{t,s} + (S_{t,s}-\id)y_s ),
	\end{align*}
	which proves the lower bound in 1) due to the regularity properties of $y$. Furthermore, using this bound on $[\dd^S y]^{(\varepsilon)}_{\gamma,\beta-\gamma}$, we analogously derive
	\begin{align*}
	 s^\varepsilon|\dd y_{t,s}|_{\beta-\gamma} &\le s^{\varepsilon}(|\dd^Sy_{t,s}|_{\beta-\gamma} + |(S_{t,s}-\id)y_s|_{\beta-\gamma} )
	 \le  |t-s|^{\gamma}[\dd y]^{(\varepsilon)}_{\gamma,\beta-\gamma}
	 +|t-s|^\gamma \tilde K |y|^{(\varepsilon)}_{0,\beta}\,.
	\end{align*}
Setting $\beta=\alpha+\varepsilon$ this proves~\eqref{delta_S_equiv}. 
For the second estimate \eqref{R_S_equiv}, we rely on similar arguments and the relation
\[
 R^{S,y}_{t,s}=R^{y}_{t,s} + (I-S_{t,s})[y_s + y'_t\cdot\dd X_{t,s}].
\]

For the upper bound in \eqref{delta_S_equiv_diff}, we note that
	\[
	\dd^{S} y_{t,s} - \dd^{\bar S}\bar y_{t,s}
	= \dd (y-\bar y)_{t,s} + (S_{t,s}-I)(\bar y_s-y_s)+ (\bar S-S)_{t,s}\bar y_s \,.
	\]
	Therefore
	\[
	s^{\varepsilon}|\dd^{S} y_{t,s} - \dd^{\bar S}\bar y_{t,s}|_{\beta-\gamma}
	\le |t-s|^{\gamma}([\dd y-\dd \bar y]^{(\varepsilon)}_{\gamma,\beta-\gamma} 
	+ \tilde K|y-\bar y|^{(\varepsilon)}_{0,\beta} + \|S-\bar S\|_{(\beta,\beta-\gamma)}N )
	\]
	and Lemma \ref{lem:K} yields our conclusion. The lower bound and the second estimate \eqref{R_S_equiv_diff} are similar.
\end{proof}


\bibliographystyle{alpha}

\end{document}